\documentclass[11pt, a4paper]{article}

\usepackage[a4paper, total={6.5in, 9.7in}]{geometry}
\usepackage{amsmath, amssymb, amsthm} 
\usepackage{graphicx,subfig}
\usepackage{caption}
\usepackage{color}

\newtheorem{alg}{Algorithm}
\newtheorem{thm}{Theorem}
\newtheorem{defn}{Definition}
\newtheorem{lem}{Lemma}
\newtheorem{prop}{Proposition}
\newtheorem{cor}{Corollary}
\newtheorem{remark}{Remark}
\newcommand{\unit}{\texttt{1}\!\!\texttt{l}}

\newcommand{\M}{\mathcal{M}}

\makeatletter
\let\gplgaddtomacro\g@addto@macro
\gdef\gplbacktext{}%
\gdef\gplfronttext{}%
\makeatother

\numberwithin{equation}{section}

\author{Charles M. Elliott\footnotemark[1]~ and Hans Fritz\footnotemark[1]}

\title{On algorithms with good mesh properties for problems with moving boundaries based on the Harmonic Map Heat Flow and the DeTurck trick}

\date{}

\begin{document}
\maketitle
\renewcommand{\thefootnote}{\fnsymbol{footnote}}
\footnotetext[1]{Mathematics Institute, Zeeman Building, University of Warwick, Coventry. CV4 7AL. UK\\ 
C.M.Elliott@warwick.ac.uk, hans.fritz@ur.de}
\abstract{
In this paper, we present a general approach to obtain numerical schemes with good mesh properties
for problems with moving boundaries, that is for evolving submanifolds with boundaries. 
This includes moving domains and surfaces with boundaries.
Our approach is based on a variant of the so-called the DeTurck trick.
By reparametrizing the evolution of the submanifold via solutions 
to the harmonic map heat flow of manifolds with boundary, we obtain a new velocity field for the motion 
of the submanifold. Moving the vertices of the computational mesh according to this velocity field
automatically leads to computational meshes of high quality both for the submanifold 
and its boundary. Using the ALE-method in \cite{ES12}, this idea can be easily built into algorithms
for the computation of physical problems with moving boundaries. 
}\\

\\

\\

\section{Introduction}
\subsection{Background}

Developing efficient methods for solving boundary values problems in complex domains 
is one of the main topics in Numerical Analysis. Often such problems arise from physical applications
for which it is quite natural that the boundary changes in time.  
The time-development of the boundary might be given explicitly.
In other cases, only the evolution law for the motion of the boundary might be known.
In any of these cases, the computational task is to solve a PDE
on a moving domain bounded by some time-dependent boundary. 
Similar problems might appear in curved spaces.
Then the task is to solve a PDE on a moving surface with boundary.

In the following, we will consider the most general case, that is an evolving submanifold
of arbitrary dimension
with boundary embedded into some higher-dimensional 
Euclidean space. While from an analytical point of view 
such problems are already more demanding than problems on
stationary submanifolds, their numerical treatment is even more involved.

Due to the nature of a computer, an approximation to a solution of a PDE can only be described
by a finite set of numbers. This entails that numerical solutions are not
defined on the actual submanifold but on a computational mesh, which we here assume
to be the union of some simplices. The approximation to the solution of the original problem
is then determined by its values on the vertices of the simplicial mesh. The fact that the
computational domain is, in general, not equal to the original submanifold --
in particular, they will usually have different boundaries -- certainly affects
the quality of the approximation of the solution and is sometimes labelled as a variational crime.

The most natural way to think of a computational mesh for an evolving submanifold
would certainly be a mesh changing in time. Alternatively, it might be possible to formulate
some sophisticated extension of the original problem 
to a time-independent ambient domain of the submanifold, which then
would enable to compute the solution on a stationary mesh. However, since 
such approaches certainly have their own difficulties, we will here fully neglect this possibility. 
This means that in the following we will only consider moving meshes. 
Of course, a moving mesh here only means a finite number of meshes that 
approximate the evolving submanifold with boundary at different discrete time levels.
The problem, which then arises, is to find a method to construct such a family of meshes.
A rather ad hoc approach would be to directly construct a mesh for the submanifold at each
discrete time level. However, this would require that the submanifold is explicitly known
at each time level. Furthermore, one would have to define a rule how to use the solution
from the previous time step defined on the previous mesh 
to compute the solution of the next
time step on the new mesh. If the new mesh is given by a deformation of the previous mesh
such a rule can be easily implemented. We therefore arrive at the task to deform a mesh
efficiently in such a way that the variational crime remains small.

If the motion of the submanifold is given by a velocity field which is either 
explicitly known or defined implicitly by the solution of
some other problem, the easiest way to deform a mesh would be to move the mesh vertices
according to this velocity field. Unfortunately, this would, in general, lead to mesh degenerations,
that is to the formation of meshes with very sharp simplices. However, it is well known that 
on such meshes the approximation of the solution to a PDE is very bad.
The idea of this paper is therefore to change the original velocity field in such a way
that firstly the shape of the submanifold is not changed and secondly mesh degenerations
are prevented. Certainly, the optimum would be to have a velocity field that can even be used to
improve the quality of the mesh. 
Indeed, it turns out that there is a quite general way
based on a PDE approach for this problem.
As we will present here, this approach is practical and 
can be easily built into numerical schemes for solving
PDEs on evolving submanifolds.
In the following, a triangulation is called a good mesh if the quotient of the diameter of a simplex and of the radius of the
largest ball contained in it is reasonably small for all simplices of the triangulation, see also
definition (\ref{definition_sigma_max}) below.

\subsection{Our approach}

Our approach is based on a variant of an idea that was originally introduced 
to prove short-time existence and uniqueness for some geometric PDEs such as the Ricci flow 
(see in \cite{CLN06, DeT83, Ham95}) and the
mean curvature flow (see in \cite{Ba10, Ham89}) on closed manifolds, 
that is on compact manifolds without boundary.
This idea, which is nowadays called the DeTurck trick, uses solutions to the harmonic map
heat flow on manifolds without boundary (see \cite{ES64}) in order to reparametrize the 
evolution of the curvature flows. This then leads to new PDEs for the reparametrized flows.
The advantage is that in contrast to the original PDEs, the reparametrized PDEs are strongly
parabolic. 
We recently showed in \cite{EF15} that this trick is also quite useful in Numerics.
A well-established method to compute the mean curvature flow (see \cite{Dz91}), 
that is the deformation of an embedded hypersurface
into the negative direction of its mean curvature vector, is based on
evolving  surface finite elements; see \cite{DE13}. Although this method is very appealing,
since it gives direct access to the evolving surface and is very efficient at the same time,
a big disadvantage was until recently that it often leads to mesh degenerations. We have tackled
this problem in \cite{EF15} 
by using a reformulation of the mean curvature flow based on the DeTurck trick.
As we have demonstrated in numerical experiments, this approach indeed leads to schemes
which prevent the mesh from degenerating.

In this paper, we aim to extend this idea to evolving submanifolds with boundary.
In order to keep our approach as general as possible, we here do not assume
that the motion of the submanifold is determined by a special PDE.
Instead, we just assume that the submanifold moves according to some velocity field,
which can be given either explicitly or implicitly. 
Moreover, we assume that the submanifold is given as the image of a time-dependent embedding
of some reference manifold with boundary.
For the DeTurck trick we will apply
the harmonic map heat flow on manifolds with boundary.
This flow has been used in \cite{Ham75} to prove existence of harmonic maps between Riemannian
manifolds with boundary. 
We will consider this flow on the reference manifold of the evolving submanifold.
In contrast to the setting in \cite{EF15}, where only closed
manifolds were considered, we now have to choose boundary conditions for the harmonic
map heat flow. In fact, this is already a delicate issue. For example, for pure Dirichlet boundary
conditions the harmonic map heat flow would be fixed on the boundary.
However, a reparametrization of the evolving submanifold by the harmonic map heat flow
would then not change the velocity field on the boundary of the evolving submanifold.
Hence, the computational mesh would then, in general, behave badly close to its boundary.
On the other hand, Neumann boundary conditions cannot be applied, since to generate a reparametrization
we need the harmonic map heat flow to map the boundary onto itself. 
It thus turns out that we have to apply some mixed boundary conditions for our purpose.
In \cite{Ham75} it is discussed why existence of the harmonic map heat flow subject to such 
boundary conditions can only be ensured if the boundary of the manifold satisfies some geometric 
constraints. We therefore have to choose the reference manifold very carefully. In fact, it turns 
out that there are good reasons to use a curved reference manifold even for flat submanifolds
such as moving domains in $\mathbb{R}^n$. It is therefore quite natural to formulate our approach in
a rather geometrical setting.

\subsection{Related work}
Finding good meshes for computational purposes 
is a research field that has been studied for quite a long time,
see for example \cite{Wi66}.
For stationary surfaces, a reasonable way of addressing this problem is to compute good parametrizations 
such as conformal parametrizations. Methods to compute conformal surface parametrizations
can be, for example, found in \cite{HATKSH00, JWYG04}.
In \cite{CLR04}, the authors relax the idea of using conformal parametrization.
A common feature of these approaches is that they were originally designed for stationary surfaces.
One possibility to transfer 
them to evolving surfaces is to apply the reparametrization on the discrete time levels.
For example, in \cite{CD03, St14} harmonic maps are used for the remeshing of closed moving surfaces.
Another approach for closed evolving manifolds based on elliptic PDEs was recently suggested 
in \cite{MRSS14}. However, we believe that the remeshing of evolving submanifolds 
should rather be based on reparametrizations by solutions to parabolic equations.

Remeshing schemes based on solutions to the harmonic map heat flow 
have also been used within the $r$-refinement (relocation refinement)
moving mesh method; see \cite{Hu01, HR99}.
In contrast to $hp$-methods, where the computational mesh is locally refined or coarsened
based on \textit{a posteriori} error estimates in order to obtain PDE-solutions within
prescribed error bounds, the $r$-refinement moving mesh method follows a different path
to get the smallest error possible for a fixed number of mesh vertices.
The idea behind this method is to move vertices to those regions where the PDE-solution
has "interesting behaviour"; see \cite{BHR09, HR11} for recent surveys of the method.  
This is achieved by mapping a given mesh for some reference domain, 
which is called the \textit{logical} or \textit{computational
domain} in this instance, into the \textit{physical domain} in which the underlying PDE
is posed in such a way that the associated map satisfies some moving mesh equations.
These equations depend on the underlying PDE via 
a so-called \textit{monitor function} which is specially designed to
guide the positions of the mesh vertices. 
One example of a moving mesh equation is
a gradient flow equation of an adaptation functional, which includes the energy of a harmonic 
mapping; see, for example, \cite{Hu01, HR99}. This approach is then 
called the moving mesh PDE method (MMPDE).
It is based on results developed in \cite{Dv90}, where harmonic maps on Riemannian manifolds are used for 
the generation of solution adaptive grids. 

Despite the formal similarities between
the MMPDE and the approach presented in this paper, such as the use of the
harmonic map heat flow, there are some crucial differences between both methods.
Firstly, the objectives of both methods are different. While the MMPDE aims to adapt the mesh
to a solution of some underlying PDE, 
the objective of our approach is to provide a good mesh for evolving submanifolds
undergoing large deformations, where the submanifold is allowed to have any dimension
or codimension. We here consider special boundary conditions
which will enable us to obtain high-quality meshes also at the boundary of the evolving submanifold
by solving just one equation!
This is in contrast to the MMPDE, where the point distribution at the boundary is often
obtained by some lower dimensional MMPDE for the boundary mesh; see the discussion in \cite{Hu01} and Section $5$ in \cite{MMNI15}. 
In order to apply our boundary conditions, we have to choose the reference submanifold
very carefully; see Section \ref{Section_numerical_examples}, where the reference manifolds are a half-sphere or a cylinder.
In contrast, the logical domain in the MMPDE is 
usually some rectangular domain. However, this does not mean that our approach is restricted in any way. It just means that we have to include curved geometries in our approach. 
Using the evolving surface finite element method (see \cite{DE13}), this can be realized very easily.
A crucial property of our approach is that it is not necessary to solve the harmonic map heat flow explicitly. 
Instead, it is sufficient to compute a reparametrized evolution equation for the motion of the evolving submanifold.
If the original velocity of the submanifold is known explicitly, this can be done by just inverting two mass matrices!
Our method is therefore computationally very cheap.
Moreover, our approach does not depend on a solution to another PDE, since  we use the Riemannian metric
determined by the embedding of the submanifold into some Euclidean space
and not by the solution to another PDE like in the MMPDE method. 
Since our method also makes use of mesh refinement and coarsening (see Algorithm 
\ref{algo_refinement_and_coarsening_strategy}), it is clearly not in the spirit of an
$r$-refinement method.

In \cite{MMNI15}, the MMPDE method was recently used for the generation of bulk and surface meshes in order to solve 
coupled bulk-surface reaction-diffusion equations on evolving two-dimensional domains.
Such problems occur, for example, in the modelling of cell migration and chemotaxis.
The mesh algorithm in \cite{MMNI15} is based on two moving mesh PDEs
-- one for the boundary and one for the interior of the domain. More precisely, the idea in \cite{MMNI15} is to use the
updated boundary points from the solution to the boundary problem as Dirichlet data for the MMPDE method applied to the interior mesh points.
Since we only make use of one PDE, that is the harmonic map heat flow with mixed boundary conditions,
the here presented method is clearly different from the approach in \cite{MMNI15}.

\subsection{Outline of the paper}
This paper is organised as follows. In Sections \ref{Section_setting} and \ref{Section_notations},
we introduce the formal setting of our approach and recall some basic facts from differential geometry.
In particular, we describe the evolving submanifold by a time-dependent embedding of some
fixed reference manifold with boundary into an Euclidean space.
This will be the framework for our further analysis.
We then present the harmonic map heat flow for manifolds with boundary in Section \ref{Section_harmonic_map_heat_flow}. 
We will consider the case of mixed boundary conditions. 
This means that the harmonic map heat flow is assumed to map the boundary of the reference manifold 
onto itself. However, the map is allowed to change in 
the tangential direction of the boundary. 
These boundary conditions imposed on the harmonic map heat flow on the boundary of the
reference manifold are the reason why we will obtain tangential 
redistributions of the mesh vertices on the boundary of the evolving submanifold
in our remeshing algorithm.
It is very important that the redistribution of the mesh vertices on the boundary only takes place 
in the tangential direction of the boundary
in order to ensure that the shape of the evolving submanifold is not changed by the remeshing method. 
We reparametrize the motion of the submanifold with boundary by the solution
to the harmonic map heat flow with boundary. This is done in Section \ref{Section_reparametrization}. 
This leads to a new velocity field for the motion of the reparametrized submanifold. 
In Section \ref{Section_weak_formulation}, a weak formulation is provided.
Since tangential gradients on submanifolds can be discretized quite naturally, we will reformulate our results
using tangential gradients in Section \ref{Section_reformulation_tangential_gradients}.
For the spatial discretization, we define appropriate finite element spaces in Section \ref{Section_finite_elements}.
In Section \ref{Section_discrete_problems},
we discretize the weak formulation based on tangential gradients and obtain a numerical scheme for the motion of the computational mesh.
In this scheme, the mesh vertices are
moved according to the reparametrized velocity field. 
A natural side effect of our approach is 
that the area of the mesh simplices tends to decrease or increase
non-homogeneously. We take this problem into account 
by introducing a refinement and coarsening strategy, which makes sure that the mesh simplices
have approximately a similar size. 
Since refinement and coarsening change the mesh quality only slightly, this step
does not affect the potential of the whole approach.
Details on the implementation of our novel scheme are given in Section \ref{Section_implementation}.
In Section \ref{Section_numerical_examples}, we present
numerical experiments which demonstrate the performance of our scheme to produce meshes of high quality in different settings. 
We show that our algorithm can be easily adapted to solve different problems with moving boundaries.
The paper ends with a short discussion of our results in Sections \ref{section_discussion}.

\section{Reparametrizations via the DeTurck trick}
\label{reparametrizations_via_the_DeTurck_trick}
\subsection{The setting}
\label{Section_setting}
Let $\Gamma(t) \subset \mathbb{R}^{n}$, $0 \leq t < T$, be a smooth family of $(n-d)$-dimensional, 
compact submanifolds with boundary in the Euclidean space $\mathbb{R}^n$,
and let $\Omega := \bigcup_{t \in [0,T)} \Gamma(t) \times \{ t \}$ 
be the corresponding space-time cylinder.
Here, $d \in \mathbb{N}_0$ denotes the co-dimension of the submanifold.
Without loss of generality, we can assume in this paper that $\Gamma(t)$ is (path-)connected. 
For example, if $d=0$, then $\Gamma(t)$ is the closure of some bounded domain $U(t) \subset \mathbb{R}^n$,
and if $d=1$, then $\Gamma(t)$ is a compact and connected hypersurface with boundary. The Euclidean metric
$\mathfrak{e}$ in $\mathbb{R}^n$
induces a metric on $\Gamma(t)$ which we denote by $e(t)$.

We now make the assumption that $\Gamma(t)$ is given as the image of 
a smooth, time-dependent embedding 
$x: \mathcal{M} \times [0,T) \rightarrow \Omega$ 
with $\Gamma(t) = x(\M,t)$ and $\partial \Gamma(t) = x(\partial \M,t)$, where $(\M,m)$ is an $(n-d)$-dimensional,
compact and connected smooth Riemannian manifold with boundary $\partial \M$. 
The manifold $\M$ is called the reference manifold of the problem.
The time-independent metric $m$ on $\M$ is arbitrary yet fixed.
We call $m$ the background metric in order to distinguish it from the
pull-back metric $g(t) := x(t)^\ast \mathfrak{e}$ on $\M$
induced by the embedding $x(t)$.
See Table \ref{list_of_symbols} for an overview of symbols used in the text.

The motion of the evolving submanifold $\Gamma(t)$ is described by the velocity field 
$v: \Omega \rightarrow \mathbb{R}^n$ given by
\begin{equation}
	v \circ x = x_t.
	\label{equation_of_motion_non-reparametrized}
\end{equation}

The embedding $x$ can be replaced by every reparametrization of the form
$\hat{x}:= x \circ \psi^{-1}$ without changing the space-time cylinder $\Omega$.
Here, $\psi: \M \times [0,T) \rightarrow \M$ 
denotes an arbitrary smooth family of diffeomorphisms on $\M$ 
with $\partial \M = \psi(\partial \M, t)$ for all $t \in [0,T)$. 
The velocity field 
$\hat{v}: \Omega \rightarrow \mathbb{R}^n$ of the reparametrization, that is
$\hat{v} := \hat{x}_t \circ \hat{x}^{-1}$, satisfies 
\begin{align*}
	\hat{v} \circ \hat{x} &= (x \circ \psi^{-1})_t 
	= x_t \circ \psi^{-1} + (\nabla x \circ \psi^{-1})(\psi^{-1})_t
	\\
	&=  v \circ x \circ \psi^{-1} + (\nabla x \circ \psi^{-1})(\psi^{-1})_t
	\\
	&=  v \circ \hat{x} + (\nabla x \circ \psi^{-1})(\psi^{-1})_t .
\end{align*}
Here, $\nabla x$ denotes the differential of the embedding $x$.
In local coordinates $\mathcal{C}$, it is given by 
$(\nabla x) \circ \mathcal{C} \big(\frac{\partial \mathcal{C}^{-1}}{\partial \theta^j} \big) = 
\frac{\partial X}{\partial \theta^j}$, where $X := x \circ \mathcal{C}^{-1}$.
Using the identities
\begin{align*}
	& (\psi^{-1})_t \circ \psi = - ( \nabla \psi^{-1} \circ \psi ) \psi_t,
	& \nabla \hat{x} = ( \nabla x \circ \psi^{-1} ) \nabla \psi^{-1},
\end{align*}
we conclude that 
$(\nabla x \circ \psi^{-1})(\psi^{-1})_t = - \nabla \hat{x} (\psi_t \circ \psi^{-1})$,
and hence,
$$
	\hat{v} \circ \hat{x} = v \circ \hat{x} - \nabla \hat{x} (\psi_t \circ \psi^{-1}).
$$
The reparametrized velocity field $\hat{v}$ depends on the time-derivative of
the reparametrization $\psi$. It is therefore an interesting question whether there is 
a general class of reparametrizations $\psi$ that lead to advantageous velocities $\hat{v}$
in the following sense: The easiest way to do computations on evolving submanifolds would
be to move the computational mesh according to the velocity field $v$. 
However, in general, this would lead to a degeneration of the mesh almost immediately.
This problem remains even if the problem is solved on the reference manifold
$\M$. Instead of mesh degenerations, one would then have to handle an induced metric $g(t)$
which becomes singular -- at least from a computational perspective. 
It would therefore be a big advantage in numerical simulations to have a velocity field
that does not lead to mesh degenerations
when it is used to move the mesh vertices. 
If such a velocity field is based on a reparametrization like above,
it does not change the space-time cylinder $\Omega$.  
This leads to the problem to find a good family of reparametrizations $\psi(t)$.
\begin{remark}
In applications, either the embedding $x$ or the velocity field $v$ might be given.
In the latter case, the embedding $x$ can be determined by solving the system of 
ordinary differential equations (\ref{equation_of_motion_non-reparametrized}) for a given initial embedding $x(\cdot, 0) = x_0(\cdot)$.
Note that the velocity field $v$ defines the parameterisation $x(t)$ as well as the domain $\Gamma(t)$. 
It does not necessarily correspond to some physical velocity. 
Often, examples with given velocity fields are free and moving boundary problems. 
\end{remark}

\subsection{Further notations}
\label{Section_notations}
Henceforward, the components of an arbitrary metric tensor $h$ 
with respect to some coordinate chart of an $(n-d)$-dimensional manifold
are denoted by $h_{ij}$ for $i,j= 1, \ldots, n-d$. The components of the inverse of the matrix 
$(h_{ij})_{i,j = 1 , \ldots, n-d}$ are denoted by $h^{ij}$ for $i,j= 1, \ldots, n-d$.
We here make use of the convention to sum over repeated indices.
The Christoffel symbols with respect to the metric $h$
are defined by
$$
\Gamma(h)^k_{ij} := \frac{1}{2} h^{km} \left( 
	\frac{\partial h_{mj}}{\partial \theta^i} + \frac{\partial h_{mi}}{\partial \theta^j}
	- \frac{\partial h_{ij}}{\partial \theta^m}
\right).
$$
The gradient $grad_{h} f$ of a differentiable function $f$ on a Riemannian manifold 
with respect to the metric $h$ is 
defined by  
$h(p)(grad_{h} f(p), \xi) := (\nabla f)(p)(\xi)$ for all tangent vectors $\xi$ at $p$.
In local coordinates, we have
\begin{align*}
	&	\big((grad_{h} f) \circ \mathcal{C}^{-1} \big)^\kappa 
	= h^{\kappa \sigma} \frac{\partial F}{\partial \theta^\sigma},
\end{align*} 
where $F := f \circ \mathcal{C}^{-1}$ and $\mathcal{C}$ is a local coordinate chart.
The Laplacian of a twice differentiable function $f$ with 
respect to the metric $h$ is defined by
$$
	( \Delta_{h} f ) \circ \mathcal{C}^{-1} := h^{\iota \eta} 
	\left( \frac{\partial^2 F}{\partial \theta^\iota \partial \theta^\eta} 
	- \Gamma(h)^\rho_{\iota \eta} \frac{\partial F}{\partial \theta^\rho} \right)
	= \frac{1}{\sqrt{\det (h_{\alpha \beta})}} \frac{\partial}{\partial \theta^\iota}
	\left( \sqrt{\det (h_{\alpha \beta})} h^{\iota \eta} \frac{\partial F}{\partial \theta^\eta} \right).
$$
The map Laplacian $\Delta_{g, m}$ of a
map $\psi: (\M, g) \rightarrow (\M, m)$ with respect to the metrics $g$ and $m$
is defined by
\begin{align}
	\left( \mathcal{C}_2 \circ (\Delta_{g,m} \psi) \circ \mathcal{C}_1^{-1} \right)^\kappa
	:= g^{ij} \left( \frac{\partial^2 \Psi^\kappa}{\partial \theta^i \partial \theta^j} 
		- \Gamma(g)^k_{ij} \frac{\partial \Psi^\kappa}{\partial \theta^k}
		+ \Gamma(m)^\kappa_{\beta \gamma} \circ \Psi \frac{\partial \Psi^\beta}{\partial \theta^i}
			\frac{\partial \Psi^\gamma}{\partial \theta^j}	
	\right),
	\label{map_laplacian_in_coordinates}
\end{align}
where $\mathcal{C}_1, \mathcal{C}_2$ are two coordinate chats of $\M$, 
and $\Psi := \mathcal{C}_2 \circ \psi \circ \mathcal{C}_1^{-1}$; see, for example, in \cite{CLN06}.
The indices $i,j,k$ refer to the chart $\mathcal{C}_1$, whereas $\kappa, \beta, \gamma$
refer to $\mathcal{C}_2$.  

\begin{table}
\small
\begin{center}
\begin{tabular}{|l | l|}
\hline
 $(\M,m)$    & reference manifold with fixed background metric $m$ \\
 $\Gamma(t) \subset \mathbb{R}^n$ & moving $(n-d)$-dimensional submanifold \\
 $x: \M \times [0,T) \rightarrow \Gamma(t)$ & embedding of $\M$ \\
 $\hat{x}(t) := x(t) \circ \psi(t)^{-1}$  & reparametrization of the embedding $x$ \\
 $\hat{y}(t) := \hat{x}(t)^{-1}$ & inverse of the embedding $\hat{x}$ \\
 $\psi: \M \times [0,T) \rightarrow \M$ & solution to the harmonic map heat flow \\
 $u(t): \Gamma(t) \rightarrow \Gamma(t)$ & identity function on $\Gamma(t)$ \\
 $\mathfrak{e}$ & Euclidean metric in the ambient space \\
 $e(t)$ & metric on $\Gamma(t)$ induced by the Euclidean metric\\
 $\hat{h}(t) := \hat{y}(t)^\ast m$ & pull-back metric on $\Gamma(t)$ \\
 $g(t) := x(t)^\ast \mathfrak{e}$,  $\hat{g}(t) := \hat{x}(t)^\ast \mathfrak{e}$ 
 & pull-back metrics on $\M$ \\
 $\nu(t)$ & unit co-normal vector field to $\partial \Gamma(t)$ with respect to $e(t)$\\
 $\mu(t)$ & unit co-normal vector field to $\partial \M$ with respect to $g(t)$ \\ 
 $\lambda$ & unit co-normal vector field to $\partial \M$ with respect to $m$ \\
\hline 
\end{tabular}
\caption{List of symbols}
\label{list_of_symbols}
\end{center}
\end{table}

The boundary $\partial \M$ of a Riemannian manifold $(\M,m)$ is called totally geodesic 
if any geodesic on the submanifold $\partial \M$ with respect to the metric induced by $m$ is also a geodesic in $(\M, m)$. 
This is equivalent to the fact that a geodesic $\gamma: (- \epsilon, \epsilon) \rightarrow \M$
in $(\M,m)$ with $\gamma(0) \in \partial \M$ and $\gamma'(0)$ tangential to $\partial M$
stays in $\partial M$.  
A simple class of such manifolds are given
by the $(n-1)$-dimensional half-spheres
\begin{equation}
	\mathbb{H}^{n-1} 
	:= \bigg\{ x \in \mathbb{R}^{n} \; | \; \sum_{j=1}^{n} x_j^2 = 1
	\; \textnormal{and} \; x_{1} \geq 0 \; \bigg\},
	\label{half_sphere_definition}
\end{equation}
with the metric induced by the Euclidean metric of the ambient space. The boundary
$$
	\partial \mathbb{H}^{n-1} 
	:= \bigg\{ x \in \mathbb{R}^{n} \; | \; \sum_{j=1}^{n} x_j^2 = 1
	\; \textnormal{and} \; x_{1} = 0 \; \bigg\}
$$
with respect to this metric is totally geodesic.

\subsection{The harmonic map heat flow on $\M$ }
\label{Section_harmonic_map_heat_flow}
In the following, we choose $\psi(t)$ to be the solution to the harmonic map heat flow
of manifolds with boundary. We will see that for this choice, a computational mesh 
of high quality is automatically generated by moving the mesh vertices 
according to the new velocity field $\hat{v}$. 
This is demonstrated in Section \ref{Section_numerical_examples} by numerical experiments.
To be precise we seek  $\psi$ solving the following initial-boundary value problem for the harmonic map heat flow 
\[
\left.
\begin{aligned}
	& \psi_t = \tfrac{1}{\alpha} \Delta_{g(t), m} \psi \quad 
	\textnormal{in $\M \times (0,T)$,}
\end{aligned}
\right\} =: (HMF) 
\]
with $g(t) = x(t)^\ast \mathfrak{e}$ and mixed boundary conditions
\[
\left.
\begin{aligned}	
	& \psi(\cdot, 0) = id(\cdot)
	& \textnormal{on $\M$}, \\
	& \nabla_{\mu(t)} \psi \perp_{m} \partial \M  
	& \textnormal{on $\partial \M \times (0,T)$}, \\
	& \psi(\partial \M, t) \subset \partial \M 
	& \textnormal{for all $t \in [0,T)$.}
\end{aligned}
\right\} =: (BC)
\]
Here, $\mu(t)$ denotes a unit co-normal vector field on $\partial \M$ with respect to the metric $g(t)$.
The second condition says that the normal derivative $\nabla_{\mu(t)} \psi$ is supposed to be
perpendicular to the boundary of $\M$ 
with respect to the metric $m$. 

\begin{itemize}
\item
We have introduced the inverse diffusion
constant $\alpha > 0$ in order to control the size of the velocity 
$\nabla \hat{x}(\psi_t \circ \psi^{-1})$ in $\hat{v}$. It corresponds to having differing time scales for the 
reparametrization and for the evolution of the surface.
This is important in applications, in particular, 
if the submanifold $\Gamma(t)$ moves very fast and the time scale $\alpha$, on which the 
redistribution of the mesh nodes takes place, has to be very small.
\item
The reason for using the mixed boundary conditions $\nabla_{\mu(t)} \psi \perp_{m} \partial \M$
and $\psi(\partial \M, t) \subset \partial \M$ is that these conditions ensure that
the boundary of $\M$ is mapped onto itself -- which would not be the case for
Neumann boundary conditions -- and that simultaneously, this map is flexible 
-- which would not be true for pure Dirichlet boundary conditions.
The latter point is crucial in order to obtain good submeshes at the boundary of $\Gamma(t)$.
\end{itemize}

\subsection{The reparametrization of the embedding}
\label{Section_reparametrization}
\begin{figure}
\begin{center}
\includegraphics[width = \textwidth]{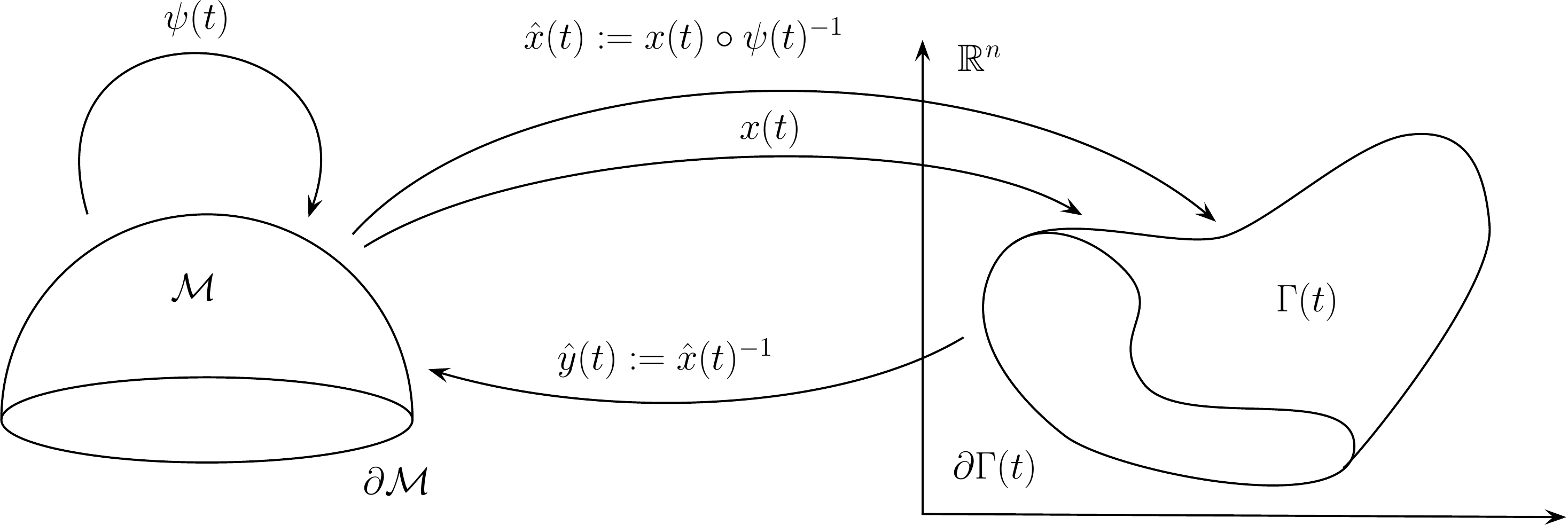}
\end{center}
\caption{Schematic picture of the reparametrization of the time-dependent embedding $x(t)$ by the solution $\psi(t)$ of the harmonic map heat flow
$(HMF)$. $\M$ is the reference manifold and $\Gamma(t) := x(\M,t)$ is the moving submanifold for which we aim to find a good computational mesh.}
\end{figure}
\begin{prop}
\label{Prop_reparametrization}
Suppose that $\psi(t): \M \rightarrow \M$ with $0 \leq t < T$ is a smooth family of diffeomorphism that solve 
the harmonic map heat flow $(HMF)$. Let $\Gamma(t) = x(\M,t)$ for $0 \leq t < T$ be a moving embedded submanifold in $\mathbb{R}^n$.  
The map $\hat{x}(t): \M \rightarrow \Gamma(t)$ for $0 \leq t < T$ defined as the pull-back 
$\hat{x}(t) := (\psi(t)^{-1})^\ast x := x(t) \circ \psi(t)^{-1}$ of the embedding $x(t)$
then satisfies the equation
\begin{align}
	& \hat{x}_t = v \circ \hat{x} - \tfrac{1}{\alpha} \nabla \hat{x}(w),
	\label{equation_for_reparametrized_motion} 
\end{align}
where $w$ is a tangent vector field on $\M$ whose components $W^k$ with respect to a coordinate chart $\mathcal{C}$,
that is $w \circ \mathcal{C}^{-1} = W^k \frac{\partial \mathcal{C}^{-1}}{\partial \theta^k}$,
are defined by
\begin{align}
	W^k := \hat{g}^{ij} \left( \Gamma(m)^k_{ij} - \Gamma(\hat{g})^k_{ij} \right).
	\label{w_in_local_coordinates}
\end{align}
Here, $\Gamma(\hat{g})^k_{ij}$ and $\Gamma(m)^k_{ij}$ denote the Christoffel symbols
with respect to the metrics $\hat{g}(t) := \hat{x}(t)^\ast \mathfrak{e}$ and $m$, respectively.
\end{prop}
\begin{proof}
With a slight abuse of notation, we first define $w$ to be the tangent vector field $w := \alpha ~ \psi_t \circ \psi^{-1}$ on $\M$. 
We then find that $\psi(t)$ and $\hat{x}(t)$ solve the following system of partial differential equations 
\begin{align*}
	& \hat{x}_t = v \circ \hat{x} - \tfrac{1}{\alpha} \nabla \hat{x}(w), 
	\\
	& w = \alpha ~ \psi_t \circ \psi^{-1}
	\\
	& \psi_t = \tfrac{1}{\alpha} \Delta_{g(t), m} \psi \quad 
\end{align*}
in $\M \times (0,T)$. 
The differential equation for the embedding $\hat{x}$ depends on the vector field $w$.
We will show now that we can eliminate the harmonic map heat flow in the above system of equations.
The reason is that the vector field $w$ can be computed from the reparametrized embedding $\hat{x}(t)$
by using formula (\ref{w_in_local_coordinates}).
This follows from Remark 2.46 in \cite{CLN06}, which states that
$$
	\Delta_{g(t), m} \psi = (\Delta_{(\psi(t)^{-1})^\ast g(t), m} id) \circ \psi,
$$
and from the fact that the pull-back metric $(\psi(t)^{-1})^\ast g(t)$ is equal to the induced
metric $\hat{g}(t)$, which can be seen as follows 
$$
	(\psi(t)^{-1})^\ast g(t) = (\psi(t)^{-1})^\ast x(t)^\ast \mathfrak{e}
	= ( x(t) \circ \psi(t)^{-1})^\ast \mathfrak{e} = \hat{x}(t)^\ast \mathfrak{e} = \hat{g}(t).
$$
This implies that
$$
	w = \alpha \; \psi_t \circ \psi^{-1} = ( \Delta_{g(t),m} \psi ) \circ \psi^{-1}
	= \Delta_{\hat{g}(t),m} id.
$$
From the definition of the map Laplacian in (\ref{map_laplacian_in_coordinates}),
we obtain that
$$
 (\mathcal{C} \circ (\Delta_{\hat{g}(t),m} id) \circ \mathcal{C}^{-1})^k
  = \hat{g}^{ij} \left( - \Gamma(\hat{g})^k_{ij} + \Gamma(m)^k_{ij} \right),
$$
which then gives (\ref{w_in_local_coordinates}). 
\end{proof}
Under sufficient smoothness conditions, the evolution equation for $\hat{x}$ also holds -- in the trace sense -- on the boundary of $\M$.
The above result is rather astonishing, since it shows that the evolution equation for the reparametrized embedding $\hat{x}(t)$
does not depend on the solution of the harmonic map heat flow $\psi(t)$. This is an important fact with respect to the computational costs
of our approach, because it means that it will not be necessary to compute the solution $\psi(t)$ to the harmonic map heat flow.
In the above proposition, we have not made use of the boundary conditions $(BC)$, which we will do now.
\begin{lem}
Suppose the harmonic map heat flow $\psi(t)$ satisfies the boundary condition $\psi(\partial \M, t) \subset \partial \M$.
Then the vector field $w$ on $\partial \M$, defined in Proposition \ref{Prop_reparametrization}, is tangential to $\partial \M$ and
$\nabla \hat{x}(w)$ is tangential to the boundary of $\Gamma(t)$.
Hence, the reparametrization by $\psi(t)$ only induces tangential motions on the boundary of $\Gamma(t)$.
\end{lem}
\begin{proof}
The statement easily follows from $w = \alpha \; \psi_t \circ \psi^{-1}$, see in the proof of Proposition \ref{Prop_reparametrization},
and from $\hat{x}(\partial \M, t) \subset \partial \Gamma(t)$.
\end{proof}
\begin{lem}
\label{Lemma_bc_normal_derivative}
Suppose the harmonic map heat flow $\psi(t)$ satisfies the boundary condition $\nabla_{\mu(t)} \psi \perp_m \partial \M$ on $\partial \M \times (0,T)$.
Furthermore, let $\hat{y}: \Omega \rightarrow \M$ be the map defined by $\hat{y}(t) := \hat{x}(t)^{-1} = \psi(t) \circ x(t)^{-1}$ for all $t \in [0,T)$,
where $\hat{x}(t)$ is the reparametrized embedding from Proposition \ref{Prop_reparametrization}. 
Then $\hat{y}(t)$ satisfies the condition
$$
\nabla_{\nu(t)} \hat{y} \perp_{m} \partial \M  \; \textnormal{on $\partial \Gamma(t) \times (0,T)$.}
$$
\end{lem}
\begin{proof}
Since $\mathfrak{e}(\nabla_{\mu(t)} x, \nabla_{\mu(t)} x) = g(t)(\mu, \mu) = 1$
and $\mathfrak{e}(\nabla_{\mu(t)} x, \nabla_{\xi} x) = g(t)(\mu(t), \xi) = 0$
for all vector fields $\xi$ that are tangential to $\partial \M$,
it follows that $\nu(t) := (\nabla_{\mu(t)} x) \circ x^{-1}$ is a unit co-normal vector field
on $\Gamma(t)$ with respect to the metric $e(t)$.
Using the fact that $\hat{y}(t) = \psi(t) \circ x^{-1}(t)$, we find the 
additional boundary condition
$$
\nabla_{\nu(t)} \hat{y} = (\nabla \psi) \circ x^{-1} (\nabla_{\nu(t)} x^{-1})
= (\nabla \psi)(\mu(t))\circ x^{-1} = (\nabla_{\mu(t)} \psi) \circ \psi^{-1} \circ \hat{y} 
\quad \textnormal{on $\partial \Gamma(t)$}.
$$
Since $(\nabla_{\mu(t)} \psi) \circ \psi^{-1} \perp_m \partial \M$,
this means that $m(\hat{y}(p,t), t)(\nabla_{\nu(t)} \hat{y}(p,t), \xi) =0$
for all tangent vectors $\xi$ of $\partial \M$ at the point $\hat{y}(p,t)$.
Hence, we have $\nabla_{\nu(t)} \hat{y} \perp_m \partial \M$
on $\partial \Gamma(t) \times (0,T)$.
\end{proof}

\begin{remark}
Due to the initial condition $\psi(\cdot,0) = id(\cdot)$ in $(BC)$, we have $\hat{x}(0) = x(0) = x_0$. 
This result is also true for arbitrary $\psi(\cdot,0) = \psi_0$, if we replace the definition of $\hat{x}(t)$
by $\hat{x}(t) := x(t) \circ \psi_0 \circ \psi(t)^{-1}$. In this case, Proposition \ref{Prop_reparametrization} still holds.  
\end{remark}

\subsubsection*{Uniqueness}
Suppose that the embeddings $\hat{x}_1(t)$ and $\hat{x}_2(t)$ for the submanifold $\Gamma(t)$ (i.e. $\Gamma(t) = \hat{x}_1(\M,t) = \hat{x}_2(\M,t)$)
are solutions to (\ref{equation_for_reparametrized_motion}), that is
\begin{equation*}
	(\hat{x}_r)_t = v \circ \hat{x}_r - \tfrac{1}{\alpha} \nabla \hat{x}_r(w_r),
\end{equation*}
with $\hat{x}_r(\cdot, 0) = x_0(\cdot)$ for $r=1,2$.
Here, $w_{r} \circ \mathcal{C}^{-1} = W^k_{r} \frac{\partial \mathcal{C}^{-1}}{\partial \theta^k}$ is given by
$$
	W^k_{r} = \hat{g}^{ij}_{r} \left( 
		\Gamma(m)^k_{ij} - \Gamma(\hat{g}_{r})^k_{ij}	
	\right),
$$
where $\hat{g}_r(t) := \hat{x}_r(t)^\ast \mathfrak{e}$ for $r=1,2$. Furthermore, assume
that the vector fields $w_r$ are tangential to the boundary of $\M$ and that the inverse maps 
$\hat{y}_r(t) := \hat{x}_r(t)^{-1}$ satisfy the boundary condition from Lemma \ref{Lemma_bc_normal_derivative}, that is
$
\nabla_{\nu(t)} \hat{y}_r \perp_{m} \partial \M.
$ 
We will now show that $\hat{x}_1(t) = \hat{x}_2(t)$, provided that $w_r$ is regular enough to ensure that the solutions 
$\psi_r: \M \times[0,T) \rightarrow \M$ for $r= 1,2$ to the ODEs 
$$
	(\psi_r)_t = \tfrac{1}{\alpha} w_r \circ \psi_r
$$
with $\psi_r(\cdot, 0) = id(\cdot)$ on $\M$, remain diffeomorphisms for all times $t \in [0,T)$.
A short calculation shows that the maps $x_r(t) := \hat{x}_r(t) \circ \psi_r(t)$, for $r=1,2$, then satisfy
$$
	(x_r)_t = v \circ x_r
$$
with $x_r(\cdot,0) = x_0(\cdot)$. Since the solution to this ODE is unique, we indeed have $x_1(t) = x_2(t)$. 
It therefore remains to show that $\psi_1(t) = \psi_2(t)$.
We observe that
\begin{align*}
	(\psi_r)_t &= \tfrac{1}{\alpha} \big( \Delta_{\hat{g}_r(t),m} id \big) \circ \psi_r
	\\
	&= \tfrac{1}{\alpha} \Delta_{\psi_r(t)^\ast \hat{g}_r(t),m} \psi_r,
\end{align*}
where we have made use of Remark 2.46 in \cite{CLN06} again.
Furthermore, we have
$$
	\psi_r(t)^\ast \hat{g}_r(t) = \psi_r(t)^\ast (\hat{x}_r(t)^\ast \mathfrak{e})
	= (\hat{x}_r(t) \circ \psi_r(t))^\ast \mathfrak{e}
	= x_r(t)^\ast \mathfrak{e}.
$$
Since $x_1(t) = x_2(t)$, this shows that $\psi_1(t)^\ast \hat{g}_1(t) = \psi_2(t)^\ast \hat{g}_2(t)$.
Hence,
$$
	(\psi_r)_t = \tfrac{1}{\alpha} \Delta_{g(t),m} \psi_r,
$$
with $g(t):= \psi_r(t)^\ast \hat{g}_r(t)$ and $\psi_r(\cdot, 0) = id(\cdot)$ on $\M$.
Since $w_r$ is tangential on the boundary of $\M$,  it also follows that $\psi_r(\partial \M, t) \subset \partial \M$.
Furthermore, we conclude that
$$
	\nabla_{\nu(t)} \hat{y}_r = \nabla_{\nu(t)} \hat{x}_r^{-1} = \nabla \psi_r (\nabla_{\nu(t)} x_r^{-1}),
$$
and thus $\nabla \psi_r (\nabla_{\nu(t)} x_r^{-1}) \perp_m \partial \M$. Like in the proof of Lemma \ref{Lemma_bc_normal_derivative},
we can choose the co-normal $\nu(t) = (\nabla_{\mu(t)} x_r) \circ x_r^{-1}$, where $\mu(t)$ is a unit co-normal field on $\partial \M$
with respect to the metric $g(t)$. 
This implies that $\nabla_{\mu(t)} \psi_r \perp_m \partial \M$. 
From the uniqueness of the harmonic map heat flow, we finally obtain that $\psi_1(t) = \psi_2(t)$ and therefore 
$\hat{x}_1(t) = x_1(t) \circ \psi_1(t)^{-1} = x_2(t) \circ \psi_2(t)^{-1} = \hat{x}_2(t)$.

\subsubsection*{Existence}
Existence of solutions to equation (\ref{equation_for_reparametrized_motion}) directly follows from the proof of Proposition \ref{Prop_reparametrization}
and the existence of solutions to the harmonic map heat flow $(HMF)$ with mixed boundary conditions $(BC)$.
In \cite{Ham75}, uniqueness and existence of solutions to this flow was proved under certain assumptions.
For long-time existence, sufficient conditions are that the Riemannian curvature of $\M$
with respect to the metric $m$ is non-positive and that the boundary $\partial \M$ is totally geodesic
with respect to the metric $m$.

\subsubsection*{Assumptions on the reference manifold $\M$}
Henceforward, we will drop the condition that $\M$ has Riemannian curvature $\leq 0$ 
with respect to the metric $m$ for the following reasons. 
First, short-time existence to $(HMF)$ does not depend on the curvature of $\M$.
This means that the following statements are valid as long as the harmonic map heat flow exists.
Second, it is known that for harmonic map heat flows with Dirichlet boundary conditions, 
the curvature condition can be replaced by a small range condition, see \cite{Jo81}.
We therefore think that the negation of the curvature condition will not affect the performance of our numerical method
in applications. 

In contrast, we will keep the condition that $\M$ has totally geodesic boundary with respect to $\M$.
The reason is that such reference manifolds can be found or constructed very easily (see the remark below).
Furthermore, it will turn out in Section \ref{Section_implementation} that for typical examples
of such reference manifolds such as the half-sphere and the cylinder (\ref{definition_cylinder}), 
the implementation of the boundary condition (\ref{boundary_condition_lambda_zeta})
becomes straightforward, since then the co-normal with respect to $m$ is a constant vector field.

The condition that $\M$ is supposed to have totally geodesic boundary, however, implies that the reference manifold
must be curved even if the moving submanifold $\Gamma(t)$ is flat. This can be seen from the following argument.
Since geodesics in an Euclidean space are straight lines, there is no bounded domain
in $\mathbb{R}^2$ that has a smooth totally geodesic boundary and that can therefore be used as reference manifold.
\begin{remark}
The half-spheres $\mathbb{H}^{n-1}$ defined in (\ref{half_sphere_definition}) provide a reference manifold $\M$ for a wide range of applications,
that is for all evolving submanifolds $\Gamma(t)$ 
that are given as a time-dependent embedding of $\mathbb{H}^{n-1}$.
For example, for $\Gamma(t)$ being the closure of a moving, simply-connected domain 
$U(t) \subset \mathbb{R}^2$, the reference manifold $\M$ can be chosen to be the 
two-dimensional half-sphere $\mathbb{H}^2 \subset \mathbb{R}^3$. 
\end{remark}

\subsubsection{The identity map $u$}
Since we are interested in the motion of $\Gamma(t)$ and not in the embedding $\hat{x}(t)$,
we aim to reformulate 
\begin{equation*}
	\hat{x}_t = v \circ \hat{x} - \tfrac{1}{\alpha} \nabla \hat{x}(w),
\end{equation*}
with $w$ given by (\ref{w_in_local_coordinates})
on the evolving submanifold $\Gamma(t)$. We therefore introduce the map $u: \Omega \rightarrow \Omega$ with
$u(p,t) = p$ for all $p \in \Gamma(t)$ and $t \in [0,T)$.
\begin{defn}
The material derivative $\partial^\bullet f$ of a differentiable function $f$ on $\Gamma(t)$ with respect to the embedding $\hat{x}$
is defined by
\begin{equation}
	(\partial^\bullet f) \circ \hat{x} := \frac{d}{dt} (f \circ \hat{x}).
	\label{material_derivative_definition}
\end{equation}
\end{defn}
The material derivative $\partial^\bullet u$ of $u$ is obviously given by
$$
	\partial^\bullet u = \hat{x}_t \circ \hat{x}^{-1}.
$$
This directly leads to the following result.
\begin{cor}
The identity map satisfies the equation
$$
	\partial^\bullet u = v - \tfrac{1}{\alpha} \big( \nabla \hat{x}(w) \big) \circ \hat{y}.
$$
\end{cor}

Due to the Nash embedding theorem, we can w.l.o.g. assume that the Riemannian manifold $(\M,m)$
is isometrically embedded into a $k$-dimensional Euclidean space $(\mathbb{R}^k, \mathfrak{e})$
for $k$ sufficiently large. The half-spheres $\mathbb{H}^{n-1}$, which we will use as
reference manifolds in Section \ref{section_numerical_results},
are embedded into $\mathbb{R}^{n}$ by definition.
Under this assumption, the metric $m$ is induced by the Euclidean metric
of the ambient space.
We are now going to prove the following result.
\begin{prop}
\label{Proposition_for_identity_map}
Suppose that the reference manifold $(\M,m)$ is isometrically embedded into an Euclidean space $(\mathbb{R}^k, \mathfrak{e})$.
The identity map $u$ then satisfies the equation
\begin{align}
	& \partial^\bullet u = v 
	- \tfrac{1}{\alpha} \nabla u \big( (grad_{\hat{h}(t)} \hat{y})^T \zeta \big),
	\label{equation_for_u}
\end{align}	
where $\hat{h}(t) := \hat{y}(t)^\ast m$ is the pull-back metric of $m$ onto $\Gamma(t)$ and the vector field $\zeta: \Omega \rightarrow \mathbb{R}^k$ is given by
\begin{align}
	&\zeta - \Delta_{e(t)} \hat{y} = 0.
	\label{equation_for_zeta}
\end{align}
Furthermore, on the boundary we find the following equations for $\zeta(t)$
and $\hat{y}(t)$,
\begin{align}	
	& (\lambda \circ \hat{y}) \cdot \zeta = 0
	& \textnormal{on $\partial \Gamma(t)$}, 
	\label{boundary_condition_lambda_zeta}
	\\
	& \nabla_{\nu(t)} \hat{y} \perp_{m} \partial \M  
	& \textnormal{on $\partial \Gamma(t) \times (0,T)$}, 
	\label{boundary_condition_nabla_y}	
	\\
	& \hat{y}(\partial \Gamma(t), t) \subset \partial \M 
	& \textnormal{for all $t \in [0,T)$}. \nonumber
\end{align}
Here, $\lambda$ is a unit co-normal vector field on $\partial \M$ with respect to the metric $m$
and  $\nu(t)$ is a unit co-normal field on $\partial \Gamma(t)$ with respect to the metric $e(t)$.
\end{prop}
\begin{proof}
Obviously, $\hat{y}$ satisfies the boundary condition 
$\hat{y}(\partial \Gamma(t), t) \subset \partial \M$ for all $t \in [0,T)$.
The boundary condition (\ref{boundary_condition_nabla_y}) has already been proved in Lemma \ref{Lemma_bc_normal_derivative}.
In the following, we use the notation 
$\hat{X}(t) := \hat{x}(t) \circ \mathcal{C}^{-1}$ for the map $\hat{x}$ on $\M$,
where $\mathcal{C}$ is a local coordinate chart of $\M$,
as well as $U(t) := u(t) \circ \hat{X} = \hat{X}$ and 
$\hat{Y}(t) := \hat{y}(t) \circ \hat{X} = \mathcal{C}^{-1}$ for the maps
$u(t)$ and $\hat{y}(t)$ on $\Gamma(t)$.
The bull-back metric $\hat{h}(t) = \hat{y}(t)^\ast m$ on $\Gamma(t)$ is then locally given by
$$
	\hat{h}_{\kappa \eta} = (m_{ij} \circ \mathcal{C} \circ \hat{Y}) 
	\frac{\partial \hat{Y}^i}{\partial \theta^\kappa} 
	\frac{\partial \hat{Y}^j}{\partial \theta^\eta},
$$
where Greek indices refer to the coordinate chart $\hat{X}(t)^{-1}$ of $\Gamma(t)$
and Latin indices to the chart $\mathcal{C}$ of $\M$. 
Please note that the charts $\mathcal{C}$ and $\hat{X}(t)^{-1}$ have the same image
and that $\hat{Y}^i(\theta) = (\mathcal{C} \circ \hat{Y})^i(\theta) = \theta^i$.
Hence, we indeed have
$$
	\hat{h}_{\kappa \eta} = m_{ij} \delta_\kappa^i \delta_\eta^j.
$$ 
Since the metric $e(t)$ satisfy $e(t) = \hat{y}(t)^\ast \hat{g}(t)$,
a similar relation holds between the components $e_{\kappa \eta}$ of $e(t)$
and the components $\hat{g}_{ij}$ of $\hat{g}(t)$. Using these relations, a short calculation
in local coordinates gives 
\begin{align*}
	\big( \nabla \hat{x} (w) \big) \circ \hat{Y}
	&= \big( \nabla \hat{x} (w) \big) \circ \mathcal{C}^{-1} 
	= W^k \frac{\partial \hat{X}}{\partial \theta^k}
	= \hat{g}^{ij} \left( \Gamma(m)^k_{ij} - \Gamma(\hat{g})^k_{ij} \right) 
		\frac{\partial \hat{X}}{\partial \theta^k}
\\	&= 	e^{\iota \eta} \left( \Gamma(\hat{h})^\kappa_{\iota \eta} 
									- \Gamma(e)^\kappa_{\iota \eta} \right) 
		\frac{\partial U}{\partial \theta^\kappa}.	
\end{align*}
We define the tangential vector field $z(t)$ on $\Gamma(t)$ locally by 
$z \circ \hat{X} := Z^\kappa \frac{\partial \hat{X}}{\partial \theta^\kappa}$
with components
$$
	Z^\kappa := e^{\iota \eta} \big( \Gamma(\hat{h})^\kappa_{\iota \eta} 
									- \Gamma(e)^\kappa_{\iota \eta} \big). 
$$   
We then have
$$
	\big( \nabla \hat{x} (w) \big) \circ \hat{y} = \nabla u (z),
$$
and thus,
$$
	\partial^\bullet u = v - \tfrac{1}{\alpha} \nabla u (z).
$$
In order to solve the above equation, we need an efficient way to compute the vector field $z$.
Since $(\M, m)$ is isometrically embedded into $(\mathbb{R}^k, \mathfrak{e})$,
the local components of the metric $\hat{h}(t) = \hat{y}(t)^\ast m$
satisfy
$$
	\hat{h}_{\kappa \eta} = \frac{\partial \hat{Y}}{\partial \theta^\kappa}
		\cdot \frac{\partial \hat{Y}}{\partial \theta^\eta},
$$
where the Euclidean metric $\mathfrak{e}$ is denoted by $\cdot$ for the sake of convenience.
We now write $Z^\kappa$ by
$$
	Z^\kappa 
	= e^{\iota \eta} \big( \Gamma(\hat{h})^\rho_{\iota \eta} - \Gamma(e)^\rho_{\iota \eta} \big)
	\frac{\partial \hat{Y}}{\partial \theta^\rho} \cdot 
	\frac{\partial \hat{Y}}{\partial \theta^\sigma} \hat{h}^{\sigma \kappa}.
$$   
A short calculation shows that
$$
  e^{\iota \eta} \Gamma(\hat{h})^\rho_{\iota \eta} 
  \frac{\partial \hat{Y}}{\partial \theta^\rho} \cdot
  \frac{\partial \hat{Y}}{\partial \theta^\sigma}
  = e^{\iota \eta} \frac{\partial^2 \hat{Y}}{\partial \theta^\iota \partial \theta^\eta}
  \cdot \frac{\partial \hat{Y}}{\partial \theta^\sigma},
$$
and thus,
$$
	Z^\kappa 
	= e^{\iota \eta} \left( \frac{\partial^2 \hat{Y}}{\partial \theta^\iota \partial \theta^\eta} 
	- \Gamma(e)^\rho_{\iota \eta} \frac{\partial \hat{Y}}{\partial \theta^\rho} \right)
	\cdot 
	\frac{\partial \hat{Y}}{\partial \theta^\sigma} \hat{h}^{\sigma \kappa}.
$$
Using the notation of the gradient with respect to the metric $\hat{h}(t)$ and of the Laplacian with respect to the metric $e(t)$, see Section \ref{Section_notations}, we conclude that 
$$
	z = (grad_{\hat{h}(t)} \hat{y}(t))^T \Delta_{e(t)} \hat{y}(t),
$$
with 
$$
	(grad_{\hat{h}(t)} \hat{y}(t))^T = \left(
			grad_{\hat{h}(t)} \hat{y}^1(t),
			\hdots,
			grad_{\hat{h}(t)} \hat{y}^k(t)
		\right),
$$
where $(\hat{y}^1(t), \ldots, \hat{y}^k(t))^T$ are the components of $\hat{y}(t)$
with respect to the coordinates of the ambient space.
We define $\zeta(t) := \Delta_{e(t)} \hat{y}(t)$ and finally obtain the system
\begin{align*}
	& \partial^\bullet u = v 
	- \tfrac{1}{\alpha} \nabla u \big( (grad_{\hat{h}(t)} \hat{y})^T \zeta \big),
	\\
	&\zeta - \Delta_{e(t)} \hat{y} = 0.
\end{align*}
The vector field $(grad_{\hat{h}(t)} \hat{y})^T \zeta$ must be tangential to the boundary
of $\Gamma(t)$, since this is true for $\nabla u (z) = z$.  
However, this is equivalent to the fact that 
$(\lambda \circ \hat{y}) \cdot \zeta = 0$, where $\lambda$ is a unit co-normal 
vector field on $\partial \M$ with respect to the metric $m$. 
Please recall that a unit co-normal vector field on $\partial \M$
with respect to the metric $g(t)$ is denoted by $\mu(t)$.
This equivalence follows easily from the fact that
$\sum_{j=1}^k \hat{h}(p,t)(grad_{\hat{h}(t)} \hat{y}^j, \xi) \zeta_j = \nabla_{\xi} \hat{y} \cdot \zeta$ 
for all tangent vectors $\xi$ of $\Gamma(t)$ at $p$ and that
$\nabla_{\xi} \hat{y} = \lambda \circ \hat{y}$ (or respectively, $\nabla_{\xi} \hat{y} = -\lambda \circ \hat{y}$)
if $\xi$ is a unit co-normal on $\partial \Gamma(t)$
with respect to $\hat{h}(t)$. 
The latter point is a direct consequence of $\hat{h}(t) := \hat{y}(t)^\ast m$
and $\hat{y}(\partial \Gamma(t), t) \subset \partial \M$. 
\end{proof}
Note, that, in general, the vector field $\zeta$ is not tangential to $\M$.

\subsection{Weak formulation}
\label{Section_weak_formulation}
In order to derive a weak formulation of (\ref{equation_for_u}) and (\ref{equation_for_zeta}), we multiply 
by test functions $\varphi \in L^2(\Gamma(t), \mathbb{R}^n)$, and respectively, by
$\phi \in \mathcal{S} := \left\{ H^{1,2}(\Gamma(t), \mathbb{R}^k) \; | \; 
(\lambda \circ \hat{y}) \cdot \phi = 0 \; \textnormal{on $\partial \Gamma(t)$} \right\}$.
On the boundary $\partial \Gamma(t)$, we multiply the equation for the identity map $u$
by a test function $\eta \in L^2(\partial \Gamma(t), \mathbb{R}^n)$.
We then integrate on $\Gamma(t)$ and on $\partial \Gamma(t)$ with respect to the Riemannian volume forms associated with the metric $e(t)$ on $\Gamma(t)$.
By abuse of notation, $do$ therefore denotes the Riemannian volume form induced by $e(t)$ on $\Gamma(t)$ and also on its boundary $\partial \Gamma(t)$. 
Altogether, we obtain
\begin{align}
& \left.
\begin{aligned}
	&\int_{\Gamma(t)} \partial^\bullet u \cdot \varphi
	+ \tfrac{1}{\alpha} \nabla u \big( (grad_{\hat{h}(t)} \hat{y})^T \zeta \big) \cdot \varphi \; do
	= \int_{\Gamma(t)} v \cdot \varphi \; do,
	\quad \forall \varphi \in L^2(\Gamma(t), \mathbb{R}^n),
	\\
	&\int_{\partial \Gamma(t)} \partial^\bullet u \cdot \eta
	+ \tfrac{1}{\alpha} \nabla u \big( (grad_{\hat{h}(t)} \hat{y})^T \zeta \big) \cdot \eta \; do
	= \int_{\partial \Gamma(t)} v \cdot \eta \; do,
	\quad \forall \eta \in L^2(\partial \Gamma(t), \mathbb{R}^n),
\end{aligned}
\right\}	
	\label{weak_formulation_for_u}
	\\
	&\int_{\Gamma(t)} \zeta \cdot \phi \; do
	+ \int_{\Gamma(t)} grad_{e(t)} \hat{y} : grad_{e(t)} \phi \; do = 0,
	\quad \forall \phi \in \mathcal{S},
	\label{weak_formulation_for_zeta}
\end{align}
where $grad_{e(t)} \hat{y} : grad_{e(t)} \phi 
:= \sum_{j=1}^k e(t)( grad_{e(t)} \hat{y}^j, grad_{e(t)} \phi^j)$	.
The equation (\ref{weak_formulation_for_zeta}) follows from the fact that
$$
	\sum_{j=1}^k \int_{\partial \Gamma(t)} e(t)(grad_{e(t)} \hat{y}^j, \nu(t)) \phi^j \; do 
	= \int_{\partial \Gamma(t)} \nabla_{\nu(t)} \hat{y} \cdot \phi \; do
	= \int_{\partial \Gamma(t)} (\nabla_{\nu(t)} \hat{y} \cdot (\lambda \circ \hat{y})) 
		(\lambda \circ \hat{y}) \cdot \phi \; do
	= 0,
$$
where we have used (\ref{boundary_condition_nabla_y}) and $(\lambda \circ \hat{y}) \cdot \phi = 0$.

Equation (\ref{weak_formulation_for_zeta}) is the only leftover from the harmonic map heat flow.
In particular, the mixed boundary conditions on $\partial \M$ are hidden in this equation.
For example, the condition $\nabla_{\mu(t)} \psi \perp_m \partial \M$ was first reformulated
as $\nabla_{\nu(t)} \hat{y} \perp_{m} \partial \M$, which we then used in order
to derive the weak formulation. The condition $\psi(\partial \M, t) \subset \partial \M$
on the other hand led to the condition $(\lambda \circ \hat{y}) \cdot \zeta = 0$.
We will take this equation into account by solving (\ref{weak_formulation_for_zeta})
in an appropriate space. For the moment, we just observe that $(\lambda \circ \hat{y}) \cdot \zeta = 0$ if $\zeta \in \mathcal{S}$.
The harmonic map heat flow itself will never be computed in our approach.

\subsection{Reformulation using tangential gradients}
\label{Section_reformulation_tangential_gradients}
In order to discretize the above weak formulation in space, we first rewrite it
using tangential gradients. Since the definition of the tangential gradient does not make use of any 
coordinate charts, it can be easily generalized to simplicial meshes.
The discretization of a weak formulation based on tangential gradients 
is hence straightforward. 
\begin{defn}
Let $f$ be a differentiable function on the submanifold $\Gamma(t) \subset \mathbb{R}^n$.
The tangential gradient of $f$ in $p \in \Gamma(t)$ is defined by
\begin{equation}
	\nabla_{\Gamma(t)} f(p) := (P \nabla \tilde{f})(p),
	\label{definition_tangential_gradient}
\end{equation} 
where $\nabla \tilde{f}$ is the usual gradient in $\mathbb{R}^n$ of a differentiable extension
$\tilde{f}$ of $f$ to an open neighbourhood of $p$. 
Here, $P$ denotes the tangential projection onto the tangent bundle of $\Gamma(t)$.
\end{defn}
It is easy to show that this definition
does not depend on the choice of the extension, see \cite{DDE05}.
Since we have 
$$
	(\nabla_{\Gamma(t)} f) \circ \hat{X} = e^{\kappa \sigma} 
	\frac{\partial F}{\partial \theta^\sigma}
	\frac{\partial \hat{X}}{\partial \theta^\kappa},
$$
with $F = f \circ \hat{X}$, it follows that
$$
	\nabla_{\Gamma(t)} f = grad_{e(t)} f,
$$
and in particular,
$$
	\xi \cdot \nabla_{\Gamma(t)} f = e(t)(\xi, grad_{e(t)} f) = \nabla_{\xi} f,
$$
for all tangent vector fields $\xi$ on $\Gamma(t)$.
In order to find a similar expression for $grad_{\hat{h}(t)} f$, we introduce the following
representation of the metric $\hat{h}(t)$.
\begin{defn}
The map $\hat{H}: \Omega \rightarrow \mathbb{R}^{n \times n}$ is defined by
$$
	\hat{H} := (\nabla_{\Gamma(t)} \hat{y})^T \nabla_{\Gamma(t)} \hat{y}
	+ \unit - P.
$$
\end{defn}
The map $\hat{H}(p)$ acts as a linear isomorphism on the tangent space of $\Gamma(t)$ in the point $p$,
and on the corresponding normal space, it is the identity. This implies that $\hat{H}$ is
invertible in each point $p$. Furthermore, we find the following results.
\begin{lem}
\label{H_represents_h}
In local coordinates the following identity holds
$$
	\frac{\partial \hat{X}}{\partial \theta^\kappa} \cdot (\hat{H} \circ \hat{X})
	\frac{\partial \hat{X}}{\partial \theta^\sigma} = \hat{h}_{\kappa \sigma}.
$$
\end{lem}
\begin{proof}
\begin{align*}
	\frac{\partial \hat{X}}{\partial \theta^\kappa} \cdot (\hat{H} \circ \hat{X})
	\frac{\partial \hat{X}}{\partial \theta^\sigma}
	&= \bigg( \frac{\partial \hat{X}}{\partial \theta^\kappa} \cdot
	          \frac{\partial \hat{X}}{\partial \theta^\iota}  \bigg) e^{\iota \eta} 
	   \bigg( \frac{\partial \hat{Y}}{\partial \theta^\eta}   \cdot 
	          \frac{\partial \hat{Y}}{\partial \theta^\gamma} \bigg) e^{\gamma \beta}
	   \bigg( \frac{\partial \hat{X}}{\partial \theta^\beta}  \cdot
		      \frac{\partial \hat{X}}{\partial \theta^\sigma} \bigg)
	\\
	&= e_{\kappa \iota} 	e^{\iota \eta} \hat{h}_{\eta \gamma} e^{\gamma \beta} e_{\beta \sigma}	      
	= \hat{h}_{\kappa \sigma}.	      
\end{align*}
\end{proof}
\begin{lem}
Let $f$ be a differentiable function on $\Gamma(t)$. Then we have
$$
	\hat{H}^{-1} \nabla_{\Gamma(t)} f = grad_{\hat{h}(t)} f.
$$
\end{lem}
\begin{proof}
 From 
\begin{align*}
	\hat{H} \circ \hat{X} \bigg( \frac{\partial \hat{X}}{\partial \theta^\kappa} 
		\hat{h}^{\kappa \sigma} \frac{\partial \hat{X}}{\partial \theta^\sigma}
		+ (\unit - P) \circ \hat{X} \bigg)
	 &= \frac{\partial \hat{X}}{\partial \theta^\iota} e^{\iota \eta} \hat{h}_{\eta \gamma}
	 	e^{\gamma \beta} e_{\beta \kappa} \hat{h}^{\kappa \sigma} 
	 	\frac{\partial \hat{X}}{\partial \theta^\sigma}
	    + (\unit - P) \circ \hat{X}
	 \\
	 &= \frac{\partial \hat{X}}{\partial \theta^\iota} e^{\iota \sigma} 
	 	\frac{\partial \hat{X}}{\partial \theta^\sigma}
	    + (\unit - P) \circ \hat{X}   
	  = \unit,  
\end{align*}
we conclude that 
$$
	\hat{H}^{-1} \circ \hat{X} = \frac{\partial \hat{X}}{\partial \theta^\kappa} 
		\hat{h}^{\kappa \sigma} \frac{\partial \hat{X}}{\partial \theta^\sigma}
		+ (\unit - P) \circ \hat{X}.
$$
It follows that
$$
	(\hat{H}^{-1} \nabla_{\Gamma(t)} f) \circ \hat{X}
	= \frac{\partial \hat{X}}{\partial \theta^\kappa} 
		\hat{h}^{\kappa \sigma} \frac{\partial \hat{X}}{\partial \theta^\sigma}
		\cdot \frac{\partial \hat{X}}{\partial \theta^\iota} e^{\iota \eta} 
		\frac{\partial F}{\partial \theta^\eta}
	= \frac{\partial \hat{X}}{\partial \theta^\kappa} 
		\hat{h}^{\kappa \sigma} e_{\sigma \iota} e^{\iota \eta} 
		\frac{\partial F}{\partial \theta^\eta}	
	= \frac{\partial \hat{X}}{\partial \theta^\kappa} 
		\hat{h}^{\kappa \sigma}
		\frac{\partial F}{\partial \theta^\sigma}.		
$$
\end{proof}
\begin{remark}
Lemma \ref{H_represents_h} says that the map $\hat{H}$ is a global representation of the metric $\hat{h}$ on $\Gamma(t)$.
\end{remark}
Using the above results, we can rewrite (\ref{weak_formulation_for_u}) and 
(\ref{weak_formulation_for_zeta}) on the moving submanifold $\Gamma(t)$.
\begin{thm}
Under the assumptions of Propositions \ref{Prop_reparametrization} and \ref{Proposition_for_identity_map}, the identity map $u(t)$
on $\Gamma(t)$ satisfies
\begin{align}
& \left.
\begin{aligned}
	&\int_{\Gamma(t)} \partial^\bullet u \cdot \varphi 
	+ \tfrac{1}{\alpha} \nabla_{\Gamma(t)} u 
		\hat{H}^{-1} (\nabla_{\Gamma(t)} \hat{y})^T \zeta \cdot \varphi \; do
	= \int_{\Gamma(t)} v \cdot \varphi \; do,
	\quad \forall \varphi \in L^2(\Gamma(t), \mathbb{R}^n),
	\\
	&\int_{\partial \Gamma(t)} \partial^\bullet u \cdot \eta
	+ \tfrac{1}{\alpha} \nabla_{\Gamma(t)} u 
		\hat{H}^{-1} (\nabla_{\Gamma(t)} \hat{y})^T \zeta \cdot \eta \; do
	= \int_{\partial \Gamma(t)} v \cdot \eta \; do,
	\quad \forall \eta \in L^2(\partial \Gamma(t), \mathbb{R}^n),
\end{aligned}	
\right\}
	\label{weak_formulation_u_surface_gradients}
	\\
	&\int_{\Gamma(t)} \zeta \cdot \phi \; do
	+ \int_{\Gamma(t)} \nabla_{\Gamma(t)} \hat{y} : \nabla_{\Gamma(t)} \phi \; do = 0,
	\quad \forall \phi \in \mathcal{S},
	\label{weak_formulation_zeta_surface_gradients}
\end{align}
where $\nabla_{\Gamma(t)} \hat{y} : \nabla_{\Gamma(t)} \phi := \sum_{j=1}^k \nabla_{\Gamma(t)} \hat{y}^j \cdot \nabla_{\Gamma(t)} \phi^j$.
\end{thm}
\begin{remark}
We just observe that $\nabla_{\Gamma(t)} u \hat{H}^{-1} (\nabla_{\Gamma(t)} \hat{y})^T = \hat{H}^{-1} (\nabla_{\Gamma(t)} \hat{y})^T$, see also
Remark \ref{Remark_more_implicit}.
\end{remark}

\section{Numerical schemes for the DeTurck reparametrization}
\label{Discretization_section}
\subsection{Finite element surface}
\label{Section_finite_elements}
We now assume that the reference manifold $\M$ is approximated by a piecewise linear, polyhedral
manifold 
$$
	\M_h = \bigcup_{S \in \mathcal{T}(\M_h)} S \subset \mathbb{R}^k,
$$
where $\mathcal{T}(\M_h)$ is an admissible triangulation consisting of $(n-d)$-dimensional,
non-degenerated simplices $S$ in $\mathbb{R}^k$. 
The finite element space $V_h(\M_h)$ is the set of continuous,
piecewise linear functions on $\M_h$, that is
$$
	V_h(\M_h) := \left\{ \varphi_h \in C^0(\M_h) \; | \; 
		\varphi_{h|S} \; \textnormal{is a linear polynomial for all $S \in \mathcal{T}(\M_h)$} \right\}.
$$
For the time discretization, we introduce the notation $f^m := f(\cdot, m \tau)$
for the discrete time levels $\{ m \tau \; | \; m= 0, \ldots, M_\tau \}$
with time step size $\tau > 0$ and $M_\tau \tau < T$. In the following, we try to find
approximations 
$$
	\Gamma^m_h = \bigcup_{S^m_\Gamma \in \mathcal{T}(\Gamma^m_h)} S^m_\Gamma \subset \mathbb{R}^n
$$
of the submanifolds $\Gamma^m$ with $\Gamma_h^m = \hat{x}_h^m(\M_h)$ for some 
$\hat{x}^m_h \in V_h(\M_h)^{n}$. The map $\hat{x}^m_h$ is supposed to be a homeomorphism
of $\M_h$ onto $\Gamma^m_h$. Note that $S^m_\Gamma = \hat{x}_h^m(S)$ for some $S \in \mathcal{T}(\M_h)$.
The finite element spaces $V_h(\Gamma^m_h)$ and $V_h(\partial \Gamma^m_h)$
are the set of continuous, piecewise linear functions on $\Gamma_h^m$ and respectively, on $\partial \Gamma_h^m$. 
Furthermore, we define the following subspaces 
$$
\overset{\circ}{V}_h(\Gamma_h^m) := \left\{ \eta_h \in V_h(\Gamma_h^m) \; | \;
		\eta_h = 0 \; \textnormal{on $\partial \Gamma_h^m$} \right\}. 
$$
The inverse $\hat{y}^m_h := (\hat{x}_h^m)^{-1}$ is in $V_h(\Gamma_h^m)^k$.
We assume that $\lambda_h: \partial \M_h \rightarrow \mathbb{R}^k$
is an approximation of the unit co-normal $\lambda$ on $\partial \M$ which is 
piecewise constant on each $(n-d-1)$-dimensional boundary simplex of $\partial \M_h$. 
The finite element space $\mathcal{S}_h(\Gamma_h^m)$ is defined by
$$
	\mathcal{S}_{h}(\Gamma^m_h)
	= \left\{ \phi_h \in V_h(\Gamma^m_h)^k \; | \; (\lambda_h \circ \hat{y}^m_h) \cdot \phi_h = 0
	\; \textnormal{on $\partial \Gamma^m_h$} \right\}.
$$
Since the $(n-d)$-dimensional 
simplices of $\Gamma_h^m$ are affine to the standard simplex in $\mathbb{R}^{n-d}$,
the only remnant of the embedding is that the vertices of $\Gamma_h^m$ 
have position vectors in $\mathbb{R}^n$. 
As a result, standard finite element definitions, such as 
the definition of the linear Lagrange interpolation $I_h$, 
can be easily carried over to the submanifold case.
The tangential gradient on $\Gamma_h^m$ is defined piecewise 
on each simplex $S^m_\Gamma \in \mathcal{T}(\Gamma_h^m)$ like in (\ref{definition_tangential_gradient}).

We choose the time step size $\tau = C ~ h_{min}^2$, where $h_{min} := \min_{S^m_\Gamma \in \mathcal{T}(\Gamma_h^m)} h(S^m_\Gamma)$
is the minimal diameter of all simplices $S^m_\Gamma \subset \Gamma_h^m$ and $C > 0$ is some positive constant.
In simulations, an optimal constant $C$ can, for example, be determined for a relatively coarse mesh and then used on a finer mesh. 
Using this time step size, the algorithm proposed below turned out to be numerically stable in all of our experiments.

The main purpose of this work is to control the quality of the mesh as defined by the following measure of mesh quality
\begin{equation}
	\sigma_{max} := \max_{S^m_\Gamma \in \mathcal{T}(\Gamma^m_h)} \frac{h(S^m_\Gamma)}{\rho(S^m_\Gamma)},
	\label{definition_sigma_max}
\end{equation} 
Here, $h(S^m_\Gamma)$ denotes the diameter of $S^m_\Gamma$ and $\rho(S^m_\Gamma)$ is the radius of the largest ball contained in $S^m_\Gamma$.
For vanishing velocity $v=0$ and time steps $m \nearrow \infty$, we expect that 
\begin{equation*}
\frac{h(S^m_\Gamma)}{\rho(S^m_\Gamma)} \approx \frac{h(\hat{y}^m_h(S^m_\Gamma))}{\rho(\hat{y}^m_h(S^m_\Gamma))}
	\quad \textnormal{for all $S^m_\Gamma \in \mathcal{T}(\Gamma_h^m)$.}
\end{equation*}

\subsection{The discrete problems}
\label{Section_discrete_problems}

\subsubsection{Fixed reference triangulation}
A natural way to define the sequence of discrete embeddings $\hat{x}^{m+1}_h$ (which is not needed in the following scheme) 
would be $\hat{x}^{m+1}_h := u^{m+1}_h \circ \hat{x}_h^m$, where $u^{m+1}_h: \Gamma_h^m \rightarrow \Gamma_h^{m+1}$ is 
an appropriate approximation to $u(t)$ on $\Gamma(t)$, see below.
Since $\hat{y}^{m+1}_h := (\hat{x}^{m+1}_h)^{-1}$, this would imply that $\hat{y}^{m+1}_h = (\hat{x}^{m}_h)^{-1} \circ (u^{m+1}_h)^{-1}$,
and therefore, $\hat{y}^{m+1}_h = \hat{y}^{m}_h \circ (u^{m+1}_h)^{-1}$.
An important consequence of this observation is that we can totally get rid of the map $\hat{x}_h^m$ for all time steps $m \geq 1$,
when we use the last identity as the definition of $\hat{y}^{m+1}_h$.
We choose the time discretization to linearize the problem in each time step and propose the following algorithm for the computation of the system
(\ref{weak_formulation_u_surface_gradients}) and (\ref{weak_formulation_zeta_surface_gradients}).
\begin{alg}
\label{algo_DeTurck}
Let $\alpha \in (0, \infty)$. For a given $(n-d)$-dimensional submanifold 
$\Gamma_h^0 = \hat{x}_h^0(\M_h) \subset \mathbb{R}^{n}$ with $\hat{x}^0_h \in V_h(\M_h)^{n}$, set 
$\hat{y}_h^0 := (\hat{x}_h^0)^{-1} 
\in V_h(\Gamma_h^0)^{k}$. For the discrete time levels $m=0, \ldots, M_\tau -1$ do
\begin{enumerate}
\item[(i)] Compute the solution $\zeta_h^m \in \mathcal{S}_h({\Gamma}_h^m)$ of
\begin{equation}
	\int_{{\Gamma}^m_h} \zeta_h^m \cdot \phi_h \; do
	+ \int_{{\Gamma}^m_h} \nabla_{{\Gamma}^m_h} \hat{y}^m_h : \nabla_{{\Gamma}^m_h} \phi_h \; do 
	= 0, \quad \forall \phi_h \in \mathcal{S}_h(\Gamma^m_h)
	\label{equation_for_zeta_h}
\end{equation}
\item[(ii)] Then determine the solution
$u^{m+1}_{h} \in V_{h}({\Gamma}_h^m)^{n}$ of
\begin{equation}
\left.
\begin{aligned}
	& 
	\int_{\Gamma_h^m} \tfrac{1}{\tau} I_h (u^{m+1}_{h} \cdot \varphi_h)  
		 + \tfrac{1}{\alpha } \sum_{\kappa = 1}^n \sum_{\sigma = 1}^k
		 ((\hat{H}^m_h)^{-1} \nabla_{{\Gamma}_h^m} \hat{y}_{h, \sigma}^{m})_{\kappa} 
		 I_h \big({\tilde{\zeta}}^{m,\sigma}_h \varphi_h^\kappa \big)\; do
	\\	 
	& \qquad = \int_{\Gamma_h^m} I_h(v^{m} \cdot \varphi_h) + \tfrac{1}{\tau} I_h (\tilde{u}^{m}_h \cdot \varphi_h) \; do,
	\quad \forall \varphi_h \in \overset{\circ}{V}_h({\Gamma}_h^m)^{n}
	\\	
	& 
	\int_{\partial \Gamma_h^m} \tfrac{1}{\tau} I_h (u^{m+1}_{h} \cdot \eta_h)  
		 + \tfrac{1}{\alpha } \sum_{\kappa = 1}^n \sum_{\sigma}^k 
		 ((\hat{H}^m_h)^{-1} \nabla_{{\Gamma}_h^m} \hat{y}_{h, \sigma}^{m})_{\kappa} 
		 I_h \big({\tilde{\zeta}}^{m,\sigma}_h (T_h^m \eta_h)^\kappa \big)\; do
	\\	 
	& \qquad = \int_{\partial \Gamma_h^m} I_h(v^{m} \cdot \eta_h) + \tfrac{1}{\tau} I_h (\tilde{u}^{m}_h \cdot \eta_h) \; do,
	\quad \forall \eta_h \in V_h({\partial \Gamma}_h^m)^{n}
\end{aligned}
\right\}
\label{equation_for_u_h}
\end{equation}
Here, $\tilde{u}_h^m = id_{\Gamma_h^m}$ is the identity map on $\Gamma_h^m$. The map $\hat{H}^m_h$ on ${\Gamma}_h^m$ is defined by
$$
	\hat{H}^m_h := (\nabla_{{\Gamma}_h^m} \hat{y}^m_h)^T \nabla_{{\Gamma}_h^m} \hat{y}^m_h
	+ \unit - P^m_h,
$$
where $(P_h^m)_{|S^m_\Gamma}$ is the (constant) tangential projection onto the tangent space of 
the simplex $S^m_\Gamma \subset \Gamma_h^m$. The value of ${\tilde{\zeta}}^m_h$ in the vertex $p_j$ of $\Gamma_h^m$ is defined by
$${\tilde{\zeta}}^m_h(p_j):= P_{\M}(\hat{y}^m_h(p_j)) \zeta^m_h(p_j),$$ 
where $P_{\M}(\hat{y}^m_h(p_j))$ is the tangential projection onto the tangent space of the reference manifold $\M$ in the point $\hat{y}^m_h(p_j)$.
The projection $T_h^m(p_j)$ onto the tangent space of the discrete boundary $\partial \Gamma_h^m$ is defined by
$$
	T_h^m(p_j) := 
	\left\{
	\begin{aligned}
		\frac{\sum_{S} \vec{\tau}_h^m(S) }{|\sum_{S} \vec{\tau}_h^m(S)|} \otimes \frac{\sum_{S} \vec{\tau}_h^m(S) }{|\sum_{S} \vec{\tau}_h^m(S)|}, 
		\quad \textnormal{if $n-d = 2$,}
	\\	
		\frac{\sum_{S} T_h^m(S) |S|}{\sum_{S} |S|},		
	    \quad \textnormal{if $n-d > 2$,}
	\end{aligned}
	\right. 
$$
for all vertices $p_j \in \partial \Gamma_h^m$. Here, 
the sum is over all $n-d-1$-dimensional boundary simplices $S \subset \partial \Gamma_h^m$ adjacent to the boundary vertex $p_j$. 
$\vec{\tau}_h^m(S)$ is a unit tangent vector to the boundary simplex $S$, where all tangent vectors in the above 
sum are chosen such that $\vec{\tau}_h^m(S) \cdot \vec{\tau}_h^m(S') \geq 0$ for two different boundary simplices $S$ and $S'$
belonging to $p_j$. The map $T_h^m(S)$ is the projection onto the tangent space of the boundary simplex $S$.
\item[(iii)]
The discrete submanifold $\Gamma_h^{m+1} \subset \mathbb{R}^n$ is then defined by
$$
	\Gamma_h^{m+1} := u_h^{m+1}(\Gamma_h^m),
$$
and finally, we set 
$$
	\hat{y}^{m+1}_h := \hat{y}^m_h \circ (u_h^{m+1})^{-1}. 
$$
\end{enumerate}
\end{alg}
We introduced the projection $T_h^m$ onto the tangent space of the discrete boundary $\partial \Gamma_h^m$ for stability reasons.
\begin{remark}
\label{Remark_lift_onto_the_next_discrete_surface} 
An important feature of our scheme is that it is, in fact, not necessary to compute the inverse of $u_h^{m+1}$.
This can be seen as follows: It turns out that the components
of the map $\hat{y}^{m+1}_h$ with respect to the Lagrange finite element basis on $\Gamma_h^{m+1}$ 
are the same as the components
of $\hat{y}^{m}_h$ with respect to the corresponding basis on $\Gamma_h^{m}$. 
To be more precise: The components of $\hat{y}_h^{m}$ with respect to the Lagrange basis on $\Gamma_h^m$
are given by the position vectors of the mesh vertices of $\M_h$, which are constant. 
Therefore, $\hat{y}^m_h$ is described by a component vector which is independent of $m$.
However, note that the map $\hat{y}_h^m$ itself changes in time, since the finite element basis changes when $\Gamma_h^{m}$ is updated.
\end{remark}
\begin{remark}
\label{Remark_more_implicit}
The linear system (\ref{equation_for_u_h}) could be made more implicit by replacing the term
$$
	(\hat{H}^m_h)^{-1} \nabla_{\Gamma_h^m} \hat{y}_{h, \sigma}^m)_\kappa 
		I_h \big( {\tilde{\zeta}}^{m,\sigma}_h \varphi_h^\kappa \big)
$$
by the term
$$
	(\nabla_{\Gamma_h^m} u^{m+1}_h 
		(\hat{H}^m_h)^{-1} \nabla_{\Gamma_h^m} \hat{y}_{h, \sigma}^m)_\kappa 
		I_h \big( {\tilde{\zeta}}^{m,\sigma}_h \varphi_h^\kappa \big).
$$ 
\end{remark}
\begin{remark}
\label{remark_regularizing_term}
In order to be able to choose larger time steps $\tau$
in the above algorithm, one could add a regularizing term to equation (\ref{equation_for_zeta_h}),
that is: Find $\zeta_h^m \in \mathcal{S}_h({\Gamma}_h^m)$ such that
$$
	\int_{{\Gamma}^m_h} \zeta_h^m \cdot \phi_h + \varepsilon \nabla_{{\Gamma}^m_h} \zeta^m_h : \nabla_{{\Gamma}^m_h} \phi_h \; do
	+ \int_{{\Gamma}^m_h} \nabla_{{\Gamma}^m_h} \hat{y}^m_h : \nabla_{{\Gamma}^m_h} \phi_h \; do 
	= 0, \quad \forall \phi_h \in \mathcal{S}_h(\Gamma^m_h),
$$
where $\varepsilon > 0$ must be chosen sufficiently small to ensure that the redistribution of the mesh points still works.
A similar idea was used for the approximation of the Ricci curvature in \cite{Fr13}
and for the approximation of the mean curvature vector in \cite{He09}.
\end{remark}

As demonstrated in the next section, the above algorithm is able to produce good meshes
for $\Gamma_h^m$, that is meshes with relatively small values of $\sigma_{max}$, provided that the parameter $\alpha$
is chosen sufficiently small and that the quality of the mesh $\M_h$ is sufficiently good (that is 
the value of the quantity $\sigma_{max}$ has to be relatively small for the reference mesh $\M_h$)

\subsubsection{Refinement and coarsening of the reference triangulation}
The redistribution of the mesh points induced by the DeTurck trick also leads to 
simplices $S^m_\Gamma \subset \Gamma_h^m$, which differ strongly with respect to their volume (area) 
after a certain number of time steps. 
The following algorithm complement the above scheme with a refinement and coarsening strategy,
which keeps the volume (area) of the simplices approximately uniform.
\begin{alg}{(Mesh refinement and coarsening strategy)}
\label{algo_refinement_and_coarsening_strategy}
Define $A^m_{target} = |\Gamma_h^m|/N(\mathcal{T}(\Gamma^0_h))$, where $N(\mathcal{T}(\Gamma^0_h))$ denotes
the number of simplices in $\mathcal{T}(\Gamma_h^0)$. Choose $T_{adapt} \in [\tau, T)$.
If $m \tau < r T_{adapt} \leq (m+1) \tau$ for some $r \in \mathbb{N}_0$,
we mark the simplices $S^{m+1}_\Gamma \in \mathcal{T}(\Gamma_h^{m+1})$ as follows:
\begin{itemize}
\item[--] Mark $S^{m+1}_\Gamma$ for one refinement step if $|S^{m+1}_\Gamma| > 2 A^{m+1}_{target}$.
\item[--] Mark $S^{m+1}_\Gamma$ for one coarsening step if $|S^{m+1}_\Gamma| < A^{m+1}_{target}/ 2$.
\end{itemize} 
Then all simplices marked for refinement are bisected once if they have 
a compatible neighbour also marked for refinement.
Otherwise, a recursive refinement of adjacent elements with an incompatible refinement edge is applied.
After the refinement procedure 
all simplices marked for coarsening are coarsened if all neighbour elements which would be affected
by the coarsening are also marked for coarsening; see \cite{SS05} for a detailed describtion of
the ALBERTA refinement and coarsening routines. 
(Two simplices marked for coarsening are compatible if they were produced by bisection
from a common "parent"-simplex. 
Coarsening is therefore the inverse of refinement in the sense that 
two compatible simplices marked for coarsening 
are replaced by their parent. In particular, 
the vertex, which was created during the refinement, is again deleted in the coarsening step.) 
While vertices are deleted in the coarsening step, new vertices are produced during refinement.
In the interior the coordinates of the new vertices are just given by
the midpoints of the corresponding refinement edges; see \cite{SS05} for details.
For the vertices at the boundary we apply a geometrically consistent mesh modification scheme; see \cite{BNP10}:
Before the refinement the discrete mean curvature vector 
$\vec{\kappa}_h^m \in V_h(\partial \Gamma^m_h)^n$ of the boundary $\partial \Gamma^m_h$ 
is determined by the equation
\begin{align*}
	\int_{\partial \Gamma_h^m} I_h (\vec{\kappa}_h^m \cdot \varphi_h) \; do 
	= \int_{\partial \Gamma_h^m} \nabla_{\partial \Gamma^m_h} id_{\partial \Gamma^m_h} 
	: \nabla_{\partial \Gamma^m_h} \varphi_h \; do,
	\quad \forall \varphi_h \in V_h({\partial \Gamma}_h^m)^{n}.
\end{align*}
During mesh refinement this vector is interpolated linearly. The coordinates of the old and new 
vertices $p_j$
at the boundary of $\partial \Gamma^m_h$ are given by the old coordinates and the coordinates
of the midpoints of the refinement edges, respectively. The map 
$u_h^m \in V_h(\partial \Gamma_h^m)$ is then determined by
\begin{align*}
	\int_{\partial \Gamma_h^m} \nabla_{\partial \Gamma^m_h} u_h^{m}
	: \nabla_{\partial \Gamma^m_h} \varphi_h \; do
	= \int_{\partial \Gamma_h^m} I_h (\vec{\kappa}_h^m \cdot \varphi_h) \; do,
	\quad \forall \varphi_h \in V_h({\partial \Gamma}_h^m)^{n},
\end{align*}
and $\int_{\partial \Gamma_h^m} u^m_h do = \int_{\partial \Gamma_h^m} id_{| \partial \Gamma_h^m} do$.
Finally, the new coordinates of the boundary vertices $p_j$ are set to be $u_h^m(p_j)$.
During the refinement step the values of $\hat{y}_h^m$ at the new vertices are
determined by Lagrange interpolation. (In the numerical examples
in Section \ref{Section_numerical_examples}, we also projected them onto the reference manifold $\M$
by rescaling them.) In the coarsening step, the corresponding values of $\hat{y}_h^m$
are just deleted.

\end{alg}
While Algorithm \ref{algo_DeTurck} should lead to meshes with small $\sigma_{max}$,
that is to meshes without any sharp simplices (triangles), 
Algorithm \ref{algo_refinement_and_coarsening_strategy}
ensures that the simplices of $\Gamma_h^m$ have similar volume (area).
The impact of the refinement and coarsening procedure on the mesh quality
is thereby almost negligible.
Of course, the refinement and coarsening strategy of Algorithm \ref{algo_refinement_and_coarsening_strategy} can be replaced by other strategies
without affecting the DeTurck reparametrization in Algorithm \ref{algo_DeTurck}. 
For example, a refinement and coarsening strategy might take the curvature of the boundary 
of $\Gamma_h^m$ into account, or the solution of a PDE solved on $\Gamma_h^m$.

\section{Numerical results}
\label{section_numerical_results}

\subsection{Implementation}
\label{Section_implementation}
\subsubsection*{The reference manifold}
We consider $n=3$ and $d=1$ in these numerical examples with two reference manifolds.

\noindent {\bf Case 1} For the case of a simply-connected domain, the reference manifold $(\M,m)$ is chosen to be the two-dimensional half-sphere 
$\mathbb{H}^2 \subset \mathbb{R}^3$ defined in (\ref{half_sphere_definition}) 
with metric $m$ induced by the Euclidean metric.
A unit co-normal vector field $\lambda$ to $\partial \mathbb{H}^2$ with respect to $m$
is given by the constant vector field $\lambda = (1, 0, 0 )^T$. 
We therefore choose $\lambda_h := \lambda$. The finite element space
$\mathcal{S}_h(\Gamma_h^m)$ is then given by
\begin{equation}
	\mathcal{S}_h(\Gamma_h^m) = 
	\left\{ \phi_h \in V_h(\Gamma^m_h)^3 \; | \; \phi^1_h = 0
	\; \textnormal{on $\partial \Gamma^m_h$} \right\}.
	\label{FEM_space_half_sphere_case}
\end{equation}
An approximation $\M_h$ of $\mathbb{H}^2$ was produced in our experiments by the global refinement of a half-octahedron, where in each
refinement step the new vertices were projected onto the half-sphere by rescaling their position 
vector to unit length.

\noindent {\bf Case 2} In order to handle a domain with a hole, it is convenient to have two boundaries for $\M$. We choose
the reference manifold $\M$ to be the cylinder
\begin{align}
	\mathcal{C} = \{ x \in \mathbb{R}^3 \; | \; -1 \leq x_1 \leq 1 \; \textnormal{and} \; x_2^2 + x_3^2 = 1 \},
	\label{definition_cylinder}
\end{align}
with metric $m$ induced by the Euclidean metric; see Figure \ref{moving_inclusion_reference_manifold_at_time_1_0}. 
The cylinder $\mathcal{C}$ has totally geodesic boundary. A unit co-normal vector field $\lambda$ to $\partial \mathcal{C}$ with respect to $m$
is given by $\lambda = (\pm 1, 0, 0)^T$. We hence choose $\lambda_h := \lambda$. We can then use
the finite element space $\mathcal{S}_h(\Gamma_h^m)$ defined as in (\ref{FEM_space_half_sphere_case}).

\subsubsection*{Linear algebra}
In order to compute the solutions $\zeta_h^m \in \mathcal{S}_h(\Gamma_h^m)$ and $u_{h}^{m+1} \in V_h(\Gamma_h^m)^3$
of steps $(i)$ and $(ii)$ of Algorithm \ref{algo_DeTurck},
the following linear systems of equations have to be solved.
For the vector $\mathbf{Z} = (\mathbf{Z}^{j\sigma})$ the system
\begin{equation}
	\widetilde{\mathbf{M}}_{ij\kappa \sigma} \mathbf{Z}^{j \sigma} = \mathbf{R}_{i\kappa},
	\; \forall i, \kappa,
	\label{equation_for_Z}
\end{equation}
where $\mathbf{R}_{i1} = 0$ for all $i$ with vertex $p_i \in \partial \Gamma_h^m$ and
$\mathbf{R}_{i\kappa} = - \mathbf{S}_{ij\kappa \sigma} \mathbf{Y}^{j \sigma}$ else; 
and for the vector $\mathbf{U}= (\mathbf{U}^{j\sigma})$ the system
\begin{equation}
	 \mathbf{M}_{ij\kappa \sigma} \mathbf{U}^{j \sigma} 
    = \mathbf{M}_{ij\kappa \sigma} ( \mathbf{U}^{j \sigma}_{old} + \tau \mathbf{V}^{j \sigma} ) - \tfrac{\tau}{\alpha} \mathbf{D}_{ij\kappa \sigma} \widetilde{\mathbf{Z}}^{j \sigma}, 
    \; \forall i, \kappa.
    \label{equation_for_U}
\end{equation} 
Here, we have made use of the representations
$\hat{y}^m_h = \sum_{j,\sigma} \mathbf{Y}^{j\sigma} \phi_j \vec{b}_\sigma$,
$\zeta^m_h = \sum_{j,\sigma} \mathbf{Z}^{j\sigma} \phi_j \vec{b}_\sigma$,
${\tilde{\zeta}}^m_h = \sum_{j,\sigma} \widetilde{\mathbf{Z}}^{j\sigma} \phi_j \vec{b}_\sigma$,
$u^{m+1}_{h} = \sum_{j,\sigma} \mathbf{U}^{j\sigma} \phi_j \vec{b}_\sigma$,
$\tilde{u}^{m}_{h} = \sum_{j,\sigma} \mathbf{U}^{j\sigma}_{old} \phi_j \vec{b}_\sigma$,
and $I_h v^{m} = \sum_{j,\sigma} \mathbf{V}^{j\sigma} \phi_j \vec{b}_\sigma$
where $\vec{b}_\sigma = (\delta_{1\sigma}, \delta_{2\sigma}, \delta_{3\sigma})^T \in \mathbb{R}^3$,
and $\phi_j$ denotes the piecewise linear Lagrange basis function
associated with the mesh vertex $p_j \in \Gamma^m_h$.
The matrices $\mathbf{M}$, $\mathbf{S}$ and $\mathbf{D}$
are computed by assembling the following element matrices
\begin{align*}
	& \mathbf{M}_{ij\kappa\sigma}(S^m_\Gamma) = 
\left\{
\begin{aligned}	
	& \delta_{\kappa \sigma} \delta_{ij} \int_{\overline{S^m_\Gamma} \cap \partial \Gamma_h^m} \phi_i \; do,
	\; &\textnormal{if $p_i \in \partial \Gamma_h^m$},
	\\
	& \delta_{\kappa \sigma} \delta_{ij} \int_{S^m_\Gamma} \phi_i \; do,
	&\textnormal{else,}
\end{aligned}	
\right.		
	\\
	& \mathbf{D}_{ij\kappa\sigma}(S^m_\Gamma) = 
\left\{
\begin{aligned}	
	&(T_h^m(p_i) (\hat{H}^m_h(S^m_\Gamma))^{-1} \nabla_{\Gamma_h^m} \hat{y}_{h,\sigma}^m(S^m_\Gamma))_\kappa
	\delta_{ij} \int_{\overline{S^m_\Gamma} \cap \partial \Gamma_h^m} \phi_i \; do,
	\; &\textnormal{if $p_i \in \partial \Gamma_h^m$},
	\\
	&( (\hat{H}^m_h(S^m_\Gamma))^{-1} \nabla_{\Gamma_h^m} \hat{y}_{h,\sigma}^m(S^m_\Gamma))_\kappa
	\delta_{ij} \int_{S^m_\Gamma} \phi_i do,
	\; &\textnormal{else,}
\end{aligned}	
\right.	
	\\
	& \mathbf{S}_{ij\kappa\sigma}(S^m_\Gamma) = \delta_{\kappa \sigma} \int_{S^m_\Gamma} \nabla_{\Gamma_h^m} \phi_i \cdot \nabla_{\Gamma_h^m} \phi_j \; do,
\end{align*}
for all $S^m_\Gamma \in \mathcal{T}(\Gamma_h^m)$. 
Here, by abuse of notation, $do$ denotes the two-dimensional, and respectively, one-dimensional Hausdorff measure on $S^m_\Gamma$ and respectively, 
on $\overline{S^m_\Gamma} \cap \partial \Gamma^m_h$.
Note that the matrices $\hat{H}^m_h(S^m_\Gamma)$ and $\nabla_{\Gamma_h^m} \hat{y}_h^m(S^m_\Gamma)$ are constant on each simplex $S^m_\Gamma$. 
Furthermore, we define $\widetilde{\mathbf{M}}_{ij11}(S^m_\Gamma) := \delta_{ij}$ for all $i,j$ with vertex $p_i \in \partial \Gamma_h^m$
or $p_j \in \partial \Gamma_h^m$,
and $\widetilde{\mathbf{M}}_{ij\kappa \sigma}(S^m_\Gamma) :=  \delta_{\kappa \sigma} \int_{S^m_\Gamma} \phi_i \phi_j \; do$ else.
A novel feature of our scheme is that the vector $\mathbf{Y}$ does not depend on the time step $m$.
In order to see this, we just observe that by definition of $\hat{y}_h^{m+1} := \hat{y}^m_h \circ (u_h^{m+1})^{-1}$,
the values of $\hat{y}_h^{m+1}$ in the vertices of $\Gamma_h^{m+1}$ are equal to the values of $\hat{y}_h^m$ in the vertices of $\Gamma_h^m$
(if the mesh is not refined or coarsened). 
In fact, $(\mathbf{Y}^{j\sigma})_{\sigma=1, \ldots,3} \in \mathbb{R}^3$ is given by the position vector of the mesh vertex $\hat{y}_h^{m+1}(p_j) \in \M_h$.

\subsubsection*{Finite element toolbox}
The following numerical experiments were performed within the Finite Element Toolbox ALBERTA,
see \cite{SS05}. In our numerical examples
all two-dimensional submanifolds $\Gamma_h^m$, including those that are flat, were treated as
hypersurfaces in $\mathbb{R}^3$. This is just due to the design of ALBERTA. 
Since $\mathbf{M}$ is a diagonal matrix, it is trivial to solve (\ref{equation_for_U}).
The system (\ref{equation_for_Z}) can be solved by the conjugate gradient method.
We produced our images in ParaView.

\subsection{Numerical examples}
\label{Section_numerical_examples}
In the examples below, we have compared two different approaches for evolving the mesh of
a moving domain or surface with respect to their impact on the mesh quality. We will show that the
method based on Algorithms \ref{algo_DeTurck} and \ref{algo_refinement_and_coarsening_strategy}
is superior to the approach when the mesh vertices are just moved with the original (surface) velocity.
However, note that the refinement and coarsening strategy, that is Algorithm 
\ref{algo_refinement_and_coarsening_strategy}, is used in both approaches.

\subsection*{Example 1: Improving the mesh quality for stationary domains}
In the first example, we demonstrate how Algorithm \ref{algo_DeTurck} can be used to improve the
computational mesh of a given domain. We first apply a stereographic projection to the approximation $\M_h$ of the 
half-sphere. This leads to a conformal triangulation of the unit disk.
The inverse of the stereographic projection defines the vector field $\hat{y}^0_h$.
In order to produce a bad mesh, that is a mesh with a relatively large value of $\sigma_{max}$,
we deform the unit disk into the shape shown in Figure \ref{Ex1_shape_at_time_0_0}
in the time interval $[-0.02,0.0)$. At time $t=0.0$ this deformation is stopped.
In the time interval $[0.0, 0.2]$, we then apply Algorithm \ref{algo_DeTurck} together
with the refinement and coarsening strategy in Algorithm \ref{algo_refinement_and_coarsening_strategy} 
in order to improve the mesh again.
The computational parameters in this experiment are $\tau = 0.005 ~ h_{min}^2$, $\alpha = 1.0$ and $T_{adapt} = 10^{-3}$,
where $h_{min} = \min_{S^m_\Gamma \in \mathcal{T}(\Gamma_h^m)} h(S^m_\Gamma)$ is the minimal diameter of all simplices $S^m_\Gamma$ of $\Gamma_h^m$. Figures \ref{Ex1_shape_with_mesh_at_time_0_0}
and \ref{Ex1_shape_with_mesh_at_time_0_2} demonstrate the improvement of the mesh quality.
Figure \ref{Ex1_mesh_quality} shows  
the quantitative behaviour of the parameter $\sigma_{max}$. 
In Figure \ref{Ex1_shape_with_mesh_at_time_0_2}, we also see that the remeshing method preserves the shape of the original domain (red area).
The result of Algorithm \ref{algo_refinement_and_coarsening_strategy}
is that the area of all triangles is of the same order. The refinement and coarsening
of the mesh can decrease the mesh quality; see the jumps in the image on the right hand side
of Figure \ref{Ex1_mesh_quality}. However, since the refinement and coarsening process is based on 
bisection, and respectively, on its inverse procedure, the mesh quality is only changed locally by Algorithm \ref{algo_refinement_and_coarsening_strategy}. 
Such changes can be improved very quickly by Algorithm \ref{algo_DeTurck}.
Algorithm \ref{algo_refinement_and_coarsening_strategy} also leads to a local refinement and coarsening
of the reference manifold $\M_h$. This is shown in Figures \ref{Ex1_reference_manifold_at_time_0_0}
and \ref{Ex1_reference_manifold_at_time_0_2}.  

\begin{figure}
\begin{center}
\subfloat[][\centering Shape of the computational domain at time $t=0.0$.]
{\includegraphics[angle=90, width=0.6\textwidth]{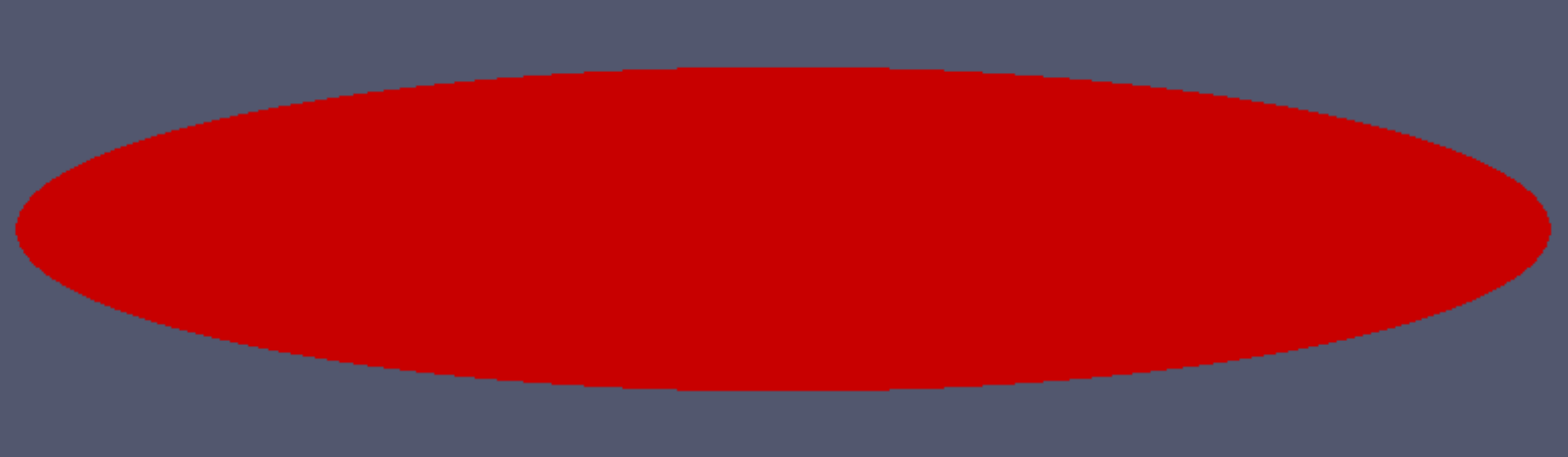}
\label{Ex1_shape_at_time_0_0}
}
\\
\subfloat[][\centering Computational mesh at time $t=0.0$.]
{\includegraphics[angle=90, width=0.4\textwidth]{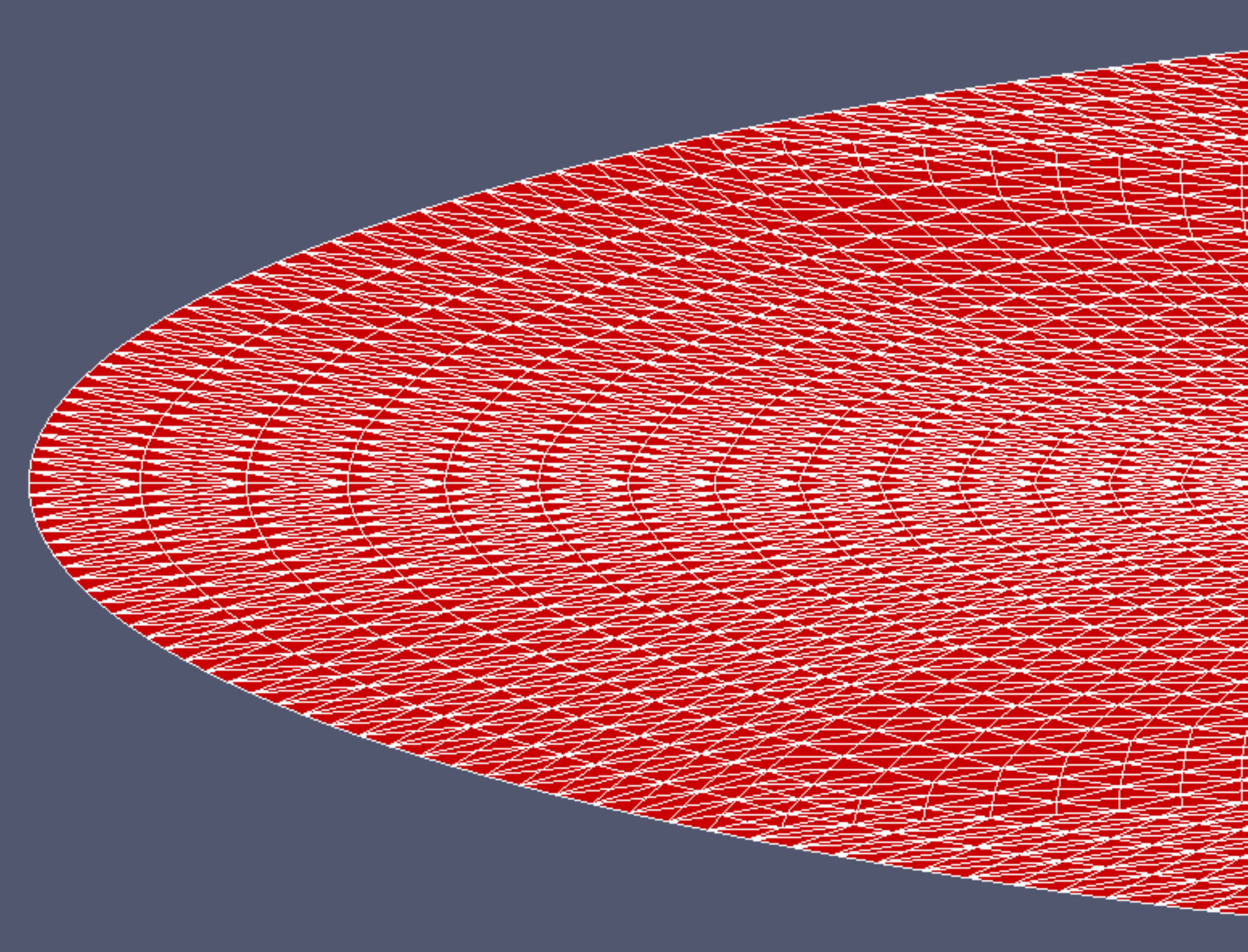}
\label{Ex1_shape_with_mesh_at_time_0_0}
} 
~~~
\subfloat[][\centering Computational mesh at time $t=0.2$.]
{\includegraphics[angle=90, width=0.4\textwidth]{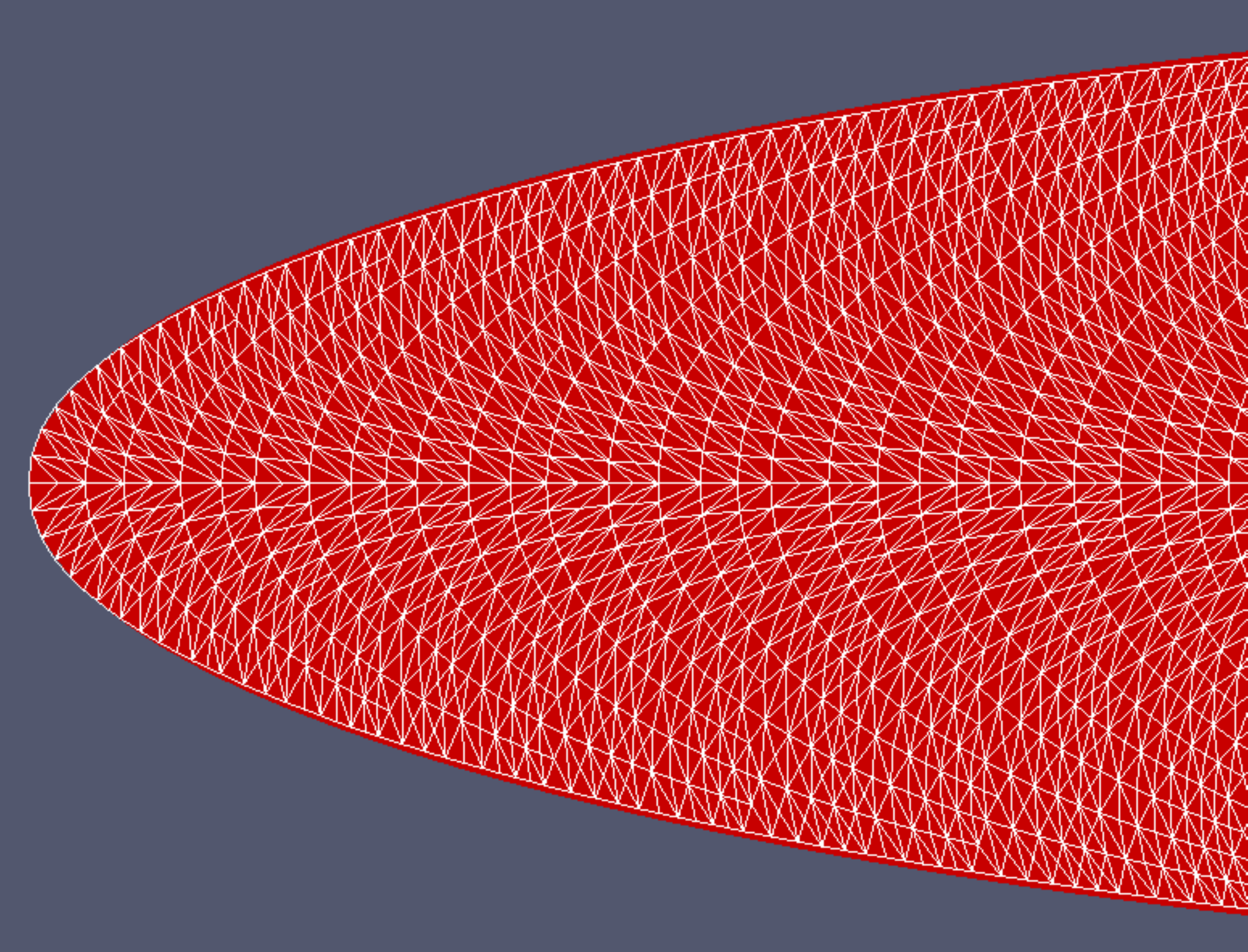}
\label{Ex1_shape_with_mesh_at_time_0_2}
}
\\
\subfloat[][\centering Reference manifold at the beginning of the simulation.]
{\includegraphics[width=0.4\textwidth]{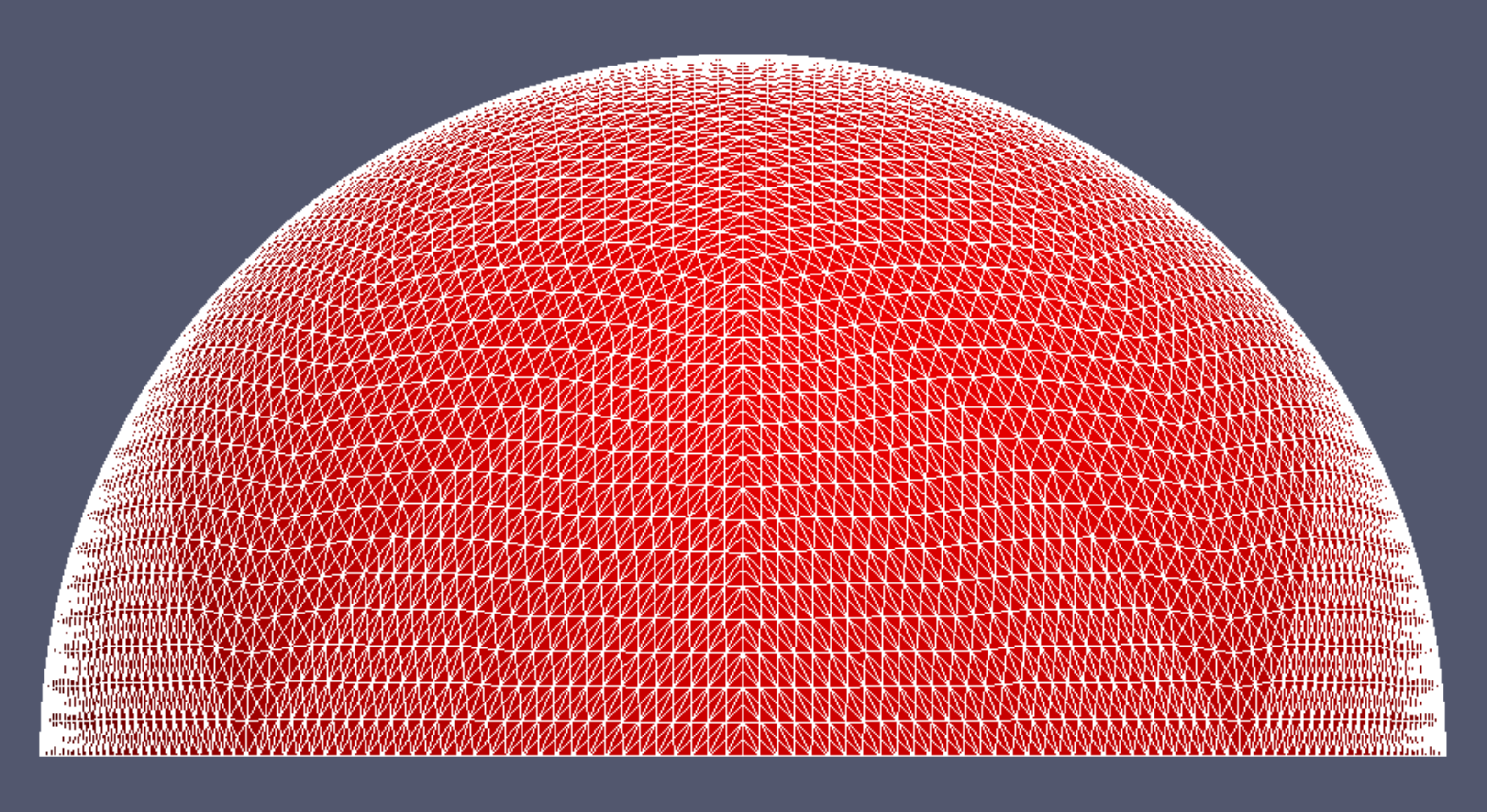}
\label{Ex1_reference_manifold_at_time_0_0}
} 
~~~
\subfloat[][\centering Reference manifold at the end of the simulation.]
{\includegraphics[width=0.4\textwidth]{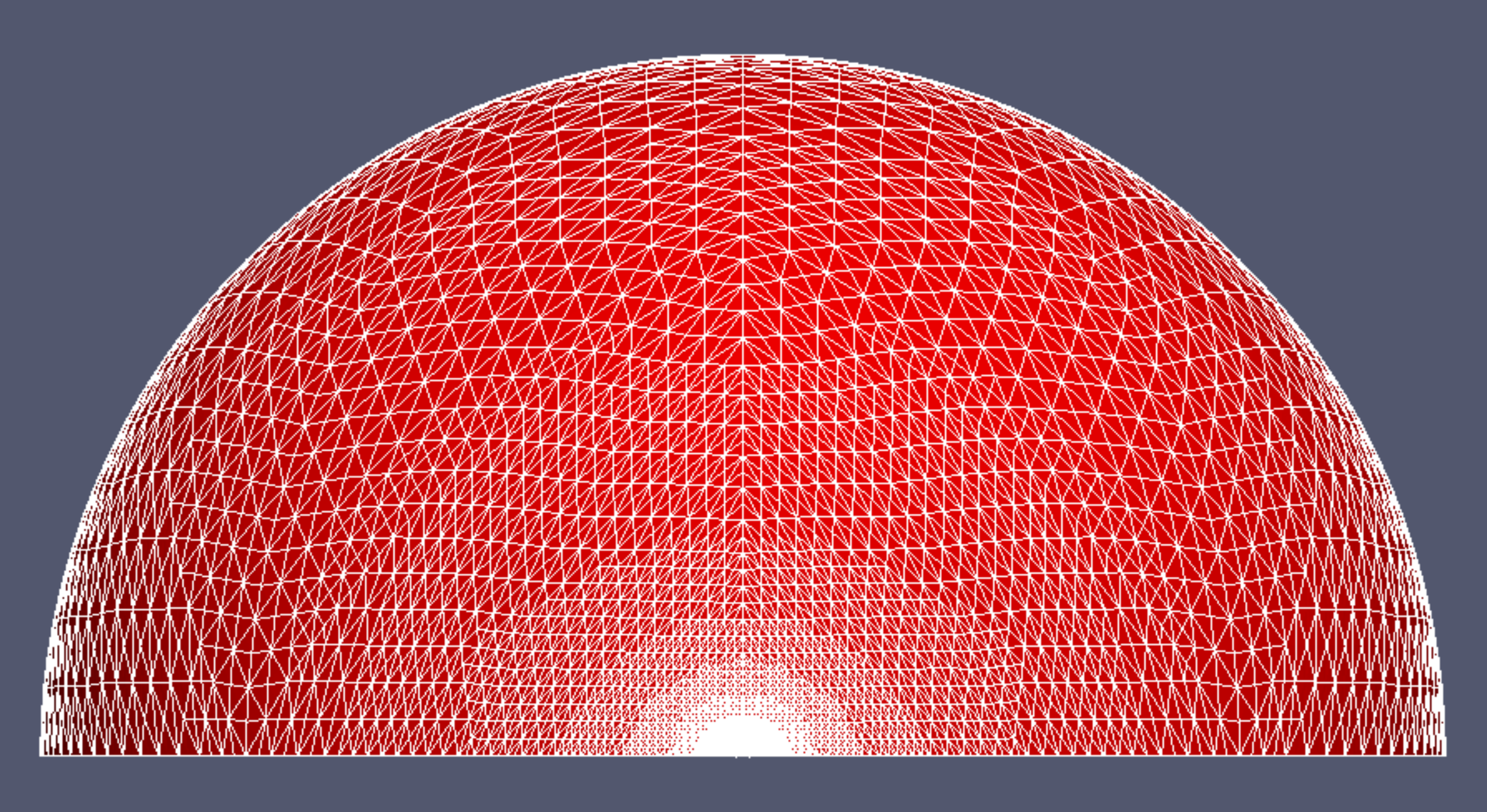}
\label{Ex1_reference_manifold_at_time_0_2}
}
\end{center}
\caption{Improvement of the computational mesh for the stationary domain presented in Figure \ref{Ex1_shape_at_time_0_0}. 
The initial mesh is shown in Figure \ref{Ex1_shape_with_mesh_at_time_0_0} (for $n=6$ initial global mesh refinements). 
The result of Algorithms \ref{algo_DeTurck}
and \ref{algo_refinement_and_coarsening_strategy} is presented
in Figure \ref{Ex1_shape_with_mesh_at_time_0_2}. 
In both pictures the red colour indicates the original shape of the domain.
Figure \ref{Ex1_shape_with_mesh_at_time_0_2} indicates that the scheme
provides a mesh of high quality with a good approximation of the original shape. 
The mesh of the reference manifold is presented in \ref{Ex1_reference_manifold_at_time_0_0} 
and \ref{Ex1_reference_manifold_at_time_0_2}.
See Example $1$ for more details.
} 
\label{Ex1_fig_domain_plus_mesh}
\end{figure}

\gdef\gplbacktext{}%
\gdef\gplfronttext{}%
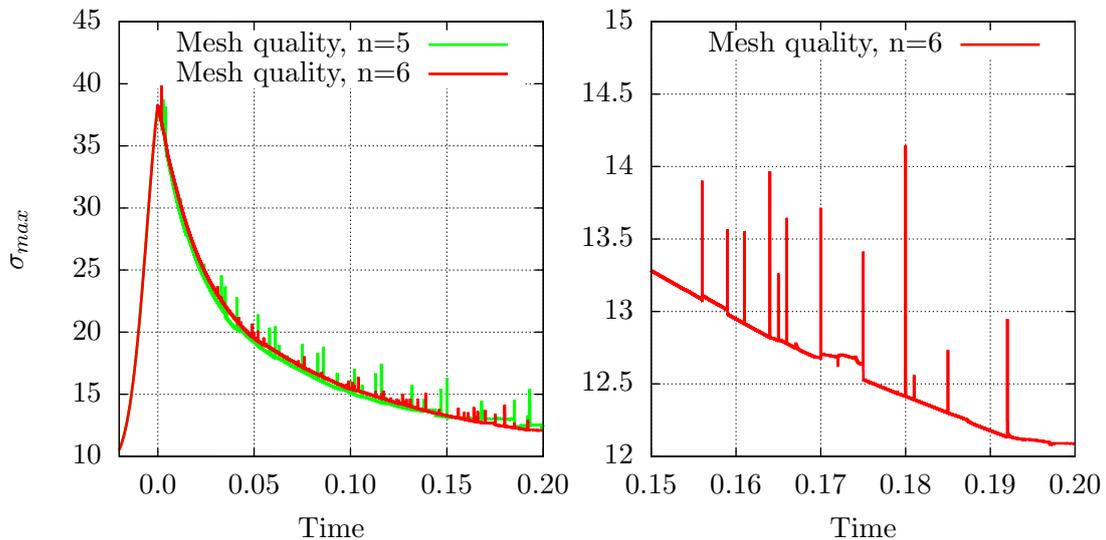
\begin{figure}
\begin{center}
\begin{picture}(7936.00,3400.00)%
    \gplgaddtomacro\gplbacktext{%
      \csname LTb\endcsname%
      \put(660,110){\makebox(0,0)[r]{\strut{} 10}}%
      \csname LTb\endcsname%
      \put(660,577){\makebox(0,0)[r]{\strut{} 15}}%
      \csname LTb\endcsname%
      \put(660,1044){\makebox(0,0)[r]{\strut{} 20}}%
      \csname LTb\endcsname%
      \put(660,1511){\makebox(0,0)[r]{\strut{} 25}}%
      \csname LTb\endcsname%
      \put(660,1977){\makebox(0,0)[r]{\strut{} 30}}%
      \csname LTb\endcsname%
      \put(660,2444){\makebox(0,0)[r]{\strut{} 35}}%
      \csname LTb\endcsname%
      \put(660,2911){\makebox(0,0)[r]{\strut{} 40}}%
      \csname LTb\endcsname%
      \put(660,3378){\makebox(0,0)[r]{\strut{} 45}}%
      \csname LTb\endcsname%
      \put(1080,-110){\makebox(0,0){\strut{}0.0}}%
      \csname LTb\endcsname%
      \put(1798,-110){\makebox(0,0){\strut{}0.05}}%
      \csname LTb\endcsname%
      \put(2517,-110){\makebox(0,0){\strut{}0.10}}%
      \csname LTb\endcsname%
      \put(3236,-110){\makebox(0,0){\strut{}0.15}}%
      \csname LTb\endcsname%
      \put(3955,-110){\makebox(0,0){\strut{}0.20}}%
      \put(22,1744){\rotatebox{90}{\makebox(0,0){\strut{}$\sigma_{max}$}}}%
      \put(2373,-440){\makebox(0,0){\strut{}Time}}%
    }%
    \gplgaddtomacro\gplfronttext{%
      \csname LTb\endcsname%
      \put(2968,3205){\makebox(0,0)[r]{\strut{}Mesh quality, n=5}}%
      \csname LTb\endcsname%
      \put(2968,2985){\makebox(0,0)[r]{\strut{}Mesh quality, n=6}}%
    }%
    \gplgaddtomacro\gplbacktext{%
      \csname LTb\endcsname%
      \put(4628,110){\makebox(0,0)[r]{\strut{} 12}}%
      \csname LTb\endcsname%
      \put(4628,655){\makebox(0,0)[r]{\strut{} 12.5}}%
      \csname LTb\endcsname%
      \put(4628,1199){\makebox(0,0)[r]{\strut{} 13}}%
      \csname LTb\endcsname%
      \put(4628,1744){\makebox(0,0)[r]{\strut{} 13.5}}%
      \csname LTb\endcsname%
      \put(4628,2289){\makebox(0,0)[r]{\strut{} 14}}%
      \csname LTb\endcsname%
      \put(4628,2833){\makebox(0,0)[r]{\strut{} 14.5}}%
      \csname LTb\endcsname%
      \put(4628,3378){\makebox(0,0)[r]{\strut{} 15}}%
      \csname LTb\endcsname%
      \put(4760,-110){\makebox(0,0){\strut{}0.15}}%
      \csname LTb\endcsname%
      \put(5393,-110){\makebox(0,0){\strut{}0.16}}%
      \csname LTb\endcsname%
      \put(6025,-110){\makebox(0,0){\strut{}0.17}}%
      \csname LTb\endcsname%
      \put(6658,-110){\makebox(0,0){\strut{}0.18}}%
      \csname LTb\endcsname%
      \put(7290,-110){\makebox(0,0){\strut{}0.19}}%
      \csname LTb\endcsname%
      \put(7923,-110){\makebox(0,0){\strut{}0.20}}%
      \put(6341,-440){\makebox(0,0){\strut{}Time}}%
    }%
    \gplgaddtomacro\gplfronttext{%
      \csname LTb\endcsname%
      \put(6936,3205){\makebox(0,0)[r]{\strut{}Mesh quality, n=6}}%
    }%
    \gplbacktext
    \put(0,0){\includegraphics{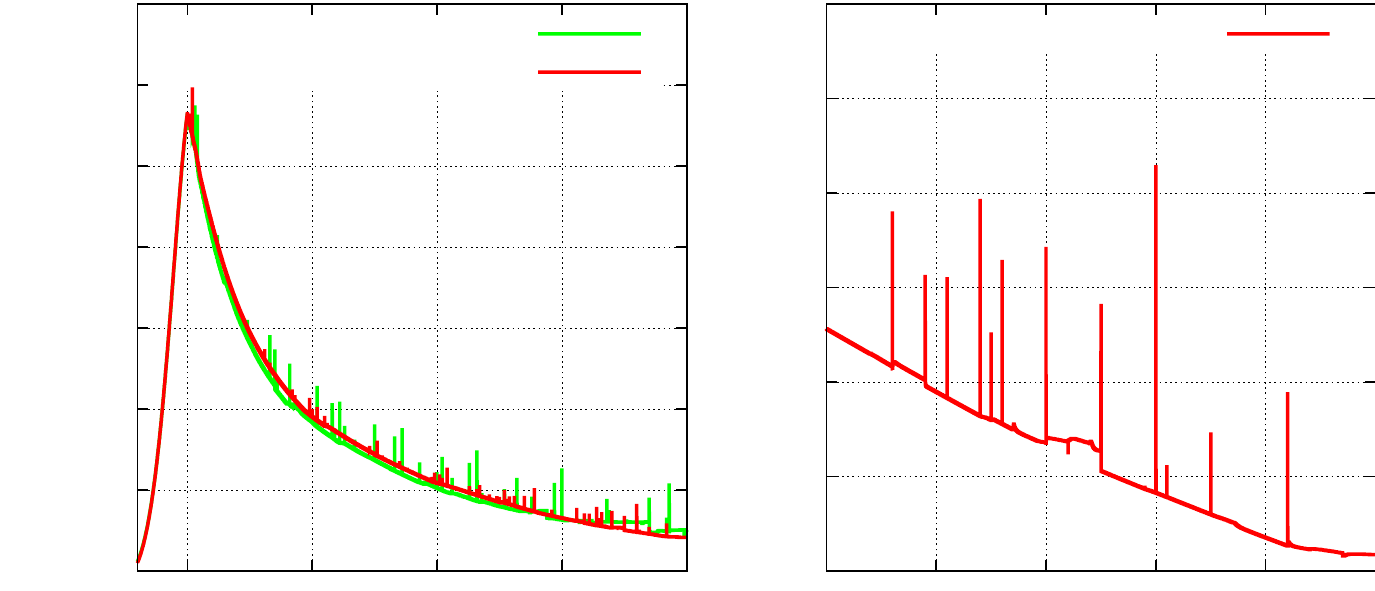}}%
    \gplfronttext
  \end{picture}%
  \vspace*{10mm}
  \caption{Mesh qualtity $\sigma_{max}$, see (\ref{definition_sigma_max}),
  for the computational mesh in Figure \ref{Ex1_fig_domain_plus_mesh}.
  In the left image the results are shown for $n=5$ and $n=6$ initial global mesh refinements.  
  For $t < 0$, the mesh is deformed in order to obtain a bad test mesh. At time $t=0$,
  this deformation is stopped. For $t > 0$, Algorithms \ref{algo_DeTurck} and 
  \ref{algo_refinement_and_coarsening_strategy} are applied to improve the mesh quality again.
  The refinement and coarsening strategy in Algorithm  
  \ref{algo_refinement_and_coarsening_strategy} can lead to sharp jumps in the mesh quality,
  see the enlarged section on the right hand side. 
  Since mesh refinement and coarsening
  is a local procedure, its effect on the mesh quality 
  is corrected very quickly by Algorithm \ref{algo_DeTurck}.
  See Example $1$ for more details.	  
  } 
  \label{Ex1_mesh_quality}
\end{center}
\end{figure}

\subsection*{Example 2: Mesh improvement for moving flat domains in $\mathbb{R}^2$}
We now use Algorithm \ref{algo_DeTurck} in order to preserve the mesh quality of 
a moving domain when the initial mesh is already of high quality. 
\subsubsection*{Example 2.1}
We first consider the unit disk 
$\Gamma(0) := B_1(0) \subset \mathbb{R}^2$ which is deformed according to (\ref{equation_of_motion_non-reparametrized})
with
\begin{align}
	v(x_1, x_2) = (0.0, -x_2 (1.0 - x_1^2)^2 + 0.2x_1)^T.
	\label{velocity_field_in_Ex2_moving_manifold}
\end{align}
The computational parameters are $\tau = 0.02 ~ h_{min}^2$, $\alpha = 1.0$ and $T_{adapt} = 0.01$. 
The results are presented in Figures \ref{Ex2_moving_domain_shape_plus_mesh} and \ref{Ex2_mesh_quality}.
 
\begin{figure}
\begin{center}
\subfloat[][\centering Intial domain with computational mesh.]
{\includegraphics[width=0.296\textwidth]{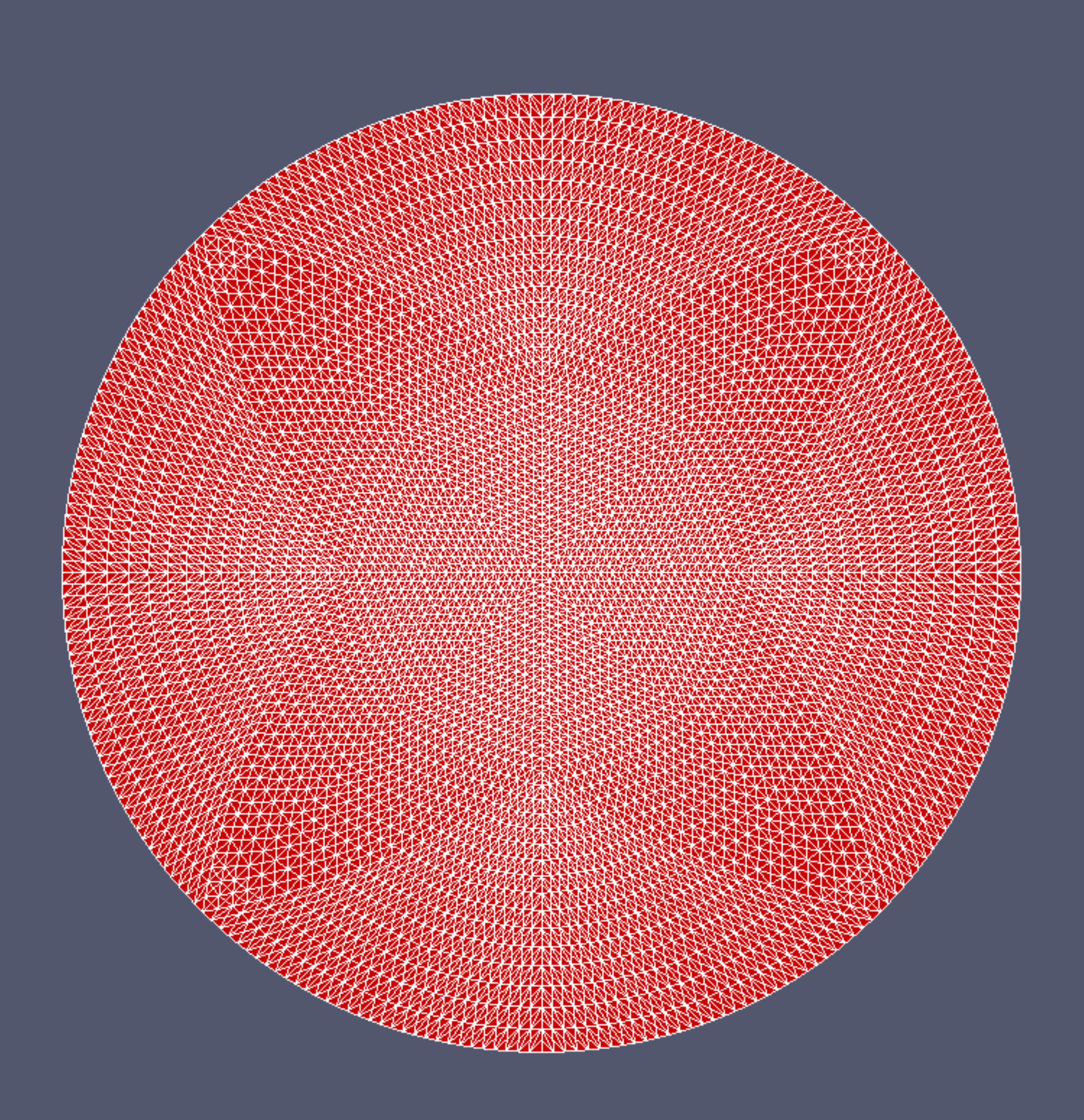}
\label{Ex2_moving_manifold_at_time_0_00}
} 
~~~
\subfloat[][\centering Computational mesh for the final domain obtained by DeTurck redistribution.]
{\includegraphics[width=0.4\textwidth]{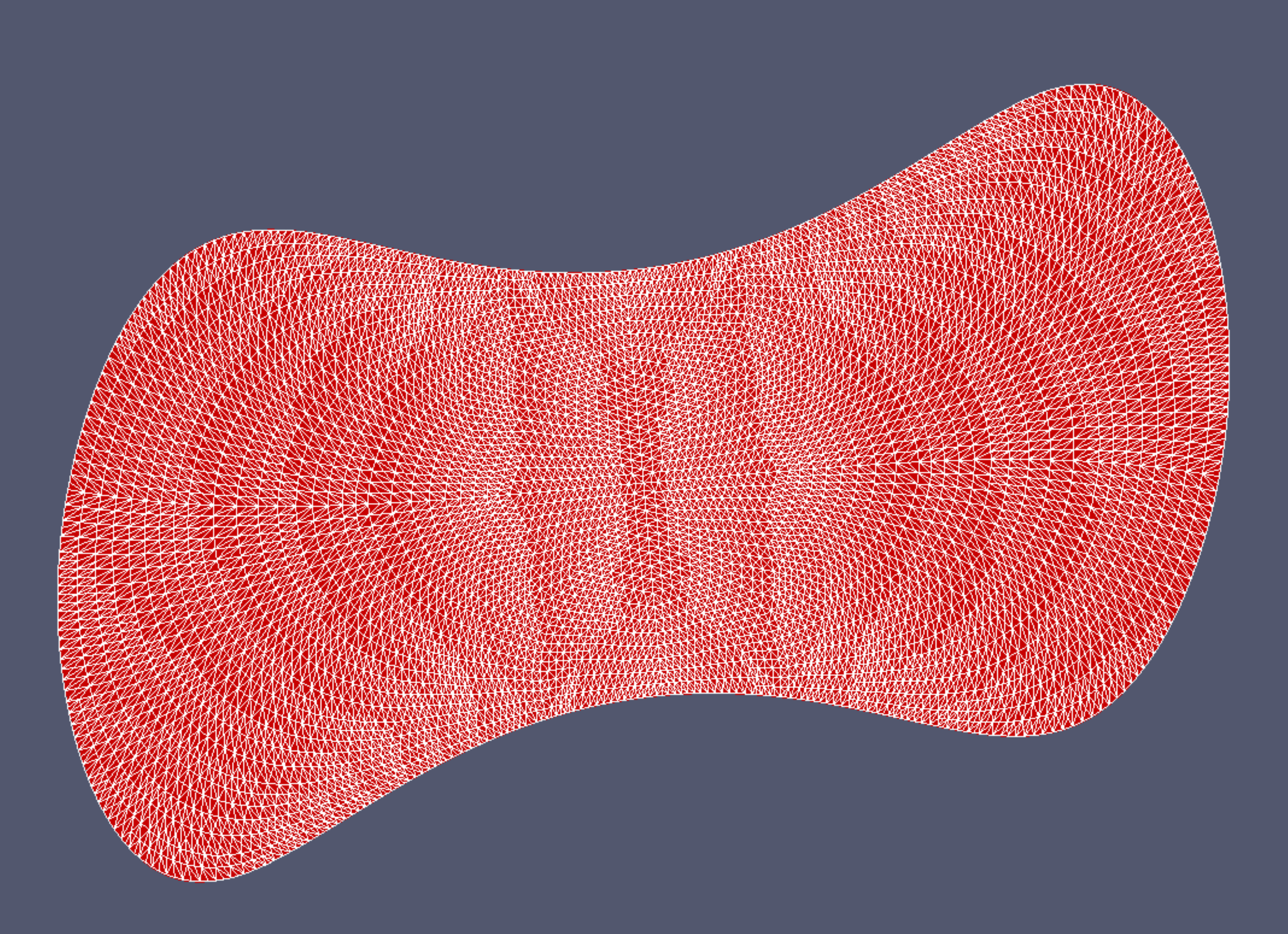}
\label{Ex2_moving_manifold_at_time_1_00}
} 
\\
\subfloat[][\centering Computational mesh for the final domain without any redistribution of mesh vertices.]
{\includegraphics[width=0.4\textwidth]{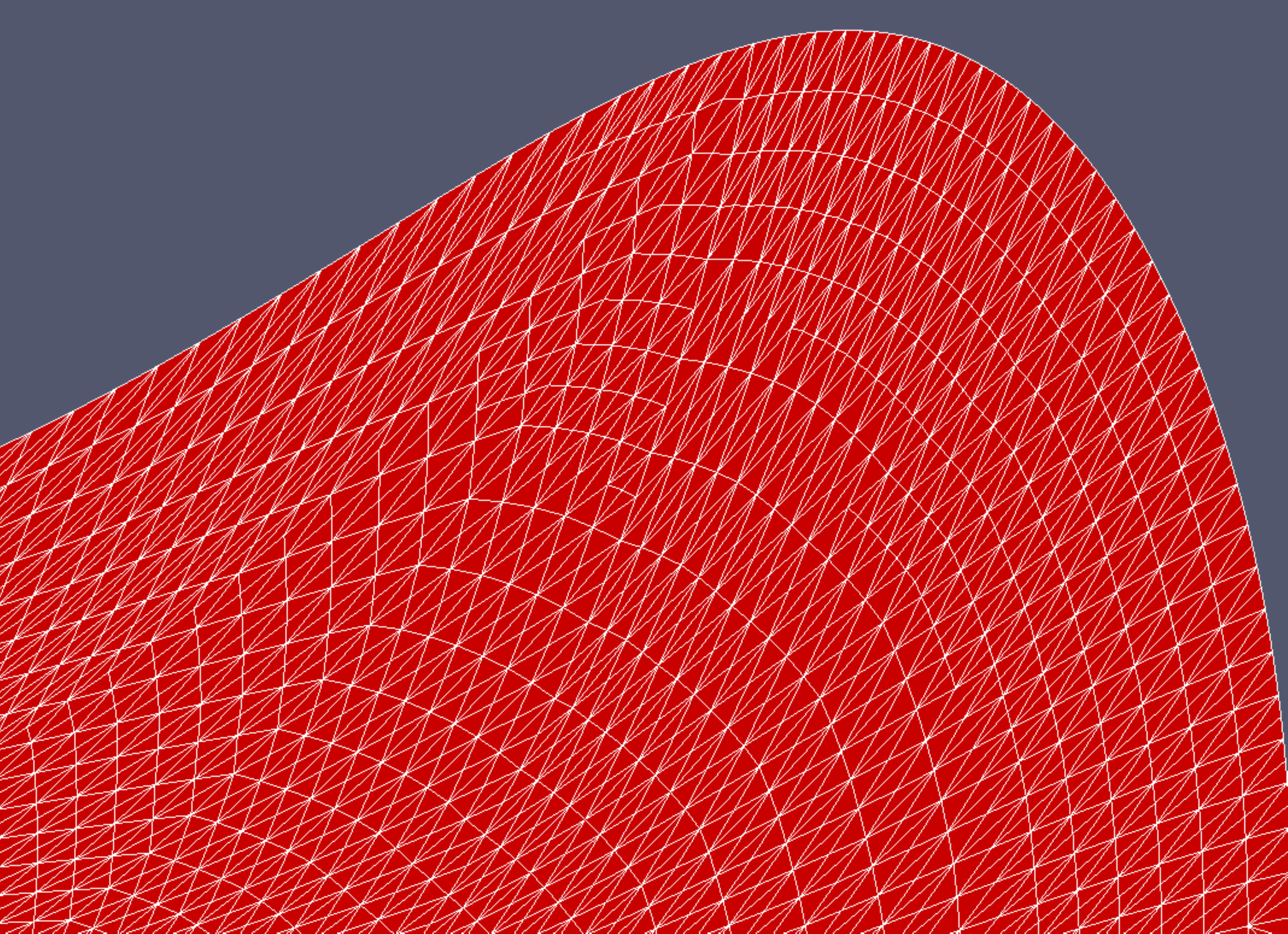}
\label{Ex2_moving_manifold_without_DeT_at_time_1_00}
} 
~~~
\subfloat[][\centering Computational mesh for the final domain obtained by DeTurck redistribution.]
{\includegraphics[width=0.4\textwidth]{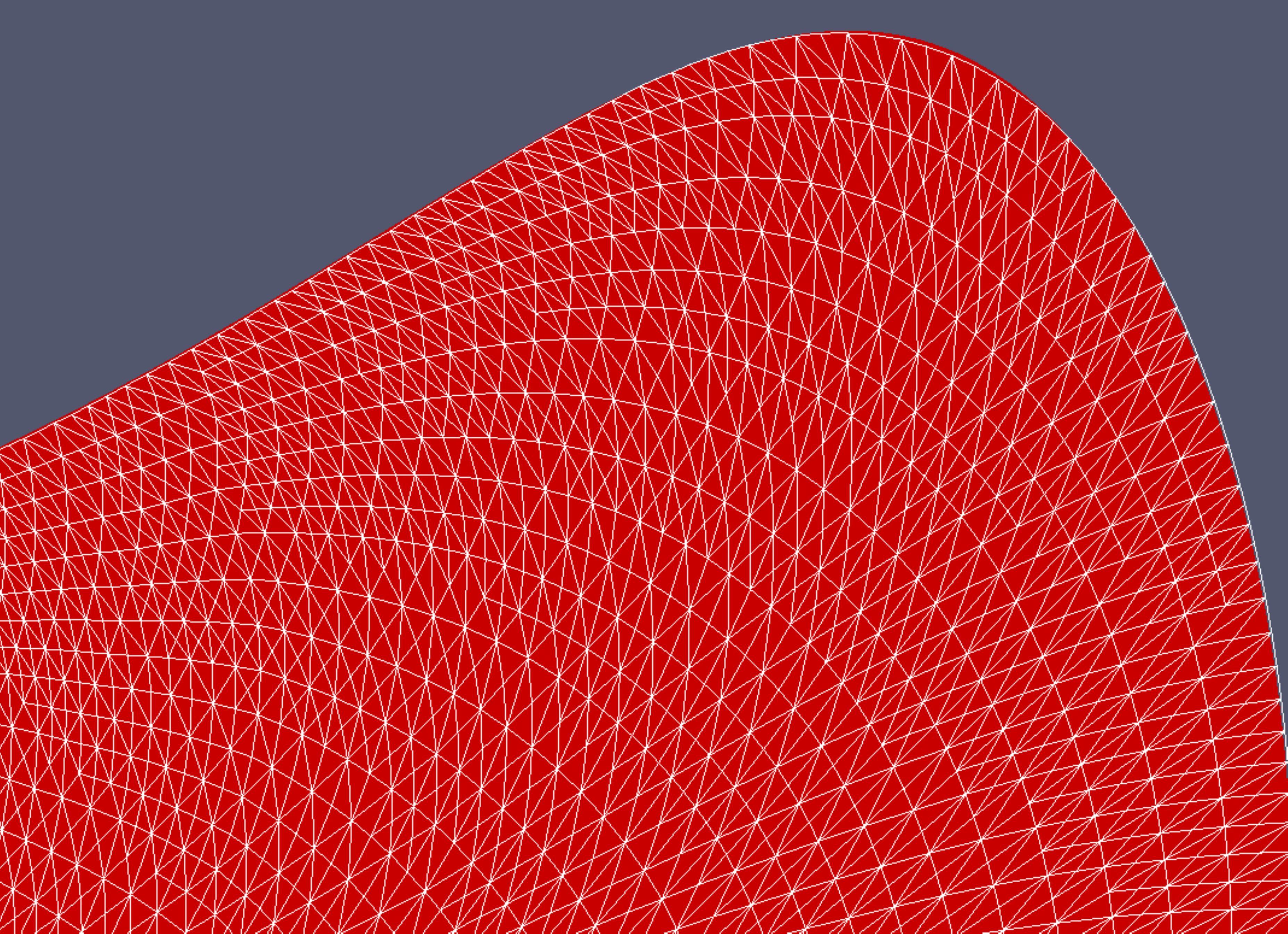}
\label{Ex2_moving_manifold_with_DeT_at_time_1_00}
} 
\end{center}
\caption{Comparison of the mesh behaviour for a moving domain in $\mathbb{R}^2$. 
The circular shape in Figure \ref{Ex2_moving_manifold_at_time_0_00} is deformed
according to the velocity field $v$ in (\ref{velocity_field_in_Ex2_moving_manifold}) into the domain shown in Figure \ref{Ex2_moving_manifold_at_time_1_00}. 
An enlarged section of the resulting mesh is presented in 
Figures \ref{Ex2_moving_manifold_without_DeT_at_time_1_00} and \ref{Ex2_moving_manifold_with_DeT_at_time_1_00}, respectively. 
If Algorithms \ref{algo_DeTurck} and \ref{algo_refinement_and_coarsening_strategy} are applied, the mesh quality is not affected by the shape 
deformation, see Figures \ref{Ex2_moving_manifold_with_DeT_at_time_1_00} and  \ref{Ex2_mesh_quality}.
See Example $2.1$ for more details.
} 
\label{Ex2_moving_domain_shape_plus_mesh}
\end{figure}

\gdef\gplbacktext{}%
\gdef\gplfronttext{}%
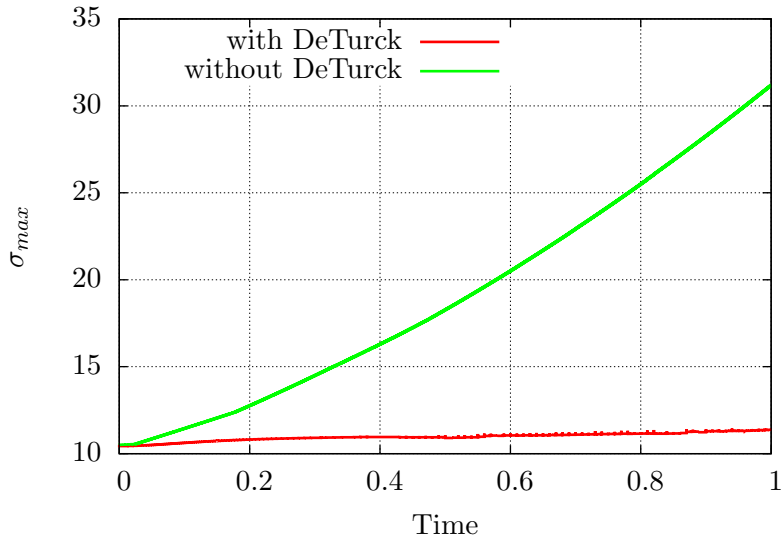
\begin{figure}
\begin{center}
 \begin{picture}(5668.00,3400.00)%
    \gplgaddtomacro\gplbacktext{%
      \csname LTb\endcsname%
      \put(660,110){\makebox(0,0)[r]{\strut{} 10}}%
      \csname LTb\endcsname%
      \put(660,764){\makebox(0,0)[r]{\strut{} 15}}%
      \csname LTb\endcsname%
      \put(660,1417){\makebox(0,0)[r]{\strut{} 20}}%
      \csname LTb\endcsname%
      \put(660,2071){\makebox(0,0)[r]{\strut{} 25}}%
      \csname LTb\endcsname%
      \put(660,2724){\makebox(0,0)[r]{\strut{} 30}}%
      \csname LTb\endcsname%
      \put(660,3378){\makebox(0,0)[r]{\strut{} 35}}%
      \csname LTb\endcsname%
      \put(792,-110){\makebox(0,0){\strut{} 0}}%
      \csname LTb\endcsname%
      \put(1765,-110){\makebox(0,0){\strut{} 0.2}}%
      \csname LTb\endcsname%
      \put(2737,-110){\makebox(0,0){\strut{} 0.4}}%
      \csname LTb\endcsname%
      \put(3710,-110){\makebox(0,0){\strut{} 0.6}}%
      \csname LTb\endcsname%
      \put(4682,-110){\makebox(0,0){\strut{} 0.8}}%
      \csname LTb\endcsname%
      \put(5655,-110){\makebox(0,0){\strut{} 1}}%
      \put(22,1744){\rotatebox{90}{\makebox(0,0){\strut{}$\sigma_{max}$}}}%
      \put(3223,-440){\makebox(0,0){\strut{}Time}}%
    }%
    \gplgaddtomacro\gplfronttext{%
      \csname LTb\endcsname%
      \put(2904,3205){\makebox(0,0)[r]{\strut{}with DeTurck}}%
      \csname LTb\endcsname%
      \put(2904,2985){\makebox(0,0)[r]{\strut{}without DeTurck}}%
    }%
    \gplbacktext
    \put(0,0){\includegraphics{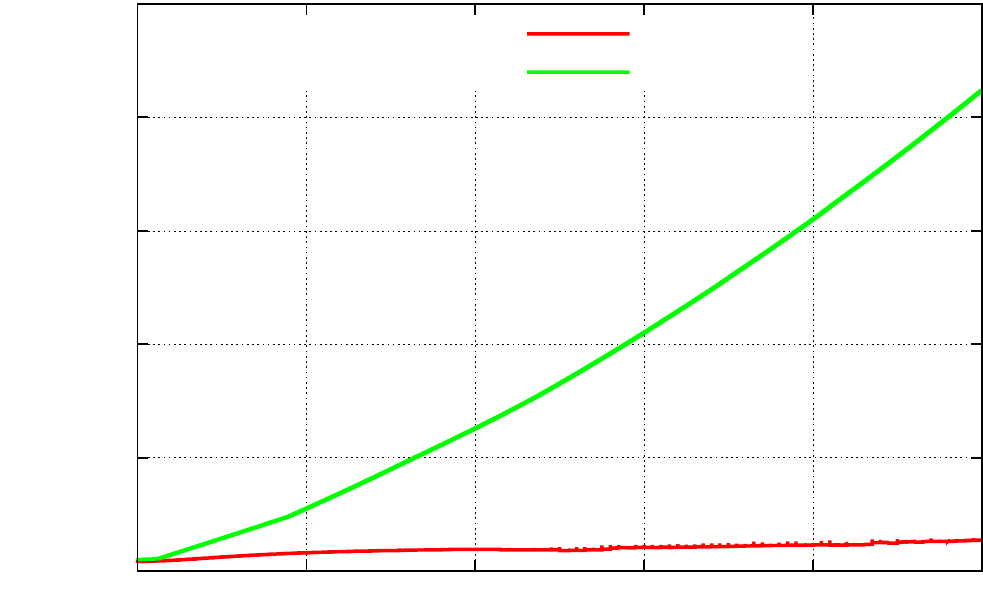}}%
    \gplfronttext
  \end{picture}%
  \vspace*{10mm}
  \caption{Mesh quality $\sigma_{max}$, see (\ref{definition_sigma_max}), for the computational mesh in Figure \ref{Ex2_moving_domain_shape_plus_mesh}. 
  The domain deformation reduces the mesh quality considerably
  if the mesh vertices are not redistributed. In contrast, the mesh quality remains high if 
  the mesh is moved according to Algorithm \ref{algo_DeTurck}.
  See Example $2.1$ for more details.	  
  }
  \label{Ex2_mesh_quality}
\end{center}
\end{figure}

\subsubsection*{Example 2.2}
In the next example, we consider the domain $\Gamma(t) = B_{r_1}(0) \setminus B_{r_2}(p(t)) \subset \mathbb{R}^2$
with the hole $B_{r_2}(p(t))$ around the center $p(t) \in \mathbb{R}^2$; see Figure \ref{moving_inclusion_at_time_0_0}. We choose $r_1 = 2.25$ and $r_2 = 0.25$.
The evolution of the domain $\Gamma(t)$ is given by the velocity field $v$ solving 
\begin{equation*}
	\Delta v = 0, \quad \textnormal{in $\Gamma(t)$},
\end{equation*}
and $v(x_1, x_2) = 4 (- \sin(2\pi t), \cos(2\pi t))^T$ on the interior boundary as well as $v(x_1,x_2) =0$ on the exterior boundary.
In the time interval $[0,1]$, this velocity field induces a circular movement of the hole $B_{r_2}(p(t))$ in the interior of the disk
$B_{r_1}(0)$. The boundary $\partial B_{r_1}(0)$ remains unchanged.
We approximate the velocity field $v$ by the solution $v_h^m \in V_h(\Gamma_h^m)^2$ of
\begin{align}
	\int_{\Gamma_h^{m}} \nabla_{\Gamma_h^m} v_h^m : \nabla_{\Gamma_h^m} \varphi_h  \; do = 0, 
	\quad \forall \varphi_h \in \overset{\circ}{V}_h(\Gamma_h^m)^2,
	\label{velocity_field_moving_inclusion}
\end{align}
and $v_h^m = I_h v$ on $\partial \Gamma_h^m$. 
Since $\Gamma(t)$ is not simply-connected, we here cannot use the half-sphere as a reference manifold.
Instead, we choose the cylinder (\ref{definition_cylinder}).
The computational parameters are $\tau = 0.001 ~ h_{min}^2$, $\alpha = 0.1$ and $T_{adapt} = 10^{-3}$.
The results of the simulations with and without redistribution of mesh vertices
are shown in Figures \ref{moving_inclusion_figures}
and \ref{moving_inclusion_mesh_quality}. When Algorithm \ref{algo_DeTurck}
is not applied, the mesh totally degenerates (including mesh entanglements) and hence could not be used to solve a PDE on
the moving domain. In contrast, the mesh obtained by our DeTurck scheme preserves its high quality
for all time.
\begin{figure}
\begin{center}
\subfloat[][\centering Initial domain with computational mesh.]
{\includegraphics[width=0.266\textwidth]{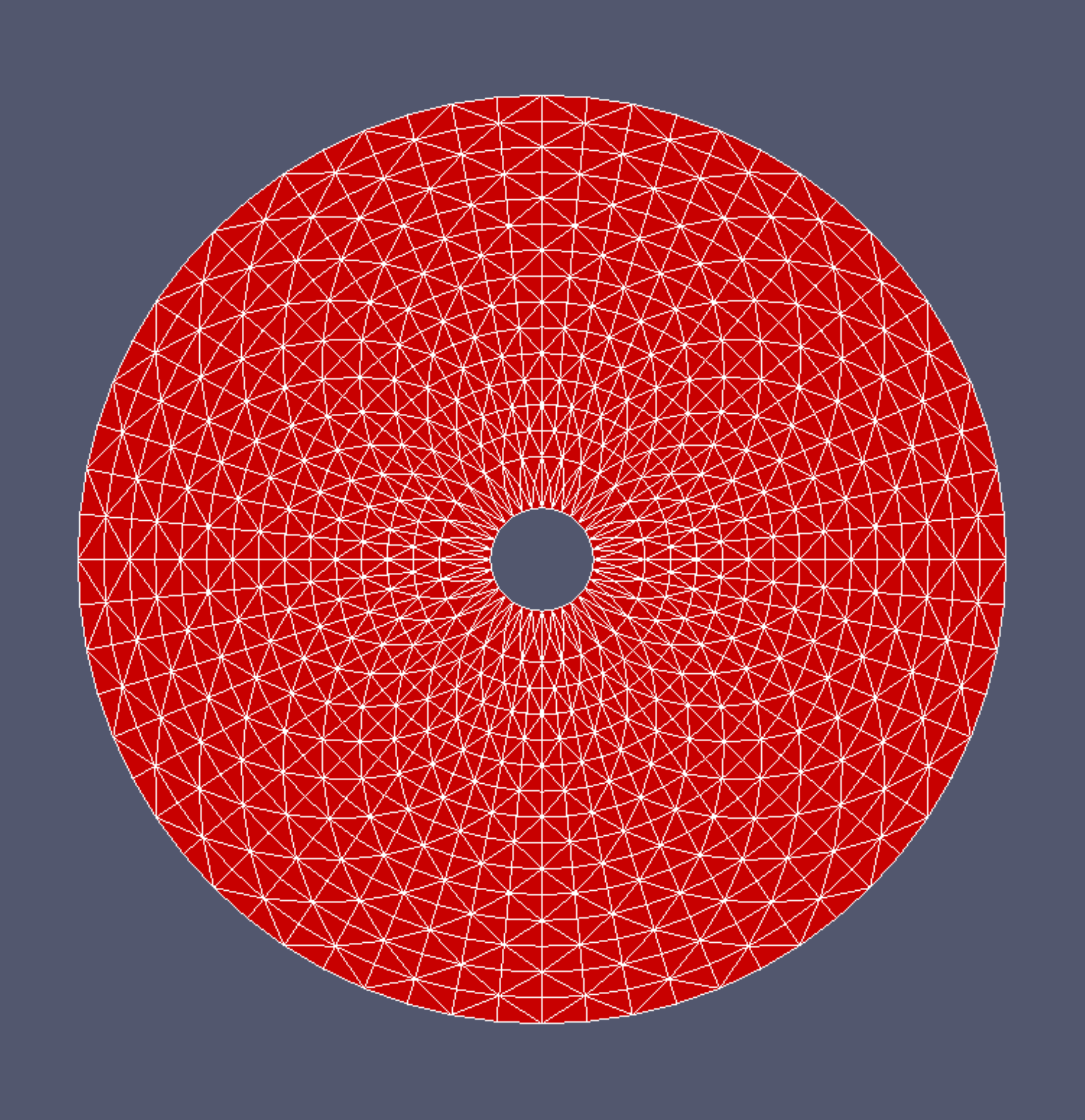}
\label{moving_inclusion_at_time_0_0}
}
~~
\subfloat[][\centering Computational mesh at $t=0.25$ with DeTurck redistribution.]
{\includegraphics[width=0.266\textwidth]{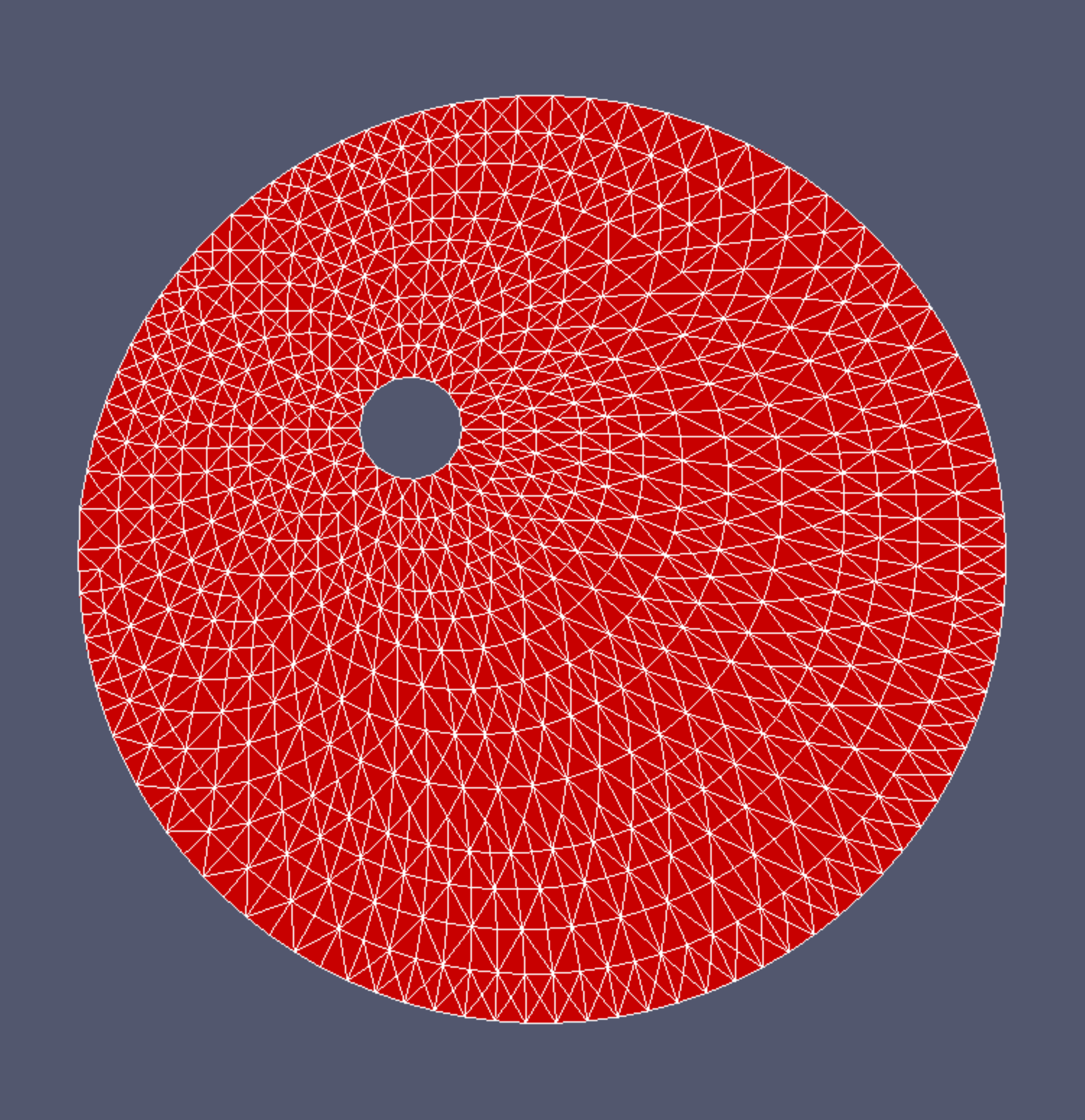}
\label{moving_inclusion_at_time_0_25}
} 
~~
\subfloat[][\centering Computational mesh at $t=0.5$ with DeTurck redistribution.]
{\includegraphics[width=0.266\textwidth]{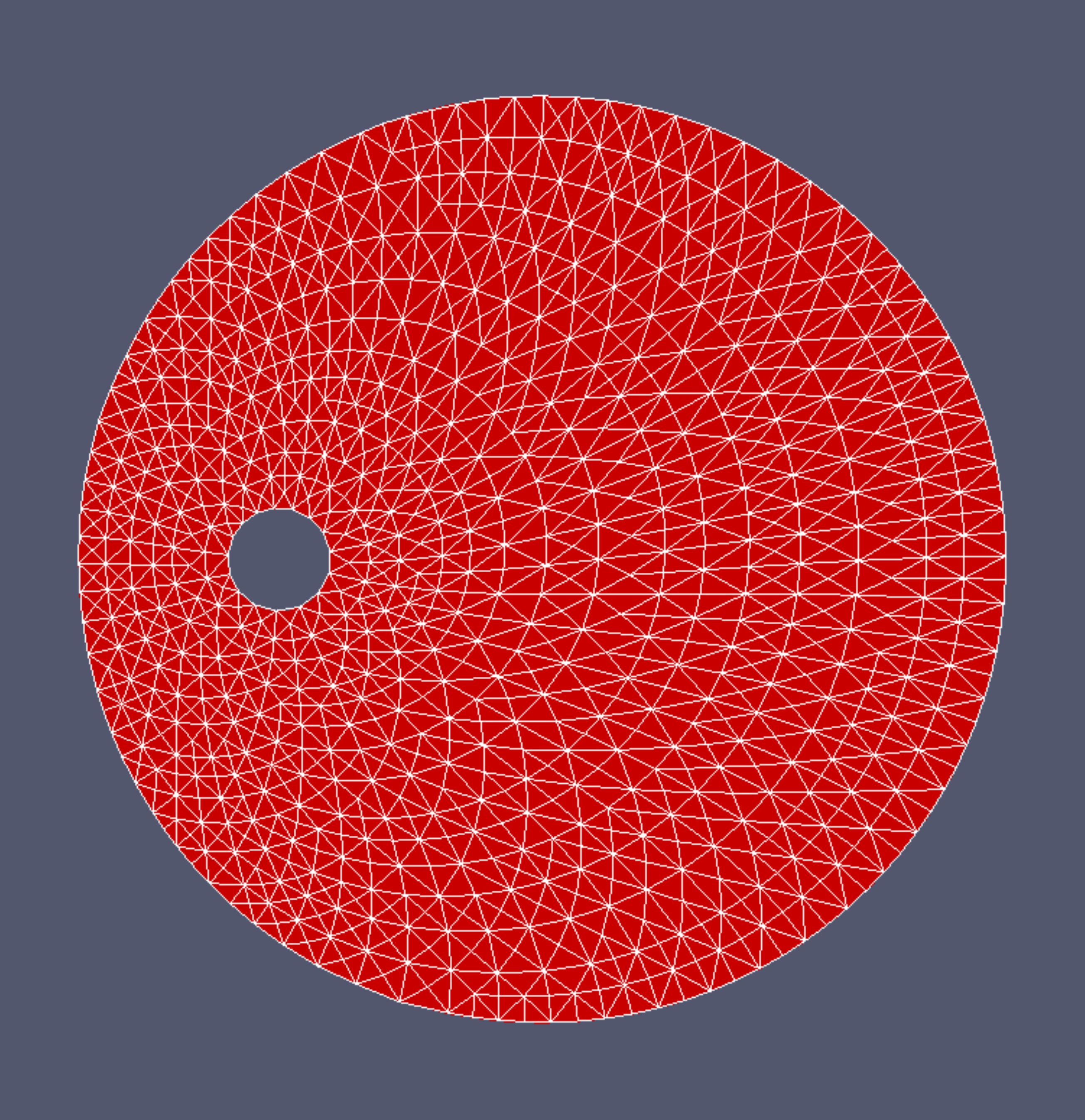}
\label{moving_inclusion_at_time_0_50}
} 
\\
\subfloat[][\centering Computational mesh at $t=0.75$ with DeTurck redistribution.]
{\includegraphics[width=0.266\textwidth]{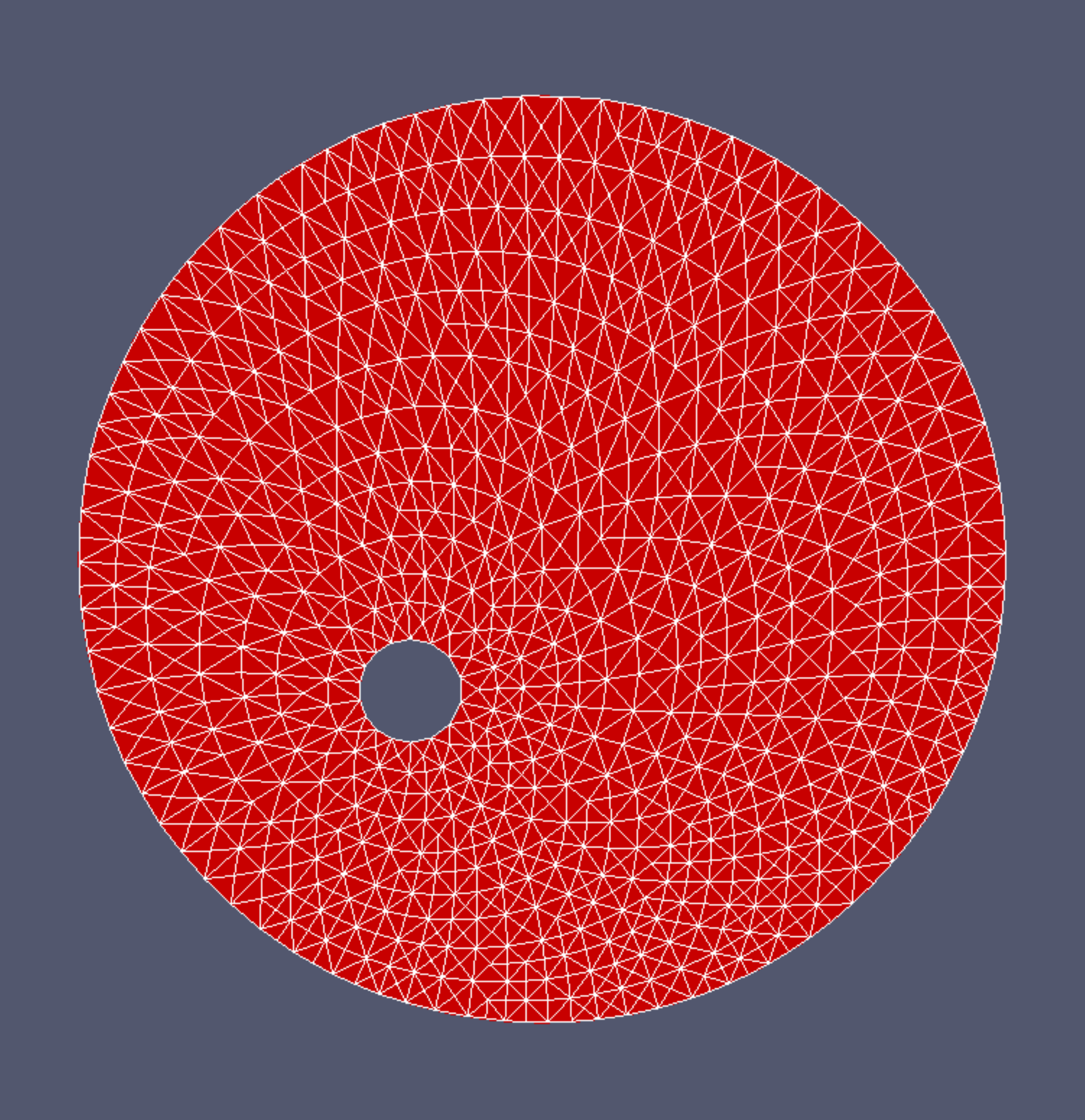}
\label{moving_inclusion_at_time_0_75}
} 
~~
\subfloat[][\centering Computational mesh at $t=1.0$ with DeTurck redistribution.]
{\includegraphics[width=0.266\textwidth]{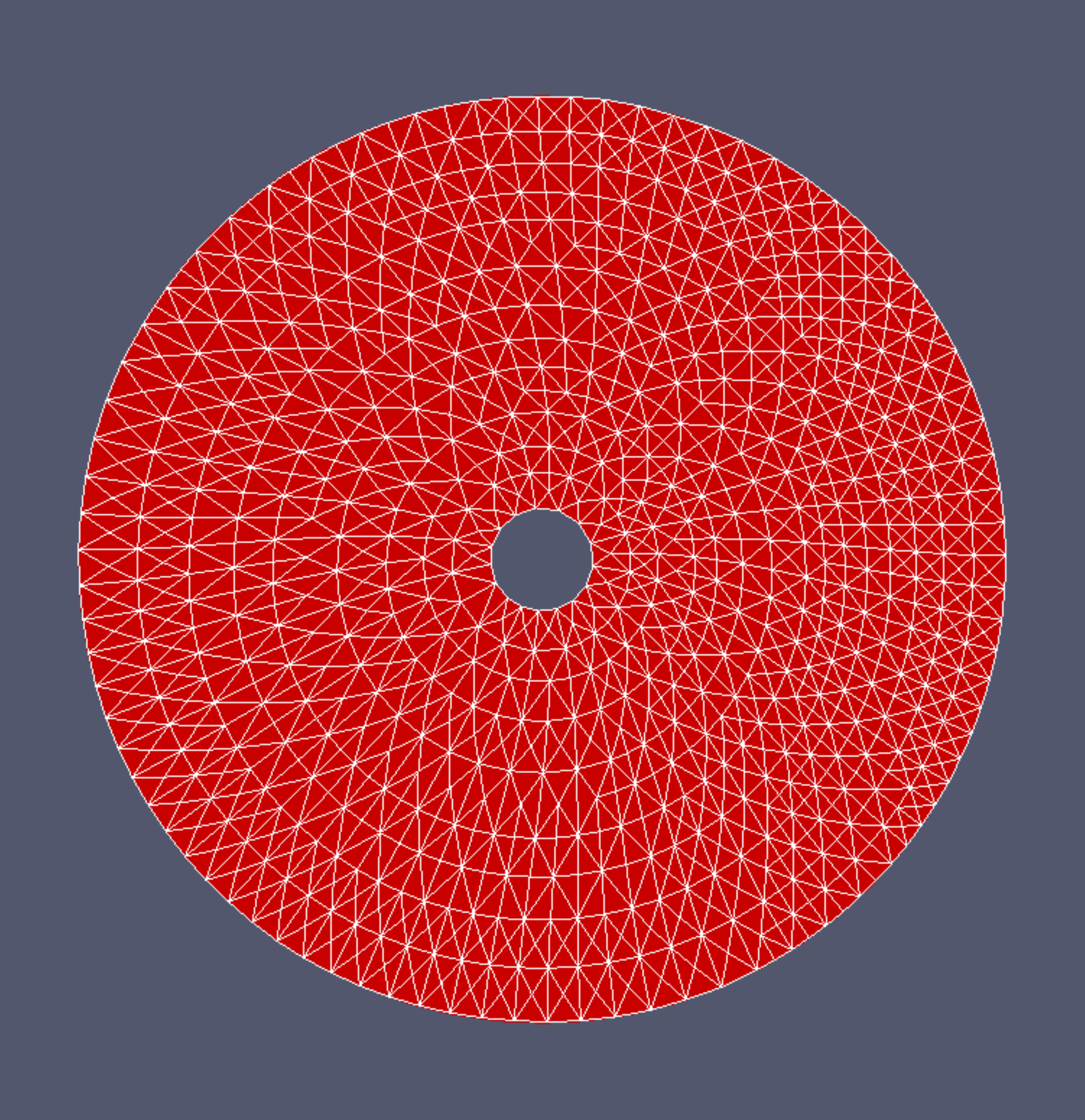}
\label{moving_inclusion_at_time_1_00}
}  
~~
\subfloat[][\centering Reference manifold at time $t=1.0$.]
{\includegraphics[width=0.266\textwidth]{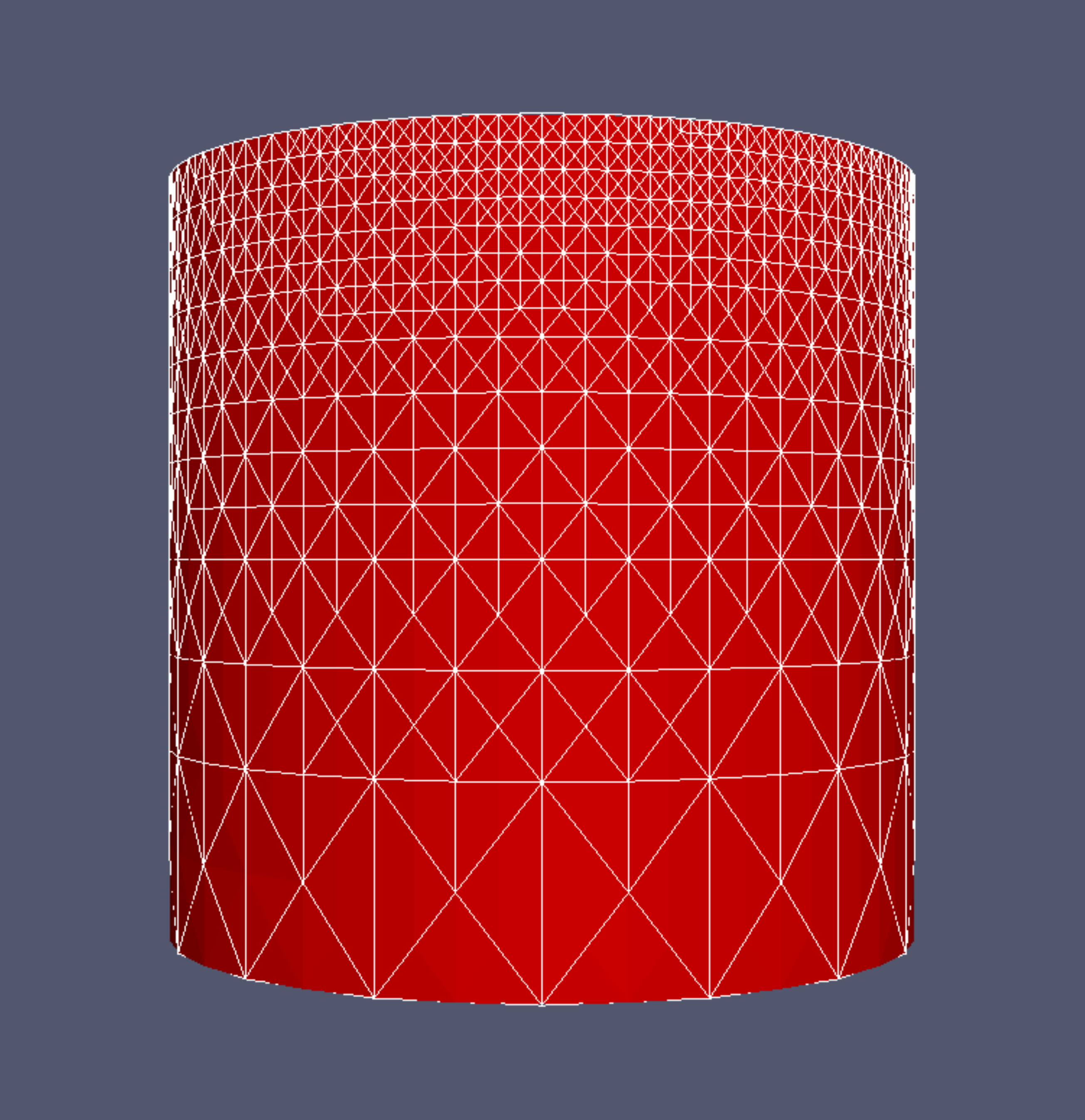}
\label{moving_inclusion_reference_manifold_at_time_1_0}
}
\\
\subfloat[][\centering Computational mesh at $t=0.25$ without redistribution.]
{\includegraphics[width=0.2\textwidth]{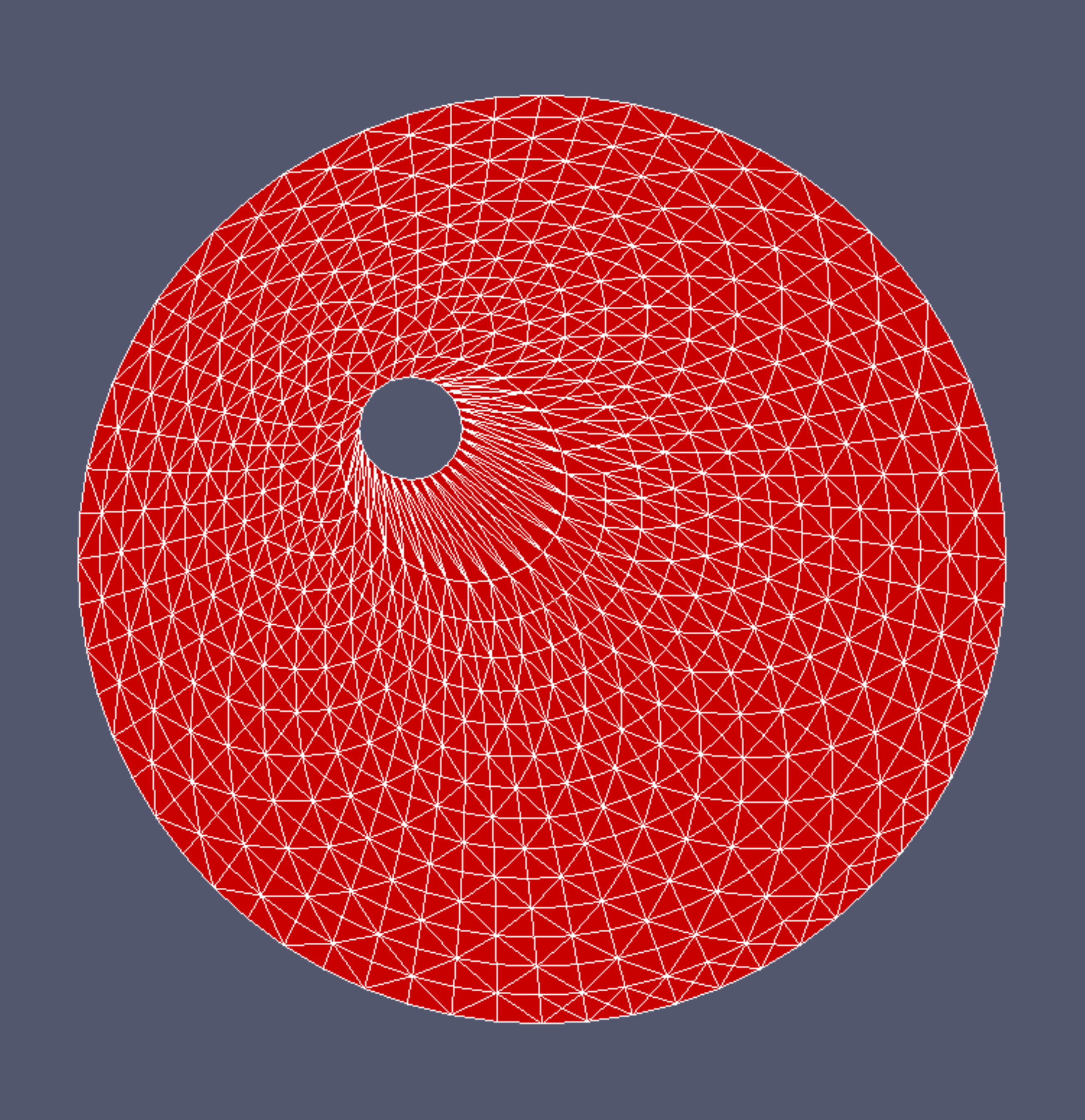}
\label{moving_inclusion_without_DeT_at_time_0_25}
} 
~~
\subfloat[][\centering Computational mesh at $t=0.5$ without redistribution.]
{\includegraphics[width=0.2\textwidth]{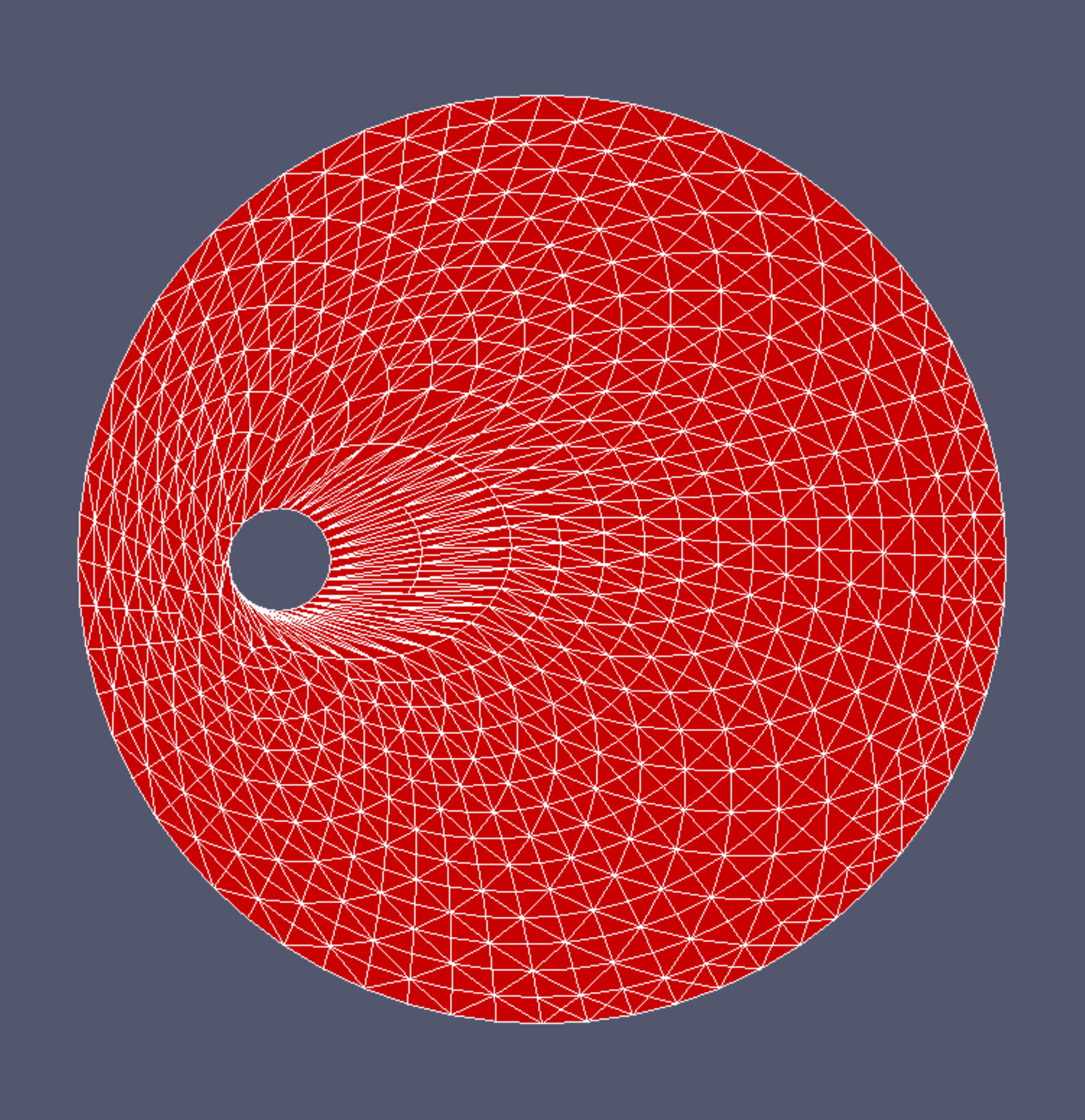}
\label{moving_inclusion_without_DeT_at_time_0_50}
} 
~~
\subfloat[][\centering Computational mesh at $t=0.75$ without redistribution.]
{\includegraphics[width=0.2\textwidth]{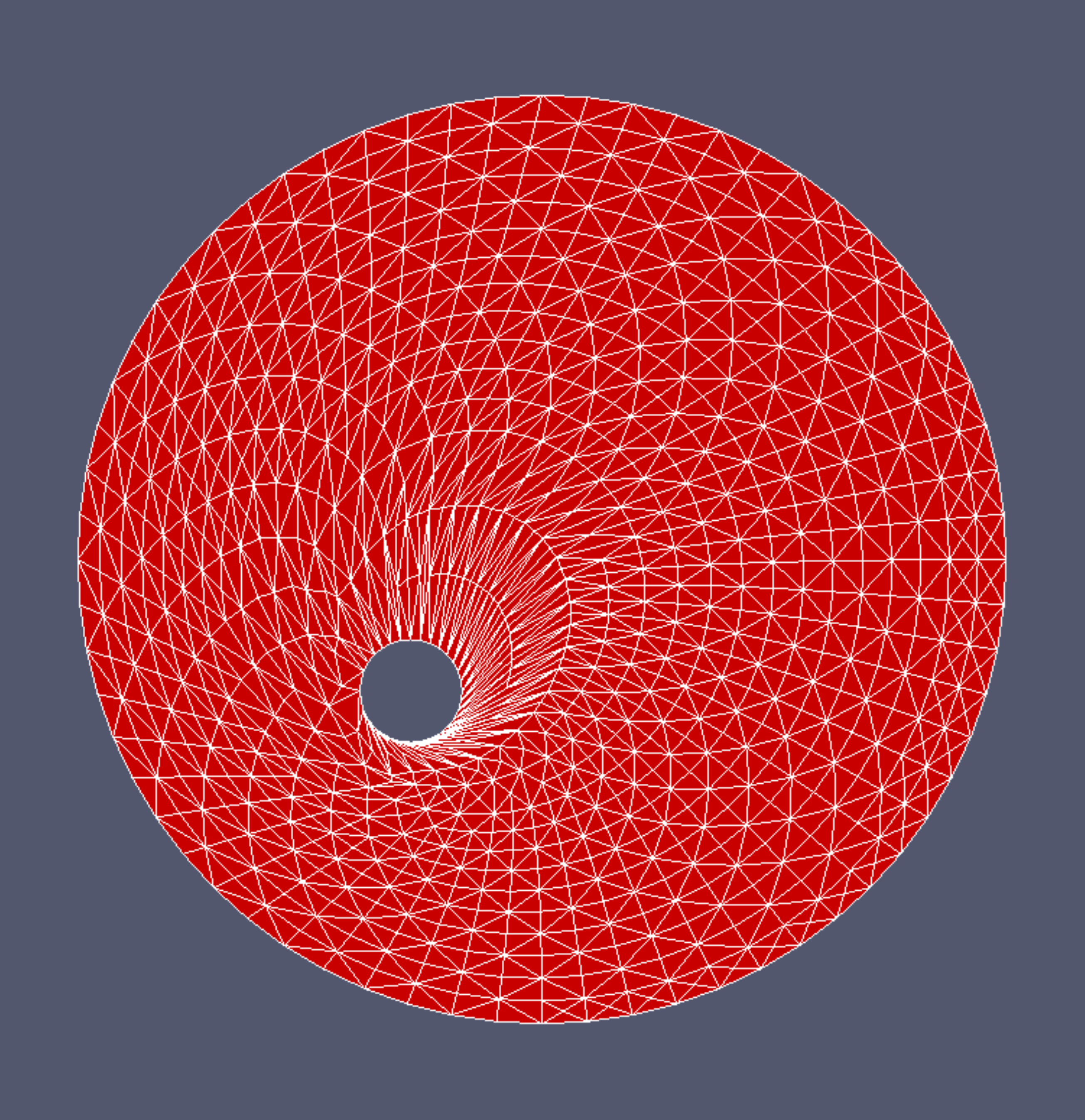}
\label{moving_inclusion_without_DeT_at_time_0_75}
} 
~~
\subfloat[][\centering Computational mesh at $t=1.0$ without redistribution.]
{\includegraphics[width=0.2\textwidth]{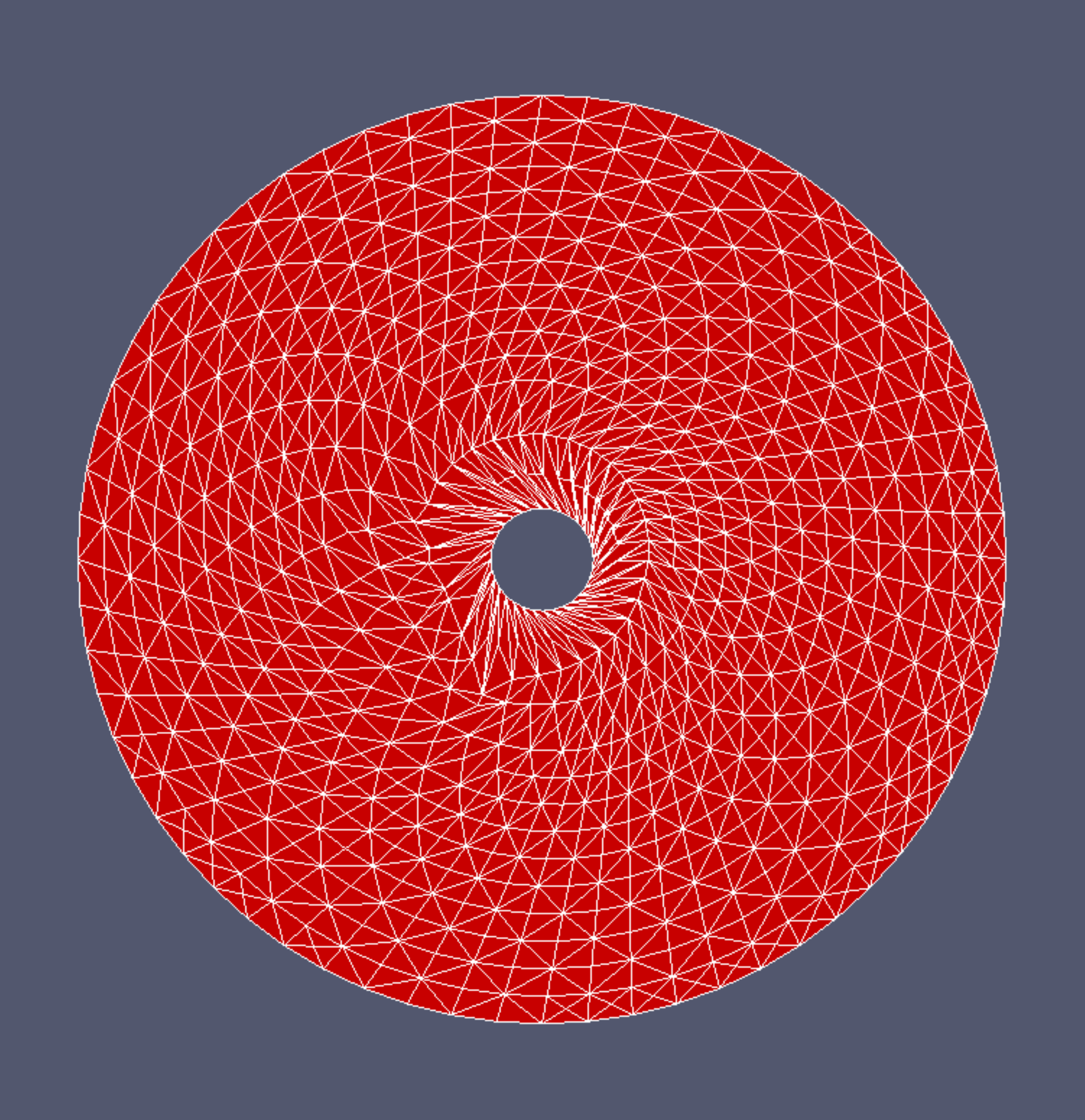}
\label{moving_inclusion_without_DeT_at_time_1_00}
} 
\caption{Comparison of the mesh behaviour for a moving hole in the interior of the disk $B_{r_1}(0)$. 
The initial mesh is presented in Figure \ref{moving_inclusion_at_time_0_0}.
In this example, the reference manifold is the cylinder (\ref{definition_cylinder}); see Figure \ref{moving_inclusion_reference_manifold_at_time_1_0}
for the reference mesh $\M_h$ at time $t=1.0$. The local mesh refinement and coarsening of the reference
mesh is due to Algorithm \ref{algo_refinement_and_coarsening_strategy}. Figures \ref{moving_inclusion_without_DeT_at_time_0_25} to \ref{moving_inclusion_without_DeT_at_time_1_00}
show the time-development of the computational mesh moved according to the velocity field defined in (\ref{velocity_field_moving_inclusion}).
This motion leads to the degeneration of the mesh (including mesh entanglements) after a short time. In contrast, our redistribution scheme
in Algorithm \ref{algo_DeTurck} (together with Algorithm \ref{algo_refinement_and_coarsening_strategy})
provides a high-quality mesh for all times;
see Figures \ref{moving_inclusion_at_time_0_25} to \ref{moving_inclusion_at_time_1_00}
as well as Figure \ref{moving_inclusion_mesh_quality}.
See Example $2.2$ for more details.
}
\label{moving_inclusion_figures}
\end{center}
\end{figure}

\gdef\gplbacktext{}%
\gdef\gplfronttext{}%
\begin{figure}
\begin{center}
 \begin{picture}(5668.00,3400.00)%
    \gplgaddtomacro\gplbacktext{%
      \csname LTb\endcsname%
      \put(660,110){\makebox(0,0)[r]{\strut{} 5}}%
      \csname LTb\endcsname%
      \put(660,577){\makebox(0,0)[r]{\strut{} 10}}%
      \csname LTb\endcsname%
      \put(660,1044){\makebox(0,0)[r]{\strut{} 15}}%
      \csname LTb\endcsname%
      \put(660,1511){\makebox(0,0)[r]{\strut{} 20}}%
      \csname LTb\endcsname%
      \put(660,1977){\makebox(0,0)[r]{\strut{} 25}}%
      \csname LTb\endcsname%
      \put(660,2444){\makebox(0,0)[r]{\strut{} 30}}%
      \csname LTb\endcsname%
      \put(660,2911){\makebox(0,0)[r]{\strut{} 35}}%
      \csname LTb\endcsname%
      \put(660,3378){\makebox(0,0)[r]{\strut{} 40}}%
      \csname LTb\endcsname%
      \put(792,-110){\makebox(0,0){\strut{} 0}}%
      \csname LTb\endcsname%
      \put(1765,-110){\makebox(0,0){\strut{} 0.2}}%
      \csname LTb\endcsname%
      \put(2737,-110){\makebox(0,0){\strut{} 0.4}}%
      \csname LTb\endcsname%
      \put(3710,-110){\makebox(0,0){\strut{} 0.6}}%
      \csname LTb\endcsname%
      \put(4682,-110){\makebox(0,0){\strut{} 0.8}}%
      \csname LTb\endcsname%
      \put(5655,-110){\makebox(0,0){\strut{} 1}}%
      \put(22,1744){\rotatebox{90}{\makebox(0,0){\strut{}$\sigma_{max}$}}}%
      \put(3223,-440){\makebox(0,0){\strut{}Time}}%
    }%
    \gplgaddtomacro\gplfronttext{%
      \csname LTb\endcsname%
      \put(4668,3205){\makebox(0,0)[r]{\strut{}with DeTurck}}%
      \csname LTb\endcsname%
      \put(4668,2985){\makebox(0,0)[r]{\strut{}without DeTurck}}%
    }%
    \gplbacktext
    \put(0,0){\includegraphics{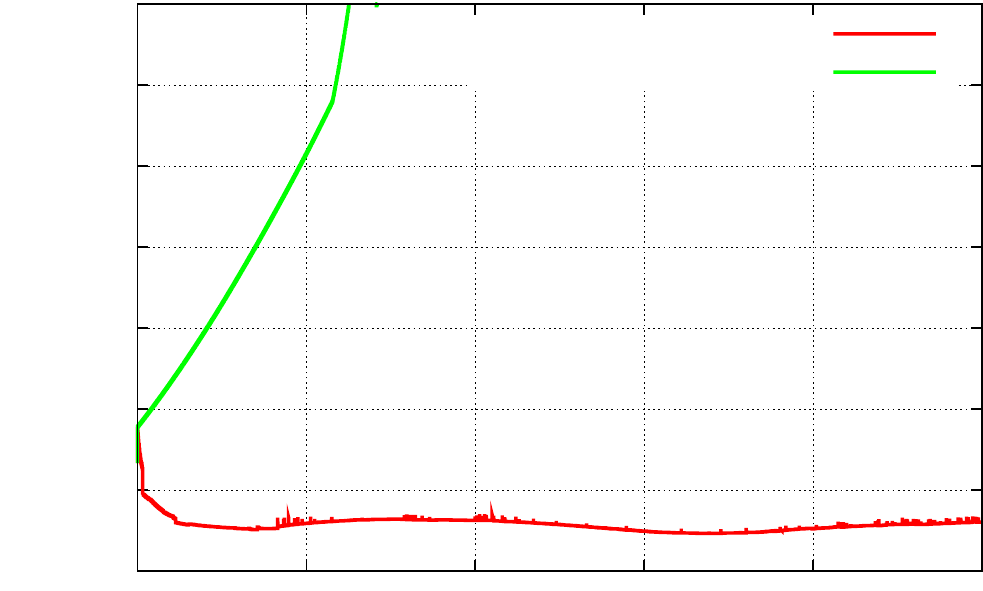}}%
    \gplfronttext
  \end{picture}%
  \vspace*{10mm}
  \caption{Mesh quality $\sigma_{max}$, see (\ref{definition_sigma_max}), 
  for the computational mesh in Figure \ref{moving_inclusion_figures}.
  The computational mesh degenerates after a short time if it is moved
  according to the velocity field defined in (\ref{velocity_field_moving_inclusion}).
  The redistribution of the mesh vertices by Algorithm \ref{algo_DeTurck}
  can prevent such mesh degenerations. See Example $2.2$ for more details.
  }
  \label{moving_inclusion_mesh_quality}
\end{center}
\end{figure}

\subsection*{Example 3: Mesh improvement for moving surfaces}
\subsubsection*{Example 3.1}
So far, we have only considered flat domains, although within our implementation
these domains were treated as hypersurfaces in $\mathbb{R}^3$. 
In this example, we study the behaviour of the DeTurck scheme for a curved surface $\Gamma(t) \subset \mathbb{R}^3$
that evolves according to (\ref{equation_of_motion_non-reparametrized}) with   
\begin{align}
 v(x_1,x_2,x_3,t) = (1.5 x_1 x_3, 0.5 x_2 x_3, r^2 \sin(4 \varphi) \cos(\pi t/2))^T,
\label{velocity_field_moving_surface}
\end{align}
where $(r, \varphi) \in [0,\infty) \times [0,2\pi)$ are such that $(x_1, x_2) = (r \cos \varphi, r \sin \varphi)$.
The initial surface is the unit disk embedded into $\mathbb{R}^3$,
that is
$$
	\Gamma(0) := \{ x \in \mathbb{R}^3 \; | \; x_1^2 + x^2_2 \leq 1 \; \textnormal{and} \; x_3 = 0 \},
$$
see Figure \ref{moving_surfaces_initial_mesh}.
The computational parameters for the simulation are $\tau = 0.02~h_{min}^2$, $\alpha = 1.0$ and $T_{adapt} = 10^{-3}$.
We first observe that the mesh moved according to Algorithm \ref{algo_DeTurck} properly approximates
the shape of $\Gamma(t)$ on the time interval $[0,0.8]$, see Figures \ref{moving_surfaces_initial_mesh} 
to \ref{moving_surfaces_with_DeT_at_time_0_80_zoom_2}. In Figure \ref{moving_surface_figures}
the red area always shows the shape that is computed without the use of Algorithm \ref{algo_DeTurck}.
Furthermore, Figure \ref{Ex3_moving_surfaces_mesh_quality} once again shows that the DeTurck scheme preserves the mesh quality.
\begin{figure}
\begin{center}
\subfloat[][\centering Initial mesh.]
{\includegraphics[width=0.266\textwidth]{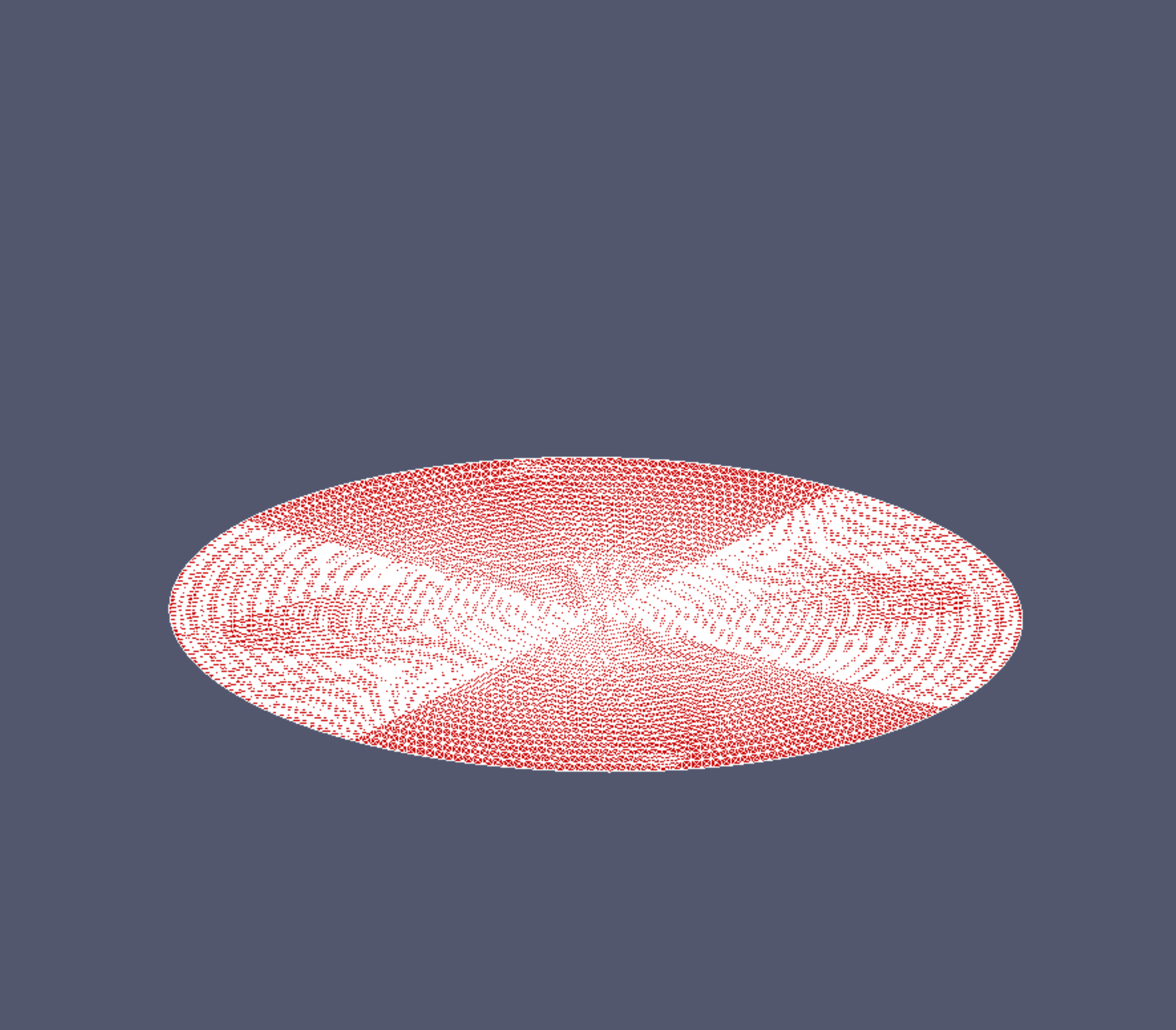}
\label{moving_surfaces_initial_mesh}
} 
~~
\subfloat[][\centering Computational mesh at $t=0.4$ with DeTurck redistribution.]
{\includegraphics[width=0.266\textwidth]{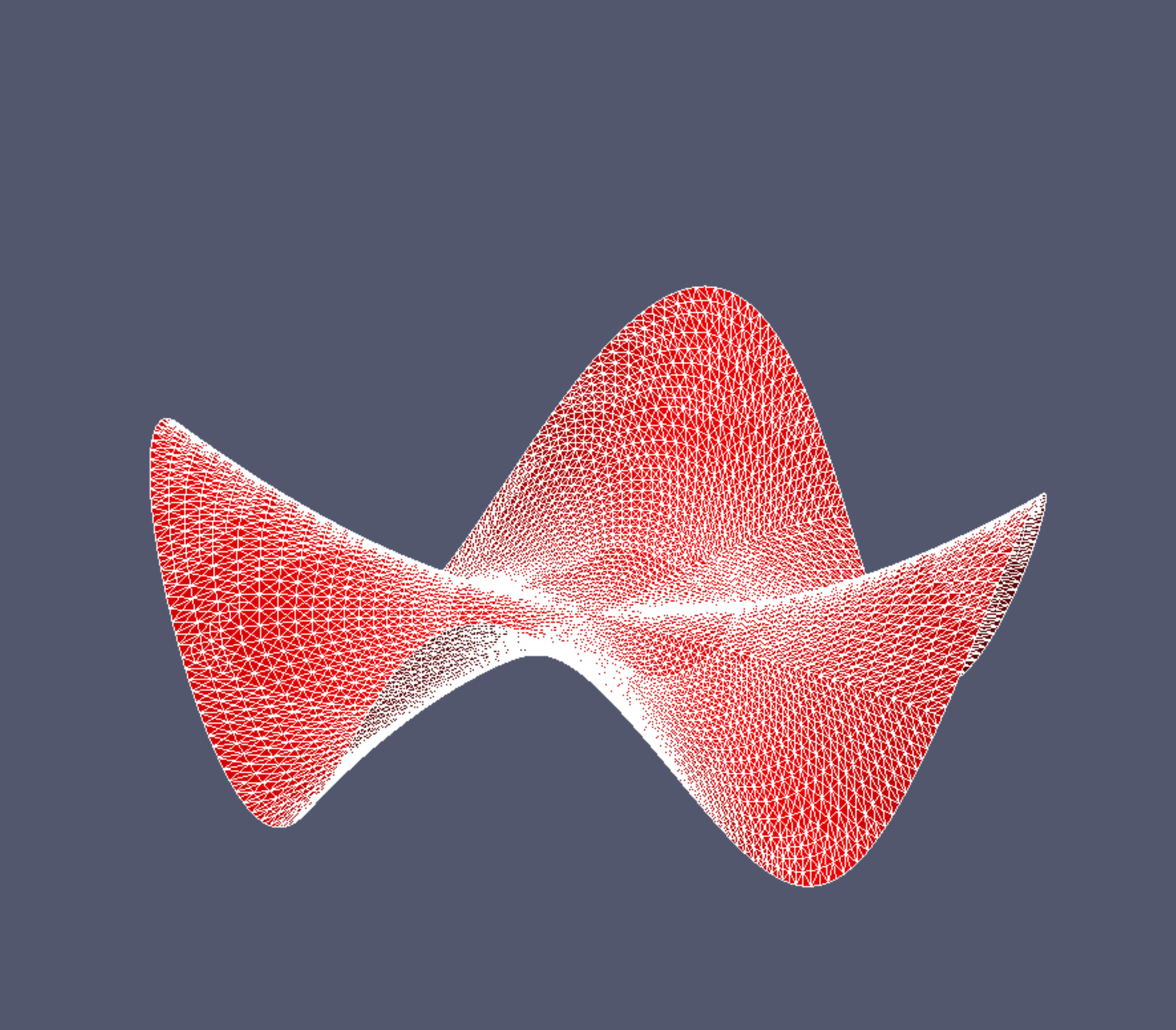}
\label{moving_surfaces_with_DeT_at_time_0_40}
}
~~
\subfloat[][\centering Computational mesh at $t=0.8$ with DeTurck redistribution.]
{\includegraphics[width=0.266\textwidth]{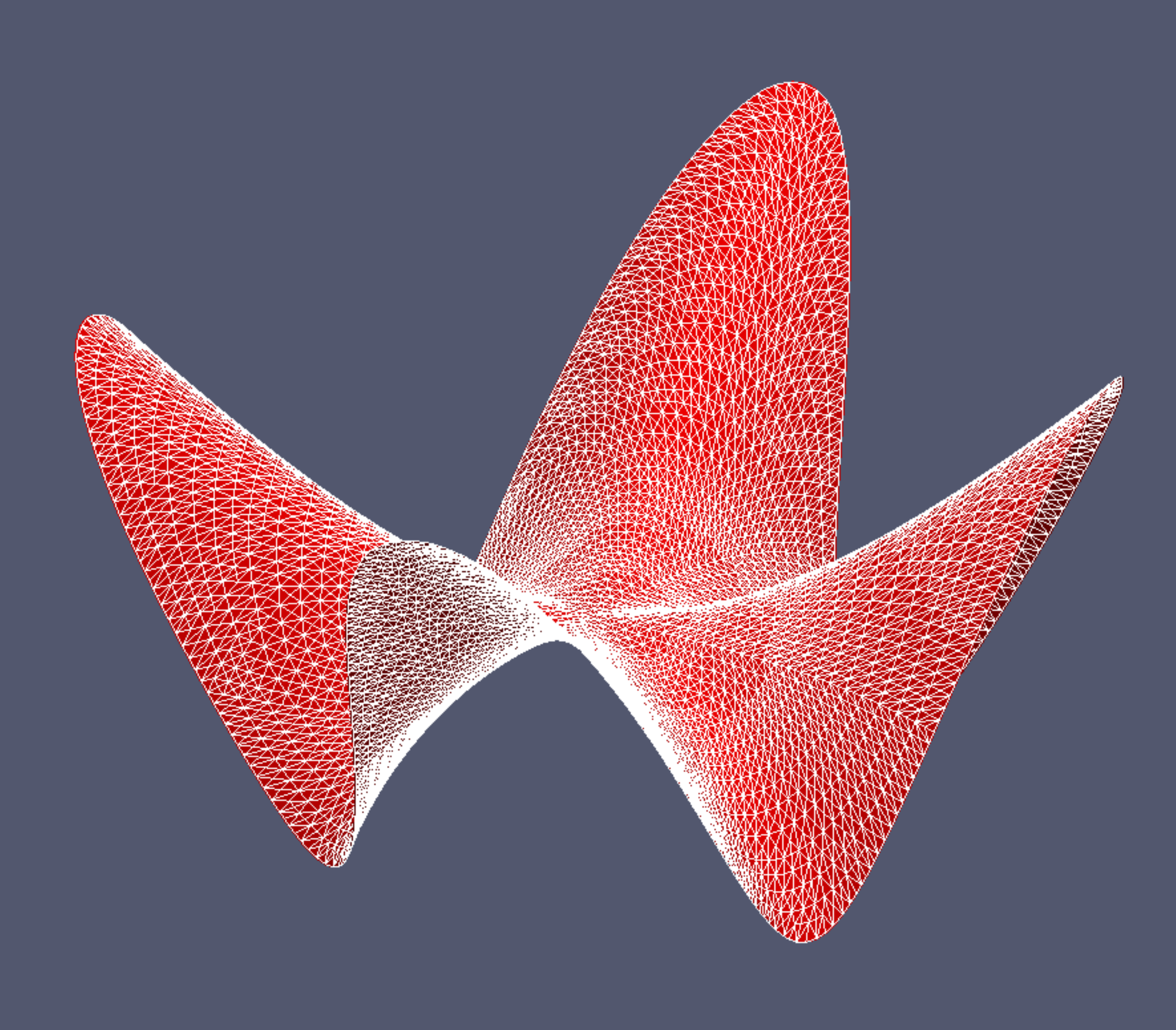}
\label{moving_surfaces_with_DeT_at_time_0_80}
} 
\\ 
\subfloat[][\centering Computational mesh at $t=0.8$ with DeTurck redistribution (Zoom).]
{\includegraphics[width=0.266\textwidth]{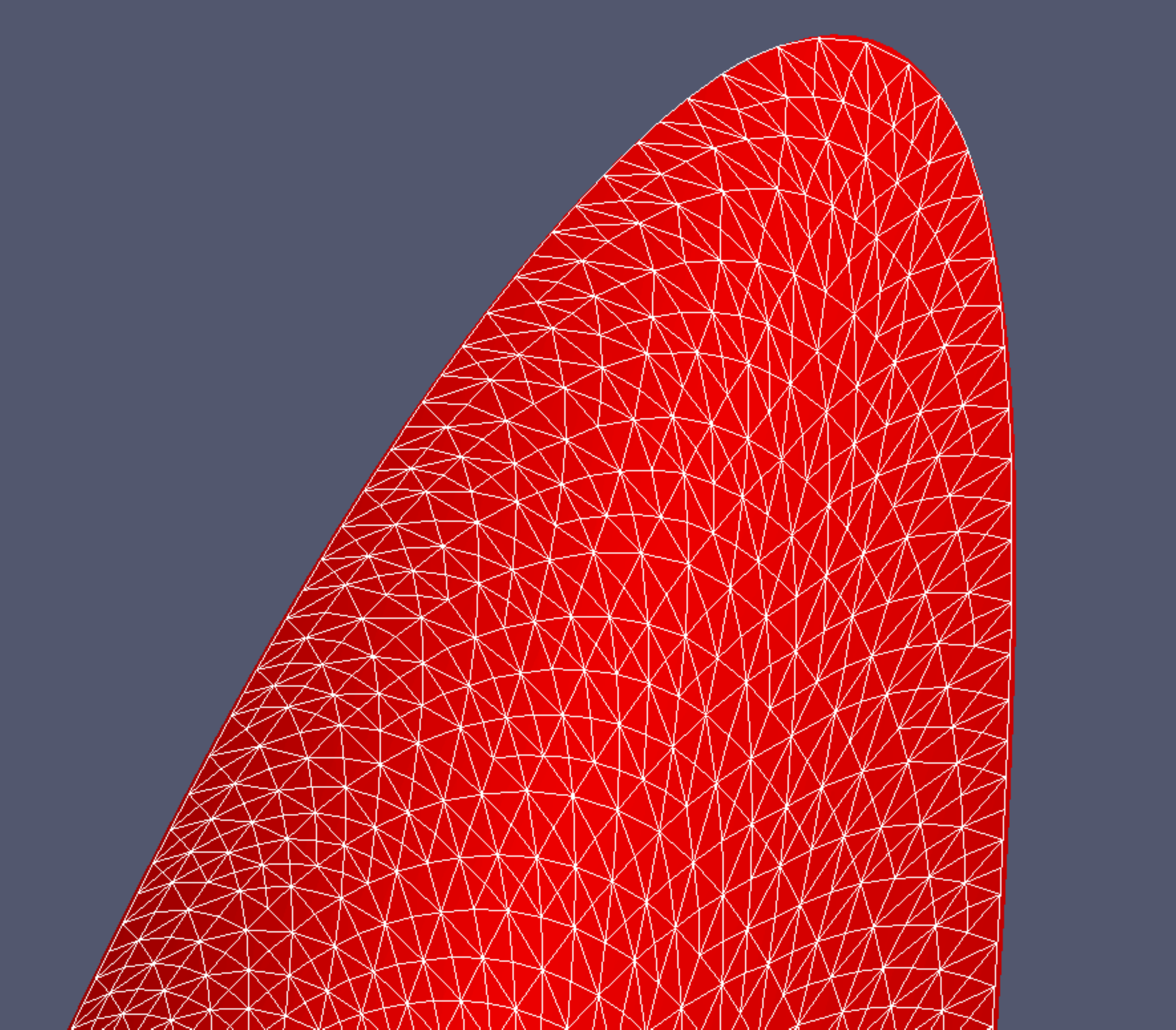}
\label{moving_surfaces_with_DeT_at_time_0_80_zoom}
} 
~~
\subfloat[][\centering Computational mesh at $t=0.8$ with DeTurck redistribution (Zoom).]
{\includegraphics[width=0.266\textwidth]{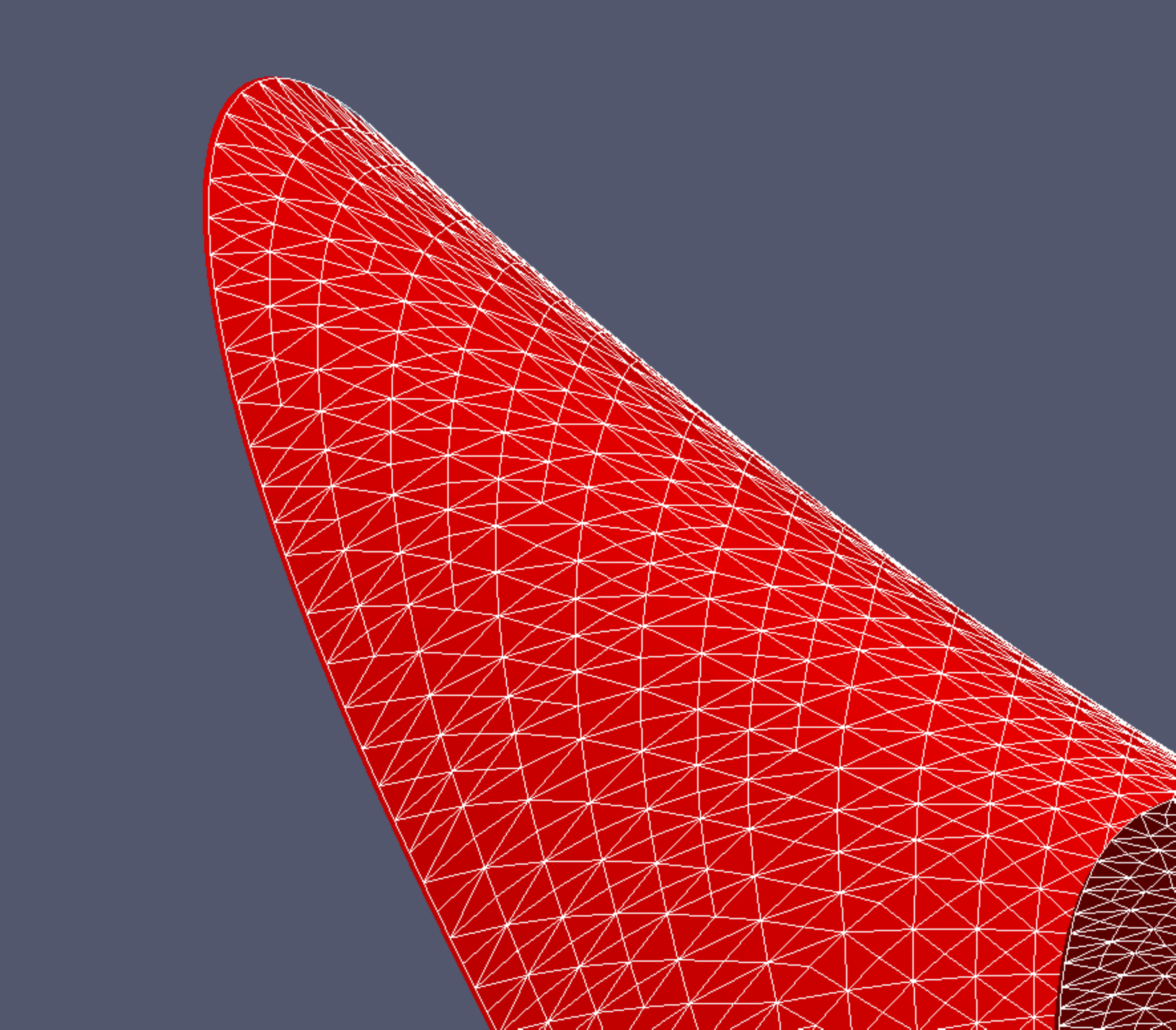}
\label{moving_surfaces_with_DeT_at_time_0_80_zoom_2}
} 
\\
\subfloat[][\centering Computational mesh at $t=0.8$ without DeTurck redistribution (Zoom).]
{\includegraphics[width=0.266\textwidth]{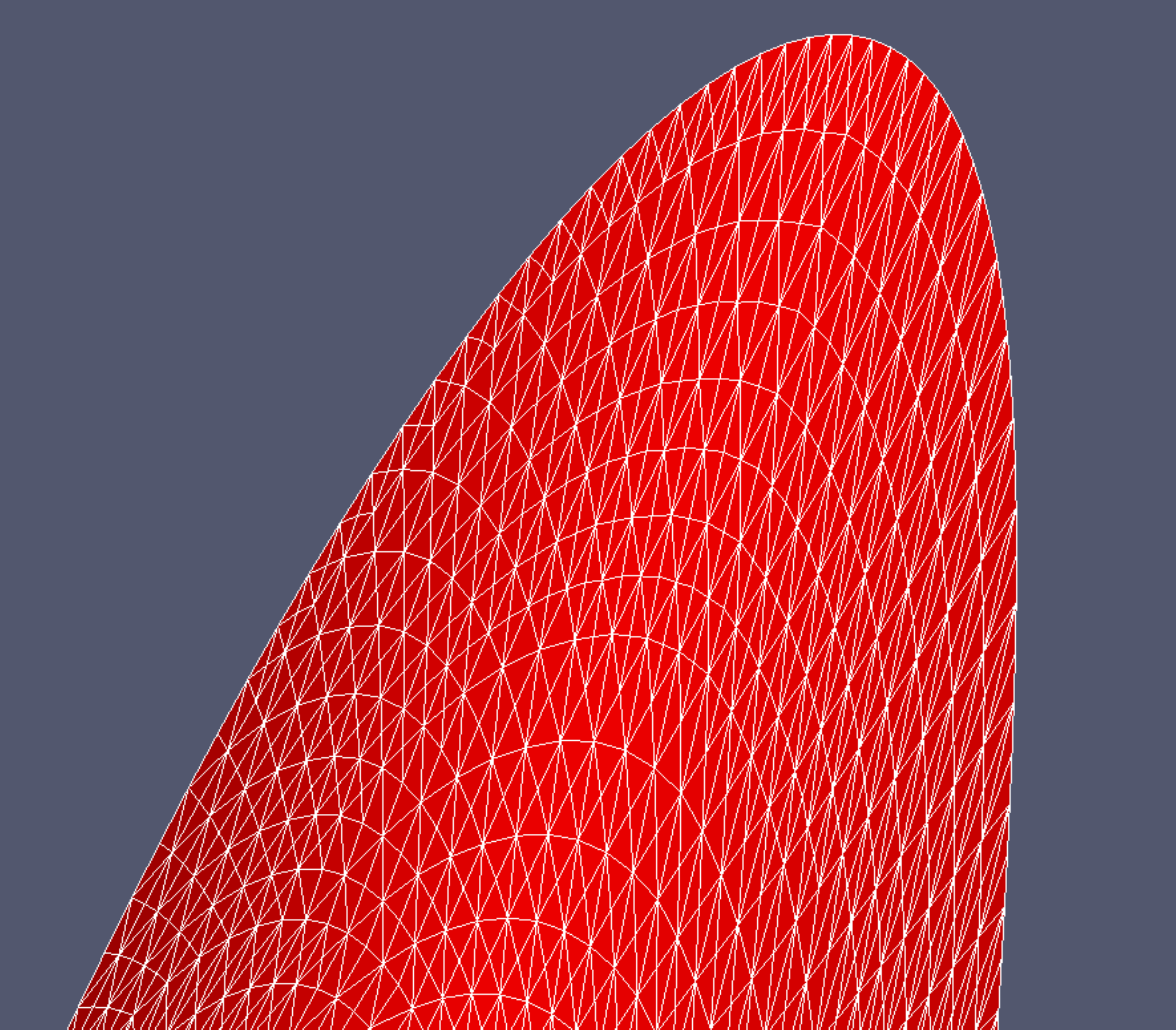}
\label{moving_surfaces_without_DeT_at_time_0_80_zoom}
} 
~~ 
\subfloat[][\centering Computational mesh at $t=0.8$ without DeTurck redistribution (Zoom).]
{\includegraphics[width=0.266\textwidth]{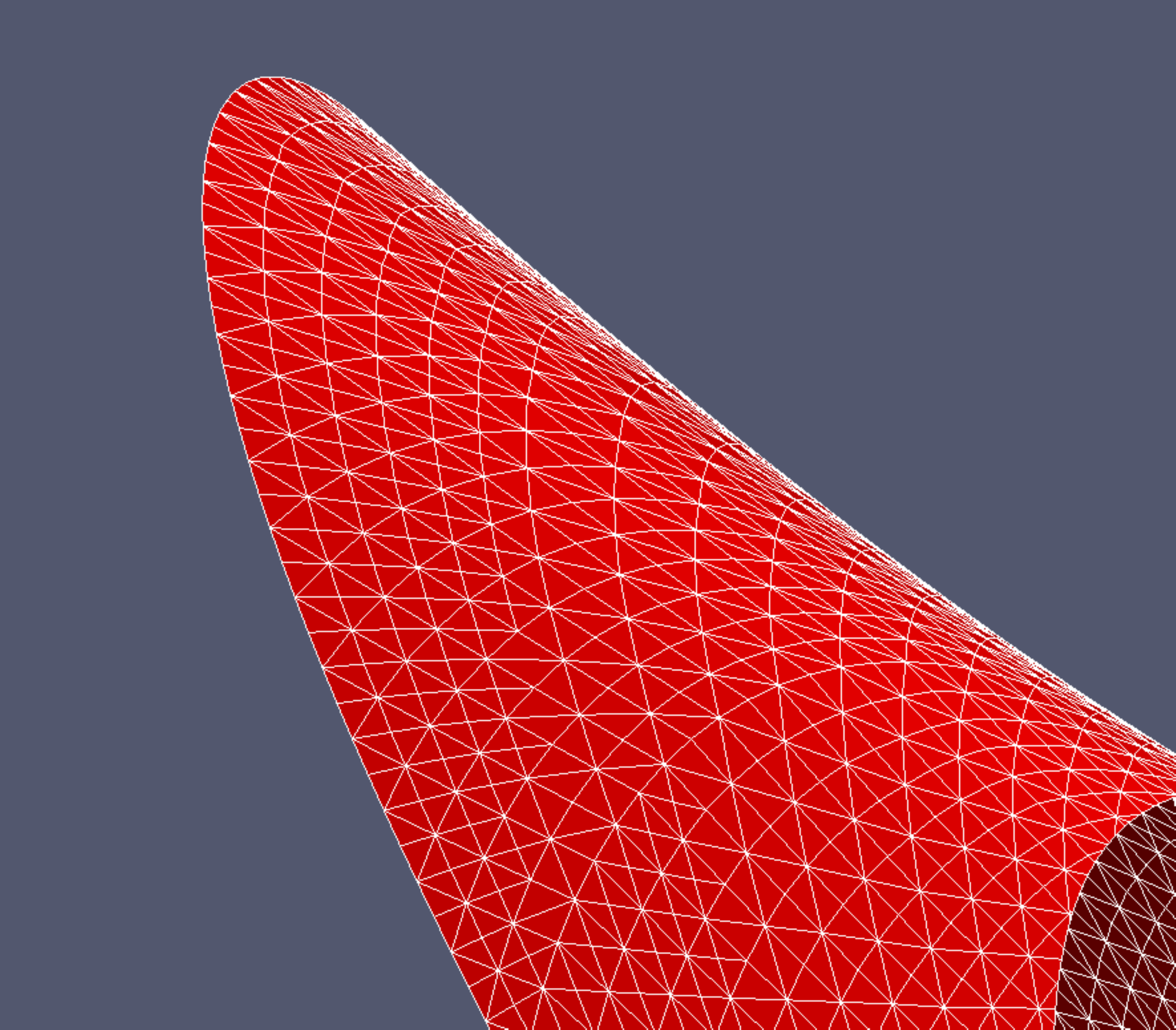}
\label{moving_surfaces_without_DeT_at_time_0_80_zoom_2}
} 
\caption{Comparison of the mesh behaviour for a moving surface in $\mathbb{R}^3$.
The initial surface is shown in Figure \ref{moving_surfaces_initial_mesh}. The surface is deformed 
according to the velocity field in (\ref{velocity_field_moving_surface}).
Figures \ref{moving_surfaces_with_DeT_at_time_0_40} to \ref{moving_surfaces_with_DeT_at_time_0_80_zoom_2}
show the computational mesh obtained by Algorithm \ref{algo_DeTurck}, while
Figures \ref{moving_surfaces_without_DeT_at_time_0_80_zoom} and \ref{moving_surfaces_without_DeT_at_time_0_80_zoom_2}
show the mesh without redistribution of mesh vertices. Figure \ref{Ex3_moving_surfaces_mesh_quality} shows 
that in Algorithm \ref{algo_DeTurck}, the mesh quality is not affected by the surface deformation. 
See Example $3.1$ for more details.
}
\label{moving_surface_figures}
\end{center}
\end{figure}
\gdef\gplbacktext{}%
\gdef\gplfronttext{}%
\begin{figure}
\begin{center}
\begin{picture}(5668.00,3400.00)%
    \gplgaddtomacro\gplbacktext{%
      \csname LTb\endcsname%
      \put(660,110){\makebox(0,0)[r]{\strut{} 5}}%
      \csname LTb\endcsname%
      \put(660,764){\makebox(0,0)[r]{\strut{} 10}}%
      \csname LTb\endcsname%
      \put(660,1417){\makebox(0,0)[r]{\strut{} 15}}%
      \csname LTb\endcsname%
      \put(660,2071){\makebox(0,0)[r]{\strut{} 20}}%
      \csname LTb\endcsname%
      \put(660,2724){\makebox(0,0)[r]{\strut{} 25}}%
      \csname LTb\endcsname%
      \put(660,3378){\makebox(0,0)[r]{\strut{} 30}}%
      \csname LTb\endcsname%
      \put(792,-110){\makebox(0,0){\strut{} 0}}%
      \csname LTb\endcsname%
      \put(1400,-110){\makebox(0,0){\strut{} 0.1}}%
      \csname LTb\endcsname%
      \put(2008,-110){\makebox(0,0){\strut{} 0.2}}%
      \csname LTb\endcsname%
      \put(2616,-110){\makebox(0,0){\strut{} 0.3}}%
      \csname LTb\endcsname%
      \put(3224,-110){\makebox(0,0){\strut{} 0.4}}%
      \csname LTb\endcsname%
      \put(3831,-110){\makebox(0,0){\strut{} 0.5}}%
      \csname LTb\endcsname%
      \put(4439,-110){\makebox(0,0){\strut{} 0.6}}%
      \csname LTb\endcsname%
      \put(5047,-110){\makebox(0,0){\strut{} 0.7}}%
      \csname LTb\endcsname%
      \put(5655,-110){\makebox(0,0){\strut{} 0.8}}%
      \put(22,1744){\rotatebox{90}{\makebox(0,0){\strut{}$\sigma_{max}$}}}%
      \put(3223,-440){\makebox(0,0){\strut{}Time}}%
    }%
    \gplgaddtomacro\gplfronttext{%
      \csname LTb\endcsname%
      \put(4668,3205){\makebox(0,0)[r]{\strut{}with DeTurck}}%
      \csname LTb\endcsname%
      \put(4668,2985){\makebox(0,0)[r]{\strut{}without DeTurck}}%
    }%
    \gplbacktext
    \put(0,0){\includegraphics{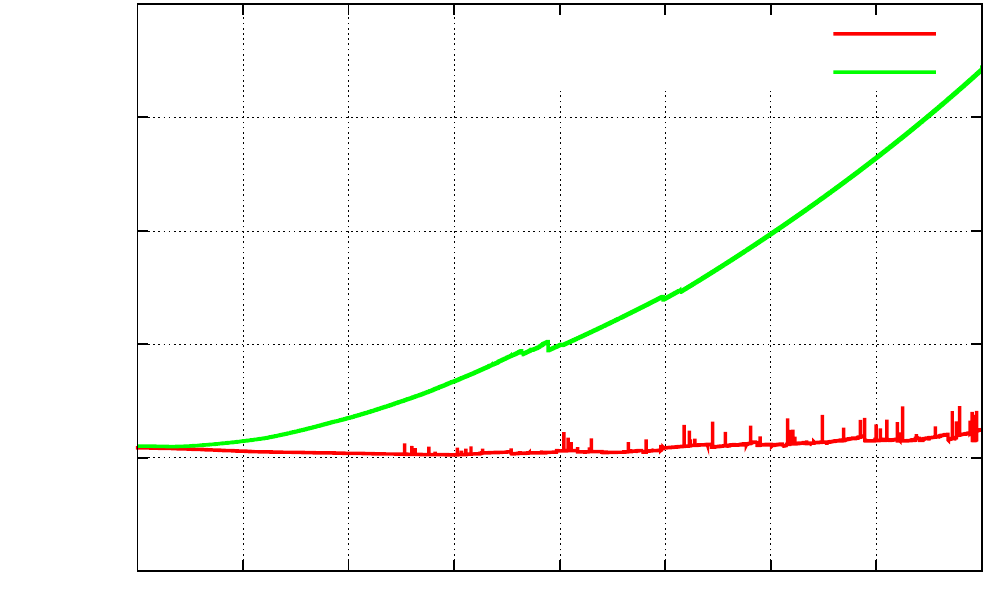}}%
    \gplfronttext
  \end{picture}%
  \vspace*{10mm}
  \caption{Mesh quality $\sigma_{max}$, see (\ref{definition_sigma_max}), for the computational mesh in Figure
  \ref{moving_surface_figures}. If Algorithms \ref{algo_DeTurck} and \ref{algo_refinement_and_coarsening_strategy}
  are applied, the surface deformation has almost no effect on the quality of the computational mesh
  in the time interval $[0,0.8]$. See Example $3.1$ for more details.}
  \label{Ex3_moving_surfaces_mesh_quality}
\end{center}
\end{figure}

\subsubsection*{Example 3.2}
So far, we have only used the DeTurck scheme to produce a nice mesh for a moving surface with a known velocity field.
We will now demonstrate that the redistribution of the mesh vertices is indeed very useful in order to
solve PDEs on evolving surfaces. 
We will couple the evolution of the surface to a PDE by
considering the mean curvature flow, that is 
$x_t = - (H \vec{n}) \circ x$. Here, $H := \nabla_{\Gamma} \cdot \vec{n}$ 
is the sum of the principal curvatures,
that is the mean curvature, and $\vec{n}$ denotes a unit normal to $\Gamma(t)$.
Note that the mean curvature flow does not depend on the choice of $\vec{n}$. 
It is well known, that this evolution equation is equivalent to the heat equation
$$
	x_t = \Delta_{g(t)} x.
$$
We will here consider Dirichlet boundary conditions. 
The initial shape $\Gamma(0) := x_0(\M)$ is shown in
Figure \ref{MCF_BC_initial_mesh}. It can be parametrized by
$$
	X(r, \varphi) = ( r  (1 + \tfrac{1}{4} \sin(4 \varphi)) \cos \varphi,
					  r  (1 + \tfrac{1}{4} \sin(4 \varphi)) \sin \varphi, 
					  \tfrac{1}{4} r^2 \sin(4 \varphi) + \tfrac{3}{4} (1 - r^2))^T
$$ 
with $(r,\varphi) \in [0,1] \times [0, 2\pi)$.
In order to compute this flow without mesh redistribution, we determine the solutions $u_h^{m+1} \in V_h(\Gamma_h^m)^3$ of
\begin{align}
 \int_{\Gamma_h^m} \tfrac{1}{\tau} I_h (u^{m+1}_h \cdot \varphi_h) + \nabla_{\Gamma_h^m} u^{m+1}_h : \nabla_{\Gamma_h^m} \varphi_h \; do
 = \int_{\Gamma_h^m} \tfrac{1}{\tau} I_h(\tilde{u}_h^m \cdot \varphi_h) \; do
 \label{scheme_MCF}
\end{align}
for all $\varphi_h \in \overset{\circ}{V}_h(\Gamma_h^m)^3$ with $u_h^{m+1} = \tilde{u}_h^m$ on $\partial \Gamma_h^m$. 
This scheme is a variation of the scheme proposed in \cite{Dz91}.
For the simulation with DeTurck redistribution, we modify Algorithm \ref{algo_DeTurck} in the following way: 
We replace the first equation in (\ref{equation_for_u_h}) by
\begin{align*}
&
	\int_{\Gamma_h^m} \tfrac{1}{\tau} I_h (u^{m+1}_{h} \cdot \varphi_h)  
		 + \nabla_{\Gamma_h^m} u^{m+1}_h : \nabla_{\Gamma_h^m} \varphi_h
		 + \tfrac{1}{\alpha } \sum_{\sigma, \kappa = 1}^3 
		 ((\hat{H}^m_h)^{-1} \nabla_{{\Gamma}_h^m} \hat{y}_{h, \sigma}^{m})_{\kappa} 
		 I_h \big({\tilde{\zeta}}^{m,\sigma}_h \varphi_h^\kappa \big)\; do
	\\	 
	& \qquad = \int_{\Gamma_h^m} \tfrac{1}{\tau} I_h (\tilde{u}^{m}_h \cdot \varphi_h) \; do,
	\quad \forall \varphi_h \in \overset{\circ}{V}_h({\Gamma}_h^m)^{3}.
\end{align*}
The second equation in (\ref{equation_for_u_h}) is not changed. This is our new DeTurck scheme for the computation of the mean curvature flow
with Dirichlet boundary conditions; see also the recent paper \cite{EF15} on the mean curvature-DeTurck flow on closed manifolds.
The parameters used for Figure \ref{MCF_figures} were $\tau = 0.01~h_{min}^2$, $\alpha = 1.0$, and $T_{adapt} = 10^{-3}$.
The numerical results in Figures \ref{MCF_figures} and \ref{Ex3_MCF_BC_mesh_quality} show that
without the redistribution of the mesh points the computational mesh totally degenerates.
In contrast, the mesh quality under the DeTurck scheme is preserved.
\begin{remark}
In the above example, we have considered the motion of a surface with fixed boundary.
For this class of problems a variant of the approach presented in this paper might be more suitable.
Reparametrizing the evolution equations by solutions to the harmonic map heat flow with Dirichlet boundary conditions
seems to be more natural if the boundary is fixed. However, it is not clear whether a scheme based on 
such a reparametrization is able to preserve the mesh quality in the interior of the surface. 
This problem has to be studied elsewhere.
\end{remark}

\begin{figure}
\begin{center}
\subfloat[][\centering Initial mesh.]
{\includegraphics[width=0.266\textwidth]{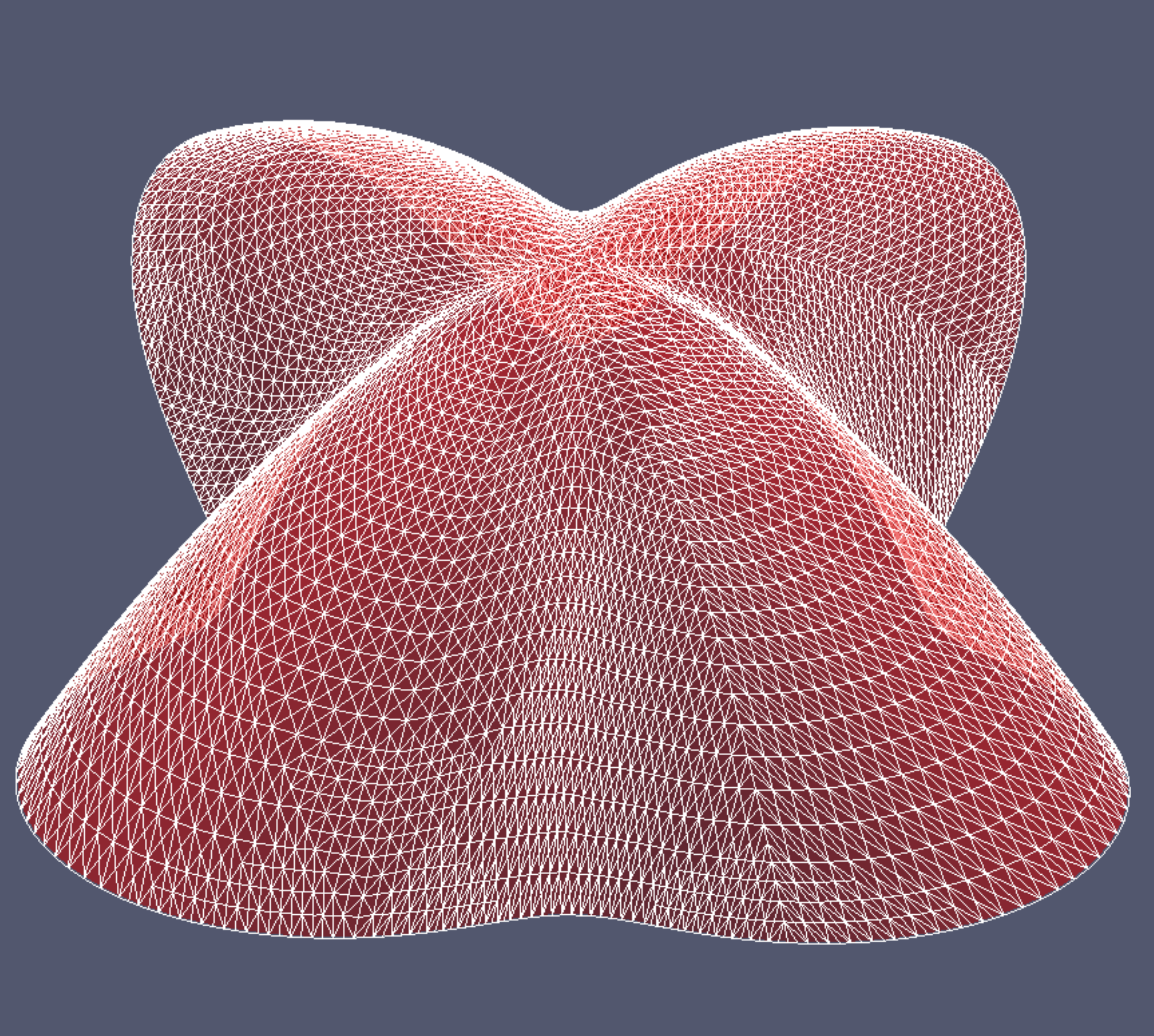}
\label{MCF_BC_initial_mesh}
} 
~~
\subfloat[][\centering Computational mesh at $t=1.0$ without redistribution.]
{\includegraphics[width=0.266\textwidth]{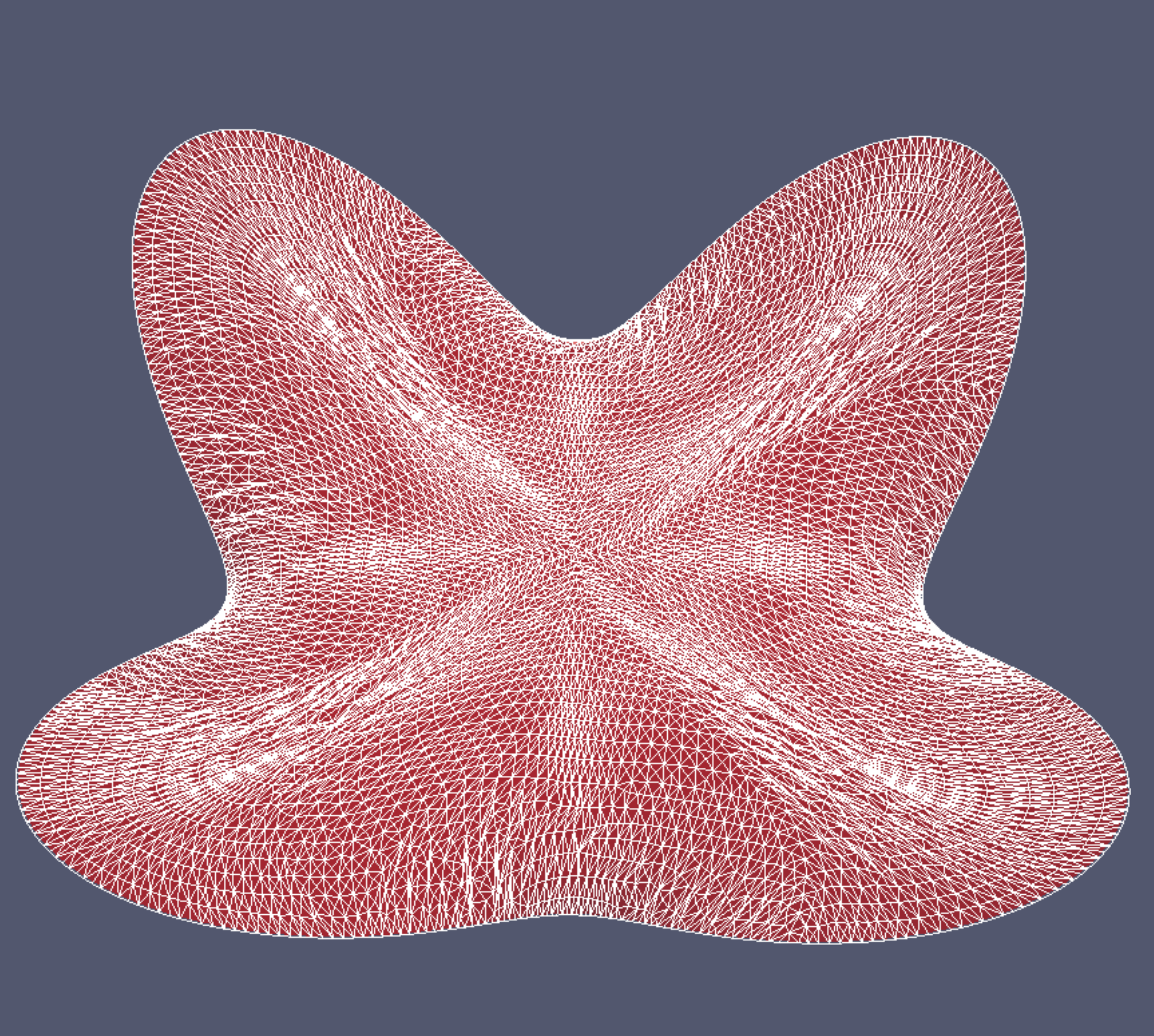}
\label{MCF_BC_without_DeT_at_time_1_00}
} 
~~
\subfloat[][\centering Computational mesh at $t=1.0$ with DeTurck redistribution.]
{\includegraphics[width=0.266\textwidth]{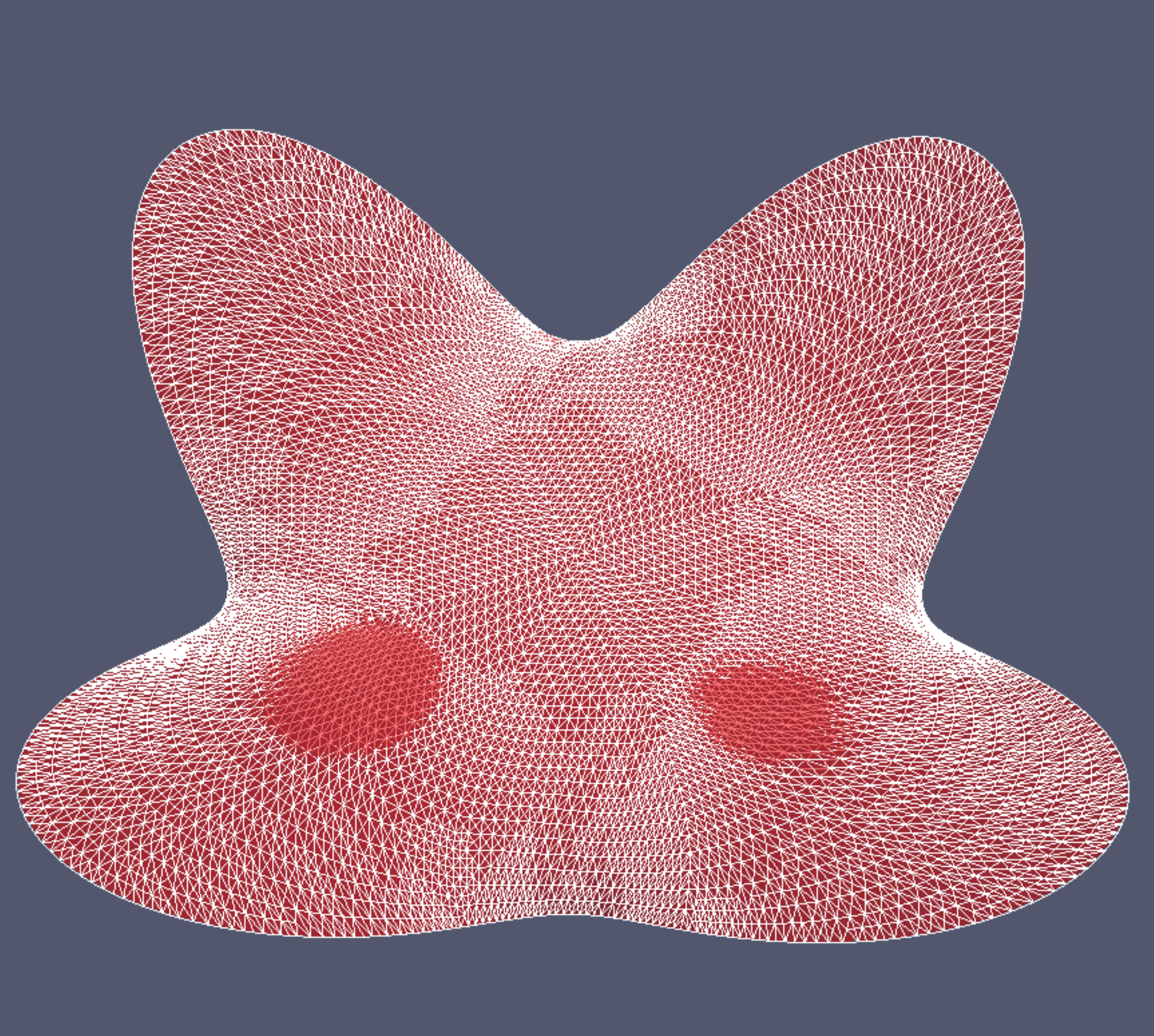}
\label{MCF_BC_with_DeT_at_time_1_00}
}
\\
\subfloat[][\centering Computational mesh at $t=1.0$ without redistribution (Zoom).]
{\includegraphics[width=0.266\textwidth]{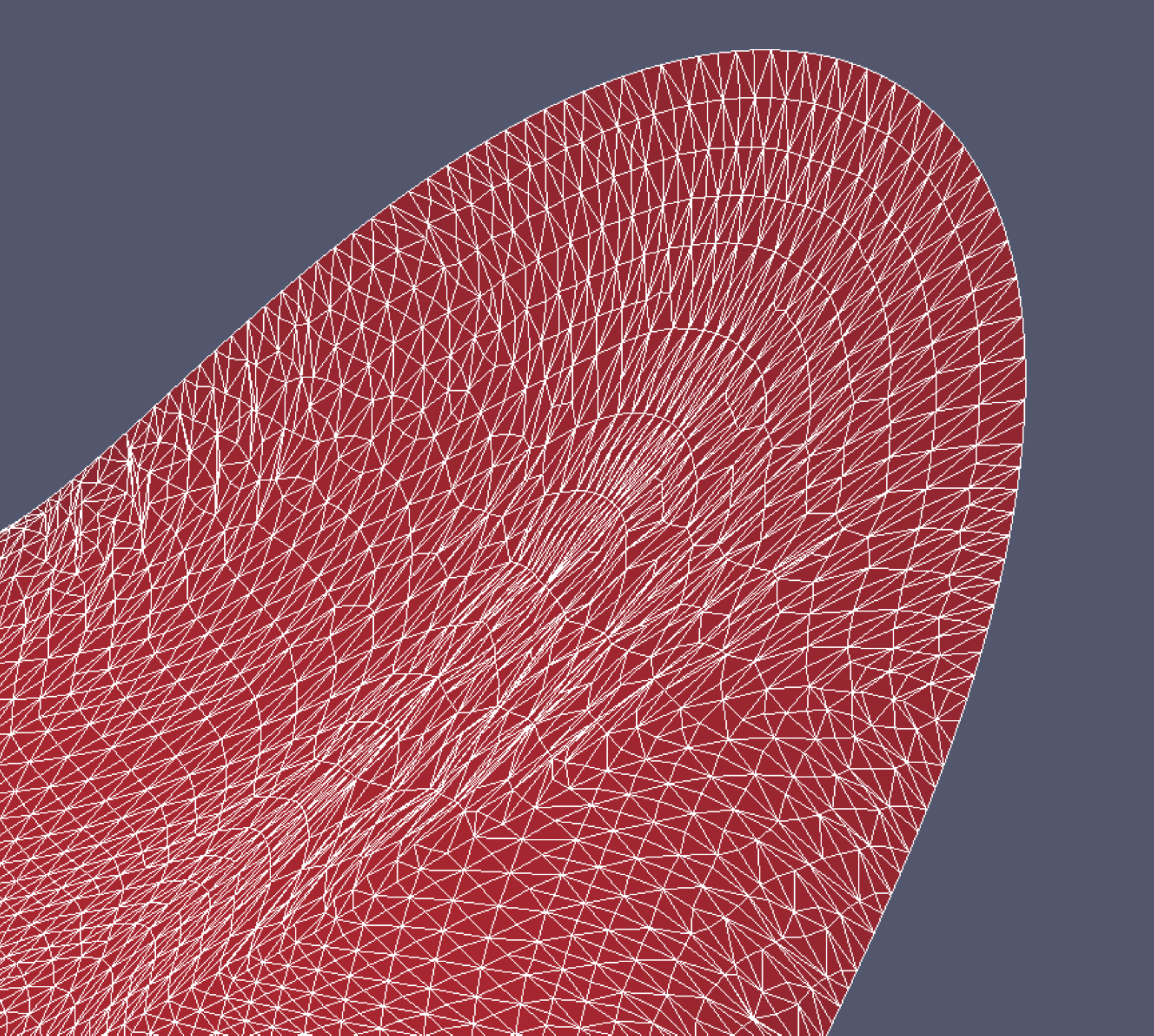}
\label{MCF_BC_without_DeT_at_time_1_00_zoom}
} 
~~
\subfloat[][\centering Computational mesh at $t=1.0$ with DeTurck redistribution (Zoom).]
{\includegraphics[width=0.266\textwidth]{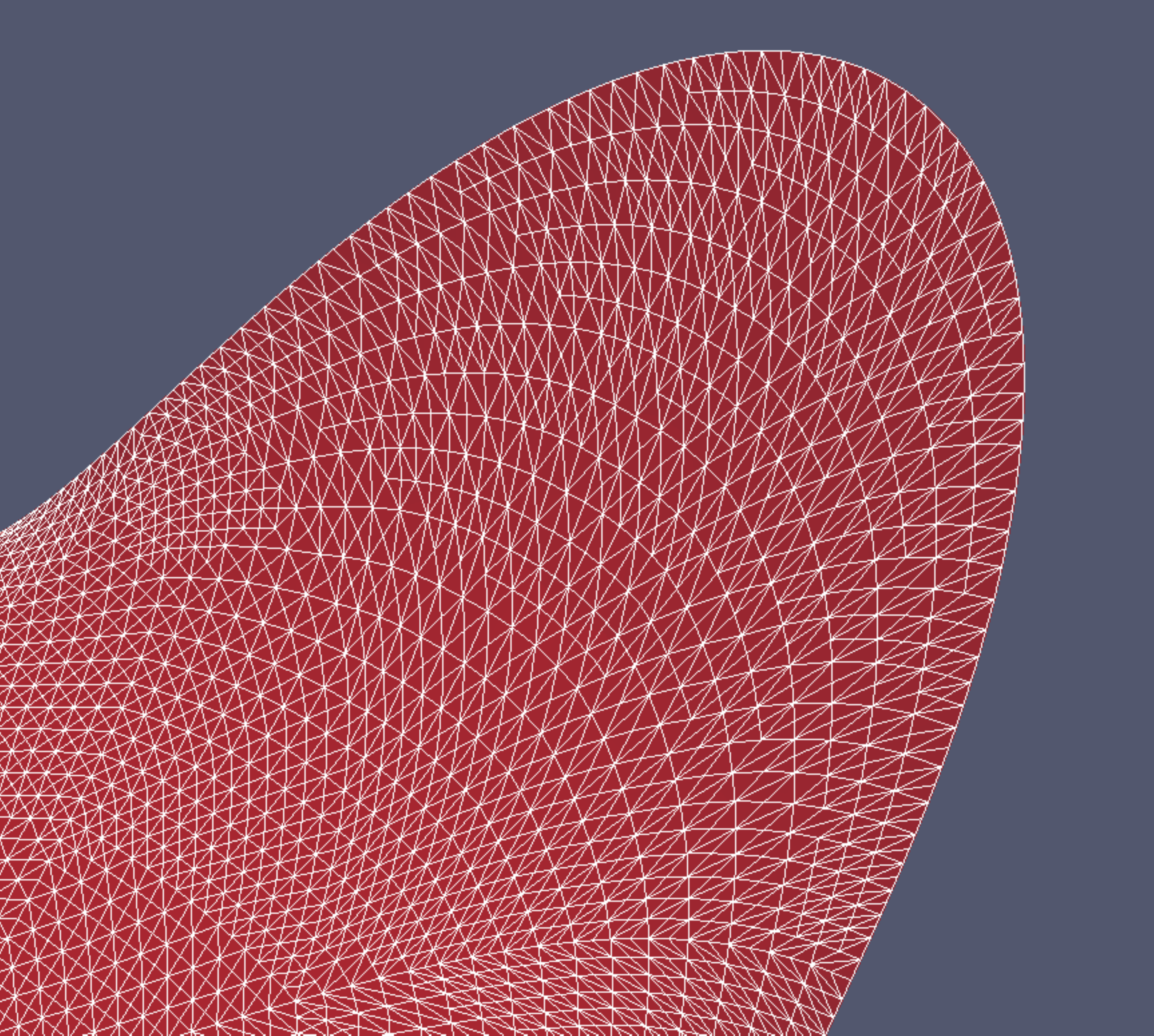}
\label{MCF_BC_with_DeT_at_time_1_00_zoom}
}
\caption{Comparison of the mesh behaviour for a surface that is deformed according to the mean curvature flow
  with Dirichlet boundary conditions. The initial surface is presented in Figure \ref{MCF_BC_initial_mesh}.
  Without redistribution the mesh totally degenerates, see Figures \ref{MCF_BC_without_DeT_at_time_1_00}
  and \ref{MCF_BC_without_DeT_at_time_1_00_zoom}.
  By using Algorithms \ref{algo_DeTurck} and \ref{algo_refinement_and_coarsening_strategy},
  the mesh remains regular; see Figures \ref{MCF_BC_with_DeT_at_time_1_00} and \ref{MCF_BC_with_DeT_at_time_1_00_zoom}
  and Figure \ref{Ex3_MCF_BC_mesh_quality}. 
  Note that the DeTurck scheme also leads to a redistribution of the vertices at the boundary of the surface.  
  See Example $3.2$ for more details.
}
\label{MCF_figures}
\end{center}
\end{figure}
\gdef\gplbacktext{}%
\gdef\gplfronttext{}%
\begin{figure}
\begin{center}
\begin{picture}(5668.00,3400.00)%
    \gplgaddtomacro\gplbacktext{%
      \csname LTb\endcsname%
      \put(660,361){\makebox(0,0)[r]{\strut{} 20}}%
      \csname LTb\endcsname%
      \put(660,697){\makebox(0,0)[r]{\strut{} 40}}%
      \csname LTb\endcsname%
      \put(660,1032){\makebox(0,0)[r]{\strut{} 60}}%
      \csname LTb\endcsname%
      \put(660,1367){\makebox(0,0)[r]{\strut{} 80}}%
      \csname LTb\endcsname%
      \put(660,1702){\makebox(0,0)[r]{\strut{} 100}}%
      \csname LTb\endcsname%
      \put(660,2037){\makebox(0,0)[r]{\strut{} 120}}%
      \csname LTb\endcsname%
      \put(660,2372){\makebox(0,0)[r]{\strut{} 140}}%
      \csname LTb\endcsname%
      \put(660,2708){\makebox(0,0)[r]{\strut{} 160}}%
      \csname LTb\endcsname%
      \put(660,3043){\makebox(0,0)[r]{\strut{} 180}}%
      \csname LTb\endcsname%
      \put(660,3378){\makebox(0,0)[r]{\strut{} 200}}%
      \csname LTb\endcsname%
      \put(792,-110){\makebox(0,0){\strut{} 0}}%
      \csname LTb\endcsname%
      \put(1765,-110){\makebox(0,0){\strut{} 0.2}}%
      \csname LTb\endcsname%
      \put(2737,-110){\makebox(0,0){\strut{} 0.4}}%
      \csname LTb\endcsname%
      \put(3710,-110){\makebox(0,0){\strut{} 0.6}}%
      \csname LTb\endcsname%
      \put(4682,-110){\makebox(0,0){\strut{} 0.8}}%
      \csname LTb\endcsname%
      \put(5655,-110){\makebox(0,0){\strut{} 1}}%
      \put(-110,1744){\rotatebox{90}{\makebox(0,0){\strut{}$\sigma_{max}$}}}%
      \put(3223,-440){\makebox(0,0){\strut{}Time}}%
    }%
    \gplgaddtomacro\gplfronttext{%
      \csname LTb\endcsname%
      \put(4668,3205){\makebox(0,0)[r]{\strut{}with DeTurck}}%
      \csname LTb\endcsname%
      \put(4668,2985){\makebox(0,0)[r]{\strut{}without DeTurck}}%
    }%
    \gplbacktext
    \put(0,0){\includegraphics{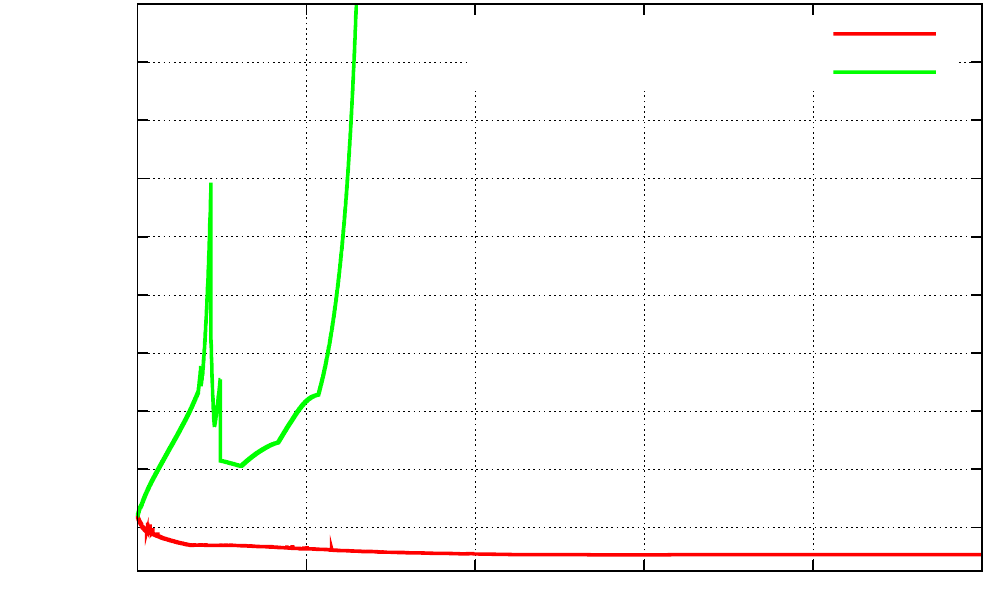}}%
    \gplfronttext
  \end{picture}%
  \vspace*{10mm}
  \caption{Mesh quality $\sigma_{max}$, see (\ref{definition_sigma_max}), for the computational mesh in Figure
  \ref{MCF_figures}. Without the redistribution of the mesh vertices induced by Algorithm \ref{algo_DeTurck}
  the motion by mean curvature leads to a strong degeneration of the computational mesh.
  See Example $3.2$ for more details.
  }
  \label{Ex3_MCF_BC_mesh_quality}
\end{center}
\end{figure}
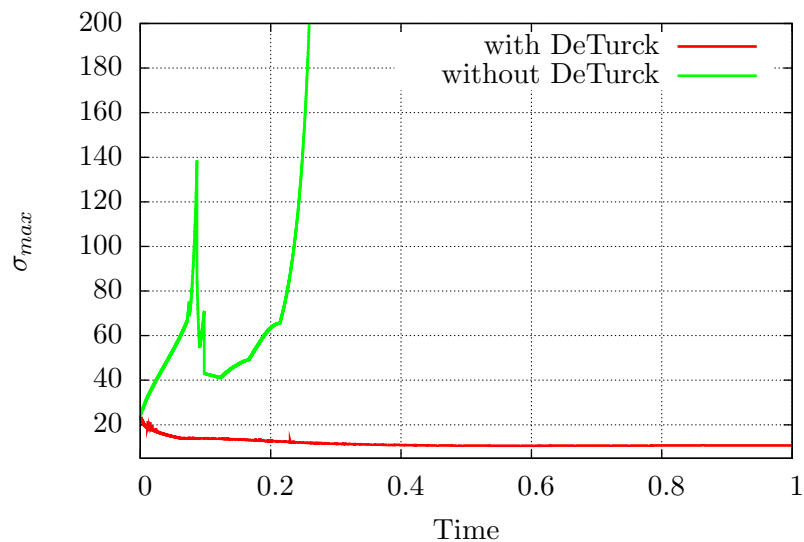

\subsection*{Example 4: A free boundary problem: The Hele-Shaw flow coupled to the DeTurck reparametrization}
Free and moving boundary problems, \cite{EllOck82},  provide a wide field of applications in which moving meshes might be useful. 
A classical well studied problem is that of Hele-Shaw flow, \cite{EllOck82, GusVas06, KelHin97}.
Let $\Gamma(t) \subset \mathbb{R}^2$ be a bounded domain in $\mathbb{R}^2$ and 
$p: \Gamma(t) \rightarrow \mathbb{R}$ the solution to the following boundary value problem
\begin{align*}
	&   \Delta p = \delta_{q}, 
		\quad \textnormal{in $\Gamma(t)$,}
	\\
	& p = \sigma \kappa, \quad \textnormal{on $\partial \Gamma(t)$,}
\end{align*}
where $\Delta$ denotes the usual Laplacian in $\mathbb{R}^2$. 
$\sigma \in (0, \infty)$ is a surface tension constant and $\kappa$ is the curvature
of the free boundary. Here, the curvature $\kappa$ is supposed to be positive 
when the domain $\Gamma(t)$ is convex. 
We choose $\Gamma(0)$ to be the unit disk with center $(0.0, -0.5) \in \mathbb{R}^2$.
There is a sink at the point $q = (0.0, 0.0) \in \Gamma(0)$. 
The co-normal velocity $v_{\partial \Gamma}$ of $\partial \Gamma(t)$ satisfies the kinematic boundary condition
\begin{align}
	v_{\partial \Gamma} = - \tfrac{1}{12} (\nabla_{\nu(t)} p) \nu(t) \quad \textnormal{on $\partial \Gamma(t)$}.
	\label{Hele_shaw_velocity_on_the_boundary}
\end{align}
In order to obtain a base velocity field $v$ for the parametrization of $\Gamma(t)$, we impose $v=v_{\partial \Gamma}$ on $\partial \Gamma(t)$ and
\begin{equation}
  \Delta v = 0 \quad \textnormal{in $\Gamma(t)$,}
  \label{Hele_Shaw_velocity}
\end{equation}
instead of taking the physical velocity $- \tfrac{1}{12} \nabla p$  which is singular at the sink.
In order to solve the Hele-Shaw flow, we determine the solution $\tilde{p}^m_h \in V_h(\Gamma_h^m)$ of
\begin{align*}
&	\int_{\Gamma^m_h} \nabla \tilde{p}^m_h \cdot \nabla \varphi_h \; do
	= 0, \quad \textnormal{$\forall \varphi_h \in \overset{\circ}{V}_h(\Gamma_h^m)$,}
\\
&	\int_{\partial \Gamma^m_h} I_h( \tilde{p}^m_h \psi_h ) \; do
	= \sigma \int_{\partial \Gamma^m_h} 
		\nabla_{\partial \Gamma^m_h} id_{|\partial \Gamma^m_h} :
		\nabla_{\partial \Gamma^m_h} I_h (\nu_h^m \psi_h) \; do 
		- \int_{\partial \Gamma^m_h} \ I_h (G_q \psi_h) \; do,
\end{align*}
for all $\psi_h \in V_h(\partial \Gamma_h^m)$,
where $G_q(x) = \tfrac{1}{2\pi} \log(|x - q|)$. The vector field $\nu_h^m \in V_h(\Gamma_h^m)^n$ is defined in each vertex $p_j \in \partial \Gamma_h^m$ 
to be the normalized sum of the two outwards co-normals associated with the two adjacent boundary simplices of $p_j$. 
We compute an approximation
$v_h^m \in V_h(\Gamma^m_h)$ to the velocity field $v$ defined in
(\ref{Hele_shaw_velocity_on_the_boundary}) and (\ref{Hele_Shaw_velocity}) by
\begin{align*}
&	\int_{\Gamma^m_h} \nabla v^m_h : \nabla \varphi_h \; do
	= 0, \quad \textnormal{$\forall \varphi_h \in \overset{\circ}{V}_h(\Gamma_h^m)^n$,}
\\
&	\int_{\partial \Gamma^m_h} I_h( v_h^m \cdot \psi_h ) \; do
	= - \tfrac{1}{12} \int_{\partial \Gamma^m_h} 
	     \nabla \tilde{p}^m_h \cdot I_h ( \nu^m_h (\nu_h^m \cdot \psi_h))
	     + I_h( \nabla G_q \cdot \nu^m_h (\nu^m_h \cdot \psi_h))
		 \; do,
\end{align*}
for all $\psi_h \in V_h(\partial \Gamma_h^m)^2$. This gives us a base velocity $v_h^m$ for the motion of $\Gamma_h^m$.
We use this vector field in (\ref{equation_for_u_h}) of Algorithm \ref{algo_DeTurck}.
In Figures \ref{Hele_Shaw_figures} and \ref{hele_shaw_mesh_quality} we compare the simulation
based on Algorithm \ref{algo_DeTurck} to the method when the mesh vertices are just moved by 
$p_j = p_j + \tau v_h^m(p_j)$.
In both approaches we use Algorithm \ref{algo_refinement_and_coarsening_strategy} 
as mesh refinement and coarsening strategy again. 
The parameters used for the simulation were 
$\tau = 0.005 ~ h_{min}^2$, $\alpha= 1.0$, $T_{adapt} = 0.01$ and $\sigma = 10^{-3}$. 

We observe a far better mesh behaviour of the approach based
on the DeTurck reparametrization. Furthermore, the numerical solutions of both approaches
strongly differ for times $t \gtrsim 5.0$.
Since the solution based on the scheme without DeTurck redistribution has sharp corners at the boundary,
it must be rejected. In contrast, the solution obtained by Algorithm \ref{algo_DeTurck} satisfies the 
theoretical expectations -- including the formation of a cusp close to the sink.
\begin{figure}
\begin{center}
\subfloat[][\centering Computational mesh at $t=4.0$ with DeTurck redistribution.]
{\includegraphics[width=0.2\textwidth]{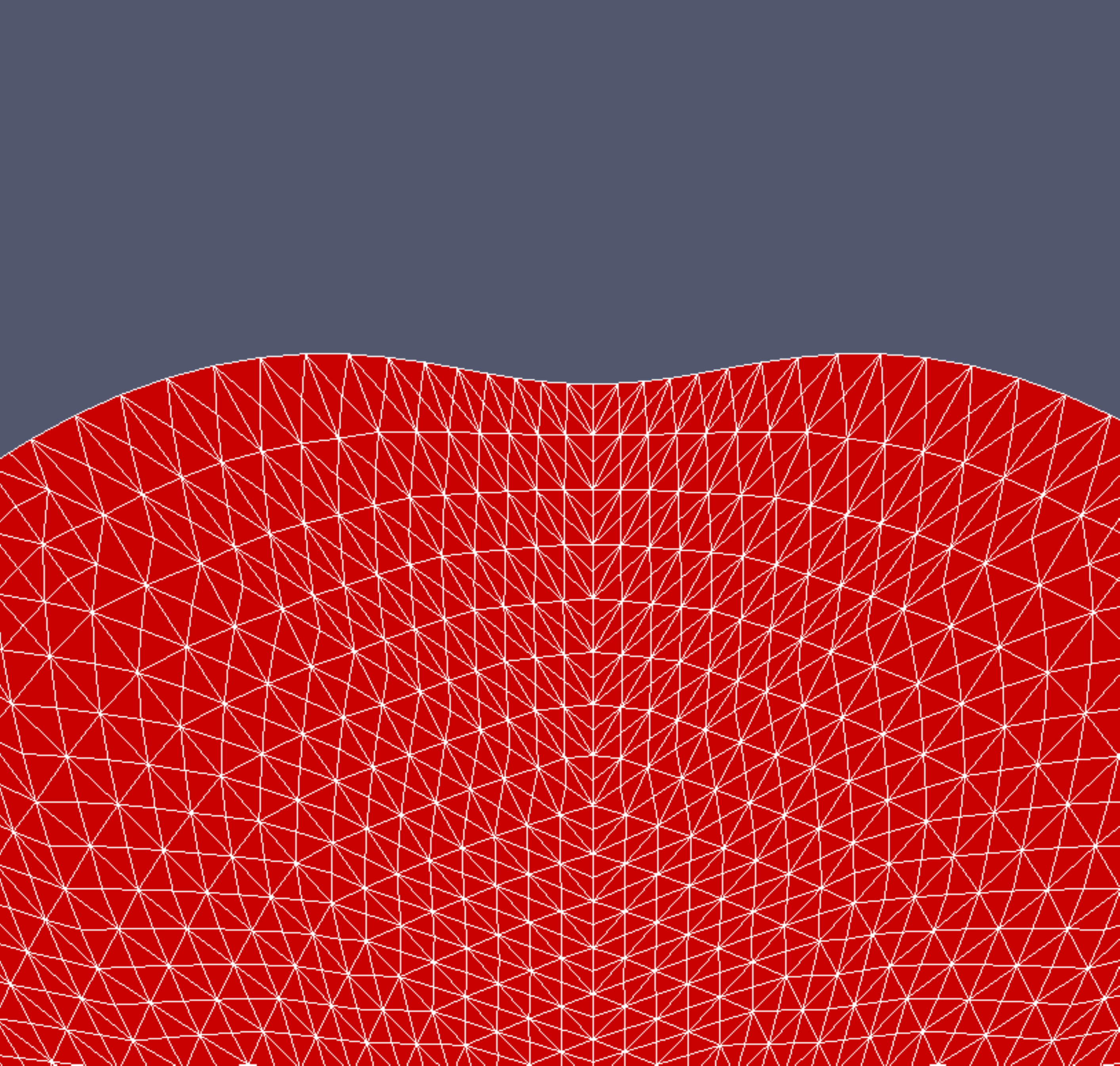}
\label{hele_shaw_with_DeT_at_time_4_00}
} 
~
\subfloat[][\centering Computational mesh at $t=5.0$ with DeTurck redistribution.]
{\includegraphics[width=0.2\textwidth]{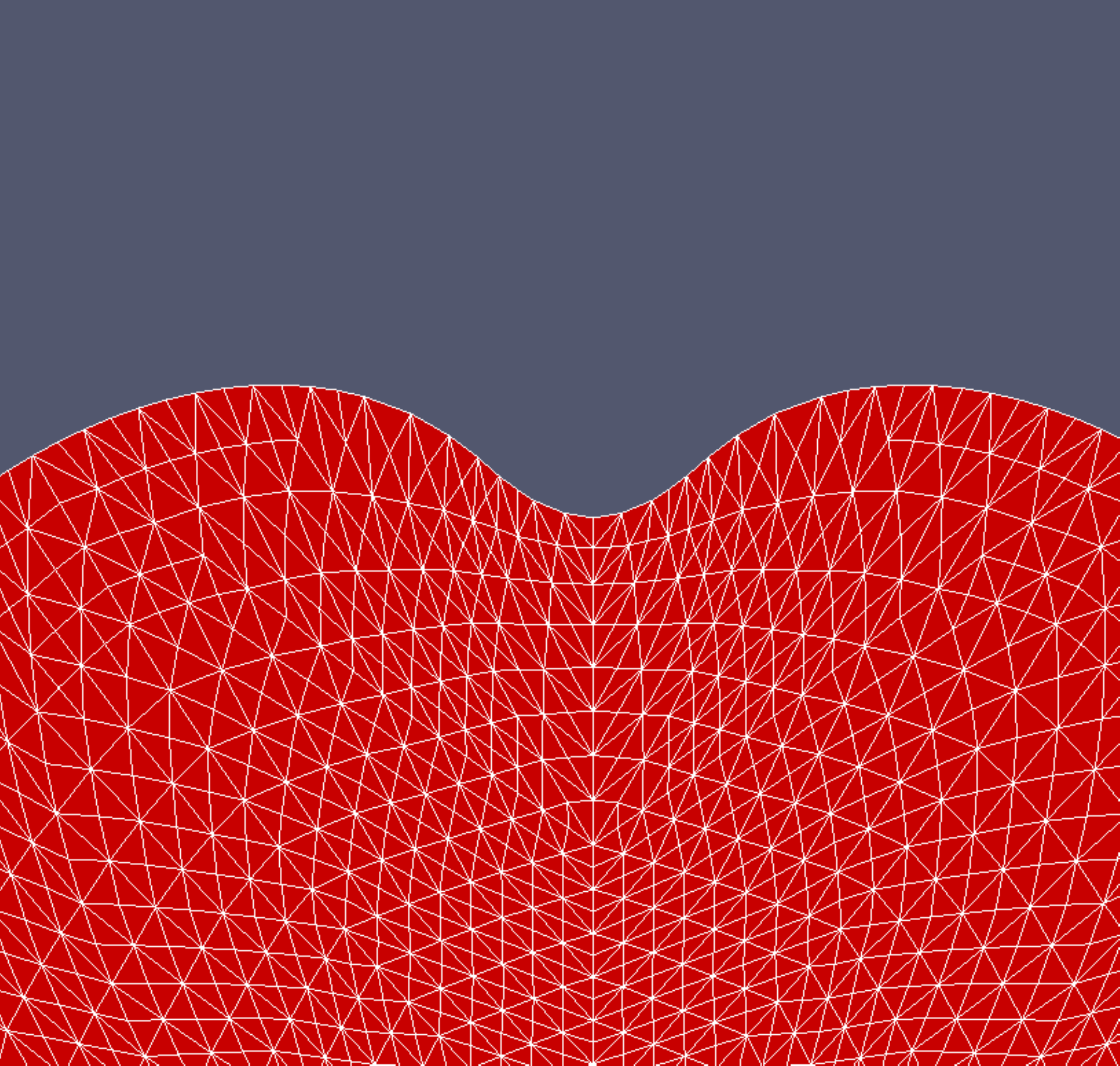}
\label{hele_shaw_with_DeT_at_time_5_00}
} 
~
\subfloat[][\centering Computational mesh at $t=5.2$ with DeTurck redistribution.]
{\includegraphics[width=0.2\textwidth]{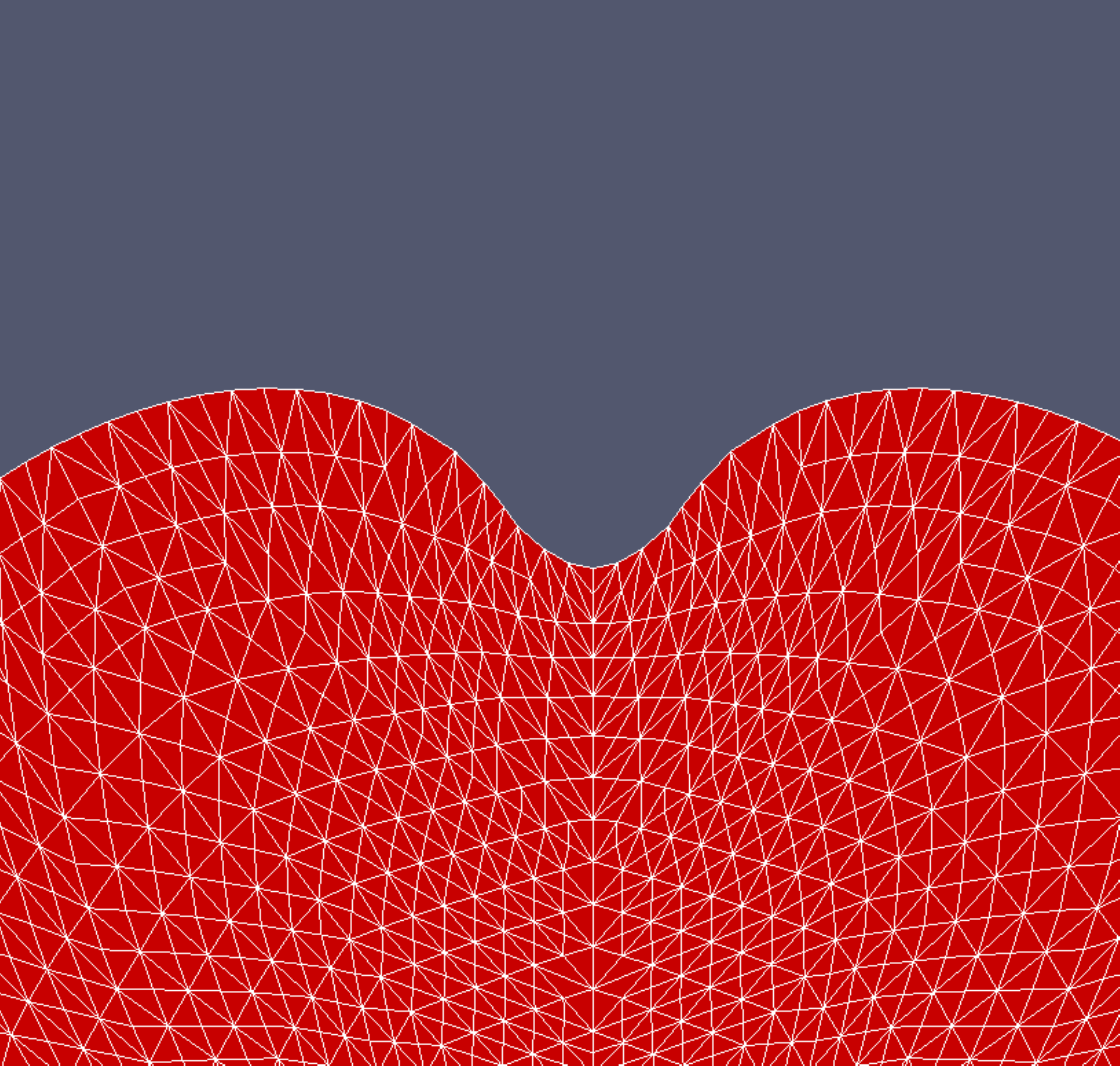}
\label{hele_shaw_with_DeT_at_time_5_20}
} 
~
\subfloat[][\centering Computational mesh at $t=5.34$ with DeTurck redistribution.]
{\includegraphics[width=0.2\textwidth]{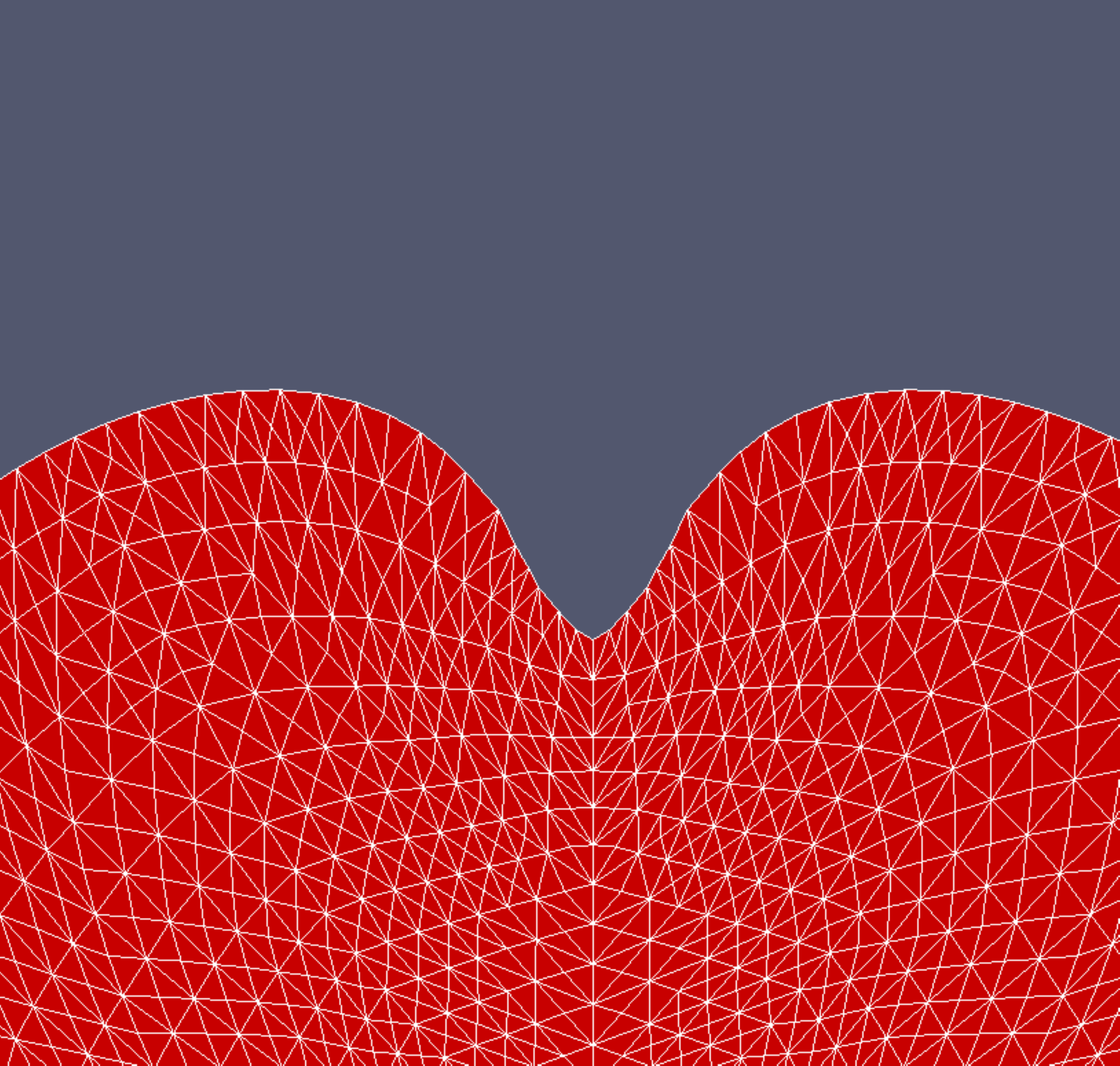}
\label{hele_shaw_with_DeT_at_time_5_34}
}
\\
\subfloat[][\centering Computational mesh at $t=4.0$ without redistribution.]
{\includegraphics[width=0.2\textwidth]{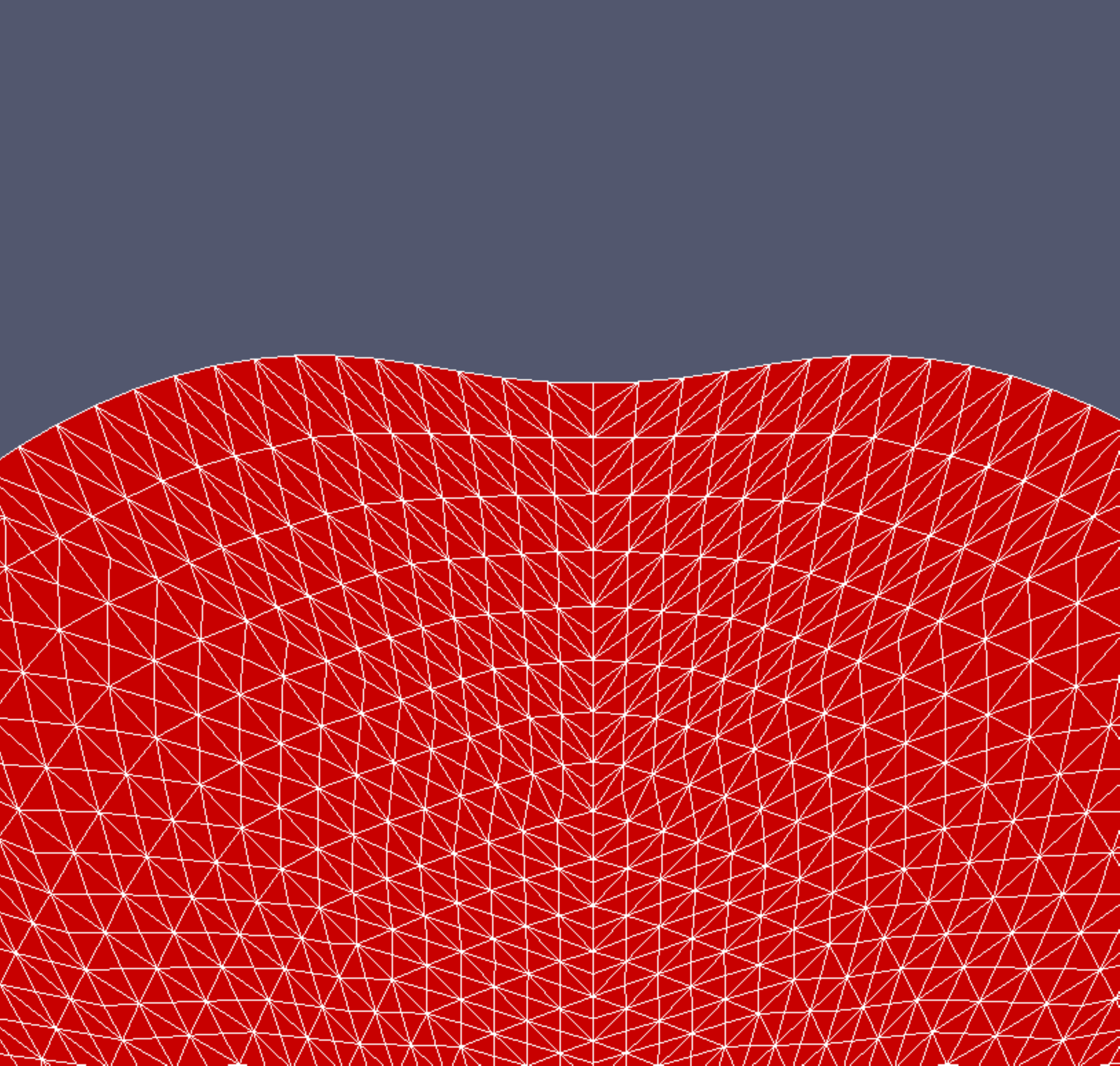}
\label{hele_shaw_without_DeT_at_time_4_00}
} 
~
\subfloat[][\centering Computational mesh at $t=5.0$ without redistribution.]
{\includegraphics[width=0.2\textwidth]{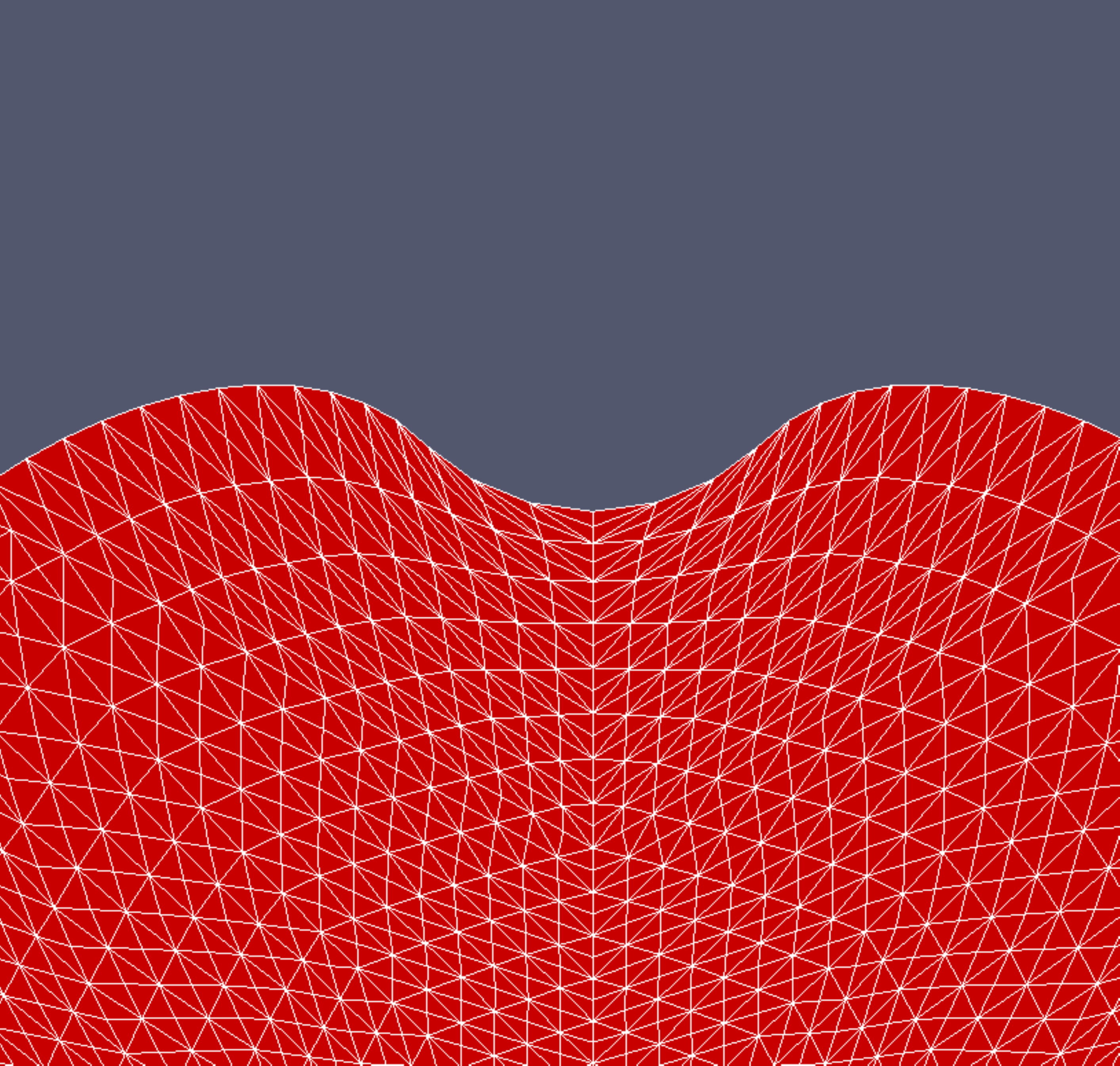}
\label{hele_shaw_without_DeT_at_time_5_00}
} 
~
\subfloat[][\centering Computational mesh at $t=5.2$ without redistribution.]
{\includegraphics[width=0.2\textwidth]{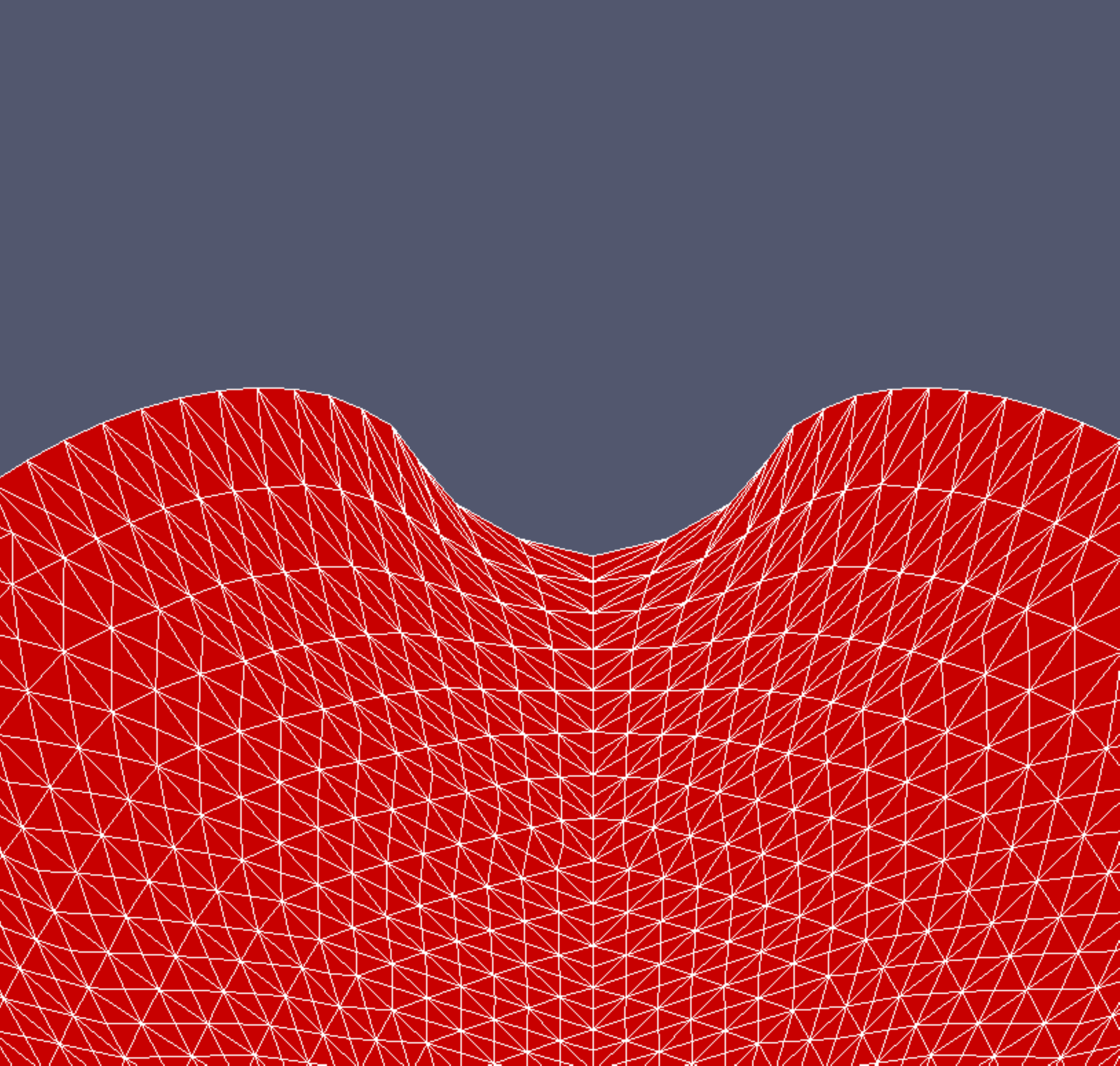}
\label{hele_shaw_without_DeT_at_time_5_20}
} 
~
\subfloat[][\centering Computational mesh at $t=5.34$ without redistribution.]
{\includegraphics[width=0.2\textwidth]{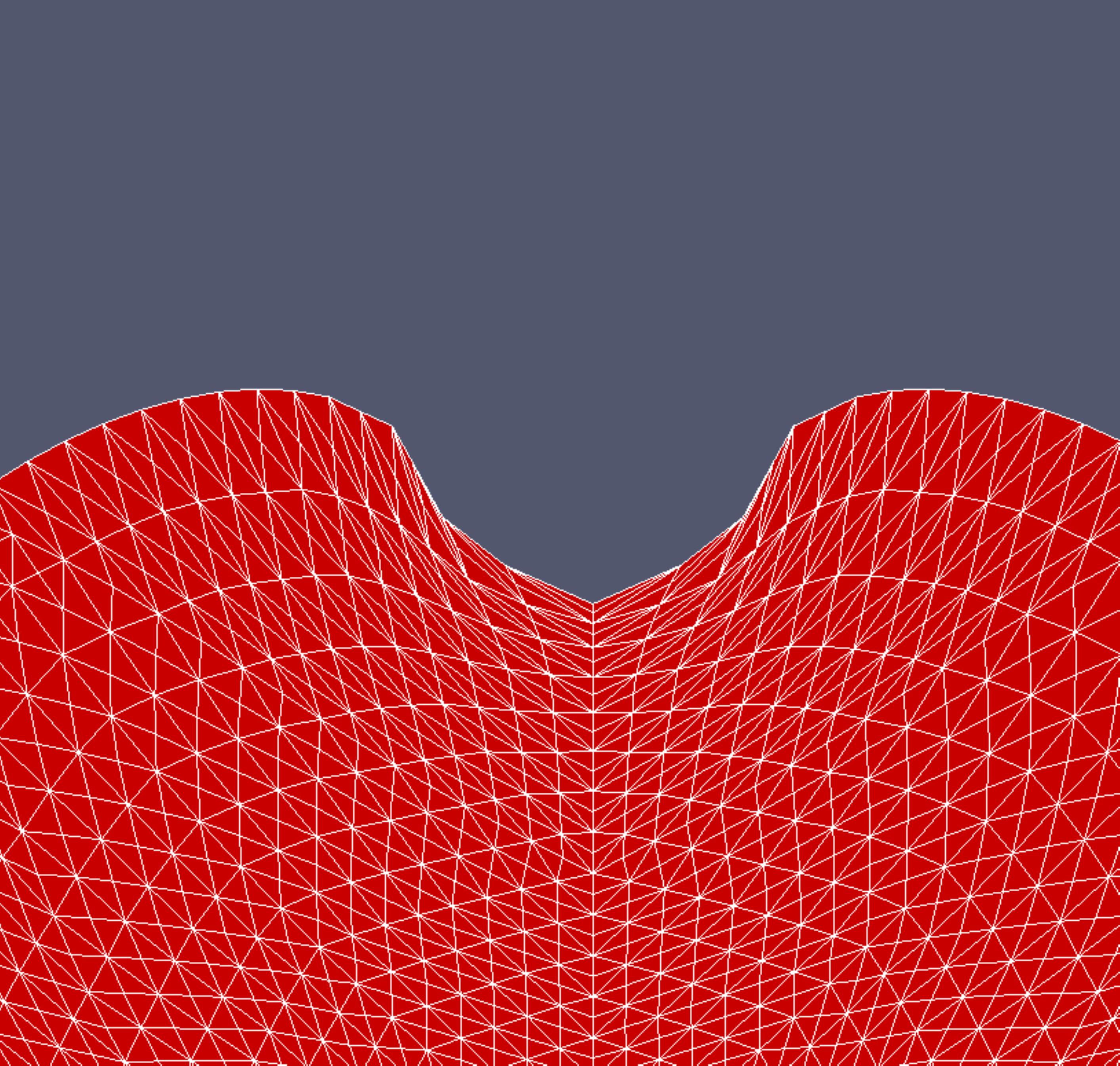}
\label{hele_shaw_without_DeT_at_time_5_34}
} 
\caption{Comparison of the mesh behaviour for a domain evolving according to the Hele-Shaw flow described in
  Example $4$. The initial shape is a unit disk. Figures \ref{hele_shaw_without_DeT_at_time_4_00} and 
  \ref{hele_shaw_without_DeT_at_time_5_34} show the results without redistribution of mesh vertices,
  whereas in Figures \ref{hele_shaw_with_DeT_at_time_4_00} to 
  \ref{hele_shaw_with_DeT_at_time_5_34} the application of Algorithm \ref{algo_DeTurck} is presented.
  In Figure \ref{hele_shaw_without_DeT_at_time_5_34} the mesh is degenerated
  (see also Figure \ref{hele_shaw_mesh_quality}) and its boundary has sharp corners.
  This observation suggests that the numerical result in Figure \ref{hele_shaw_with_DeT_at_time_5_34}
  is a much better approximation of the solution to the Hele-Shaw flow
  at time $t=5.34$. See Example $4$ for more details.
  }
\label{Hele_Shaw_figures} 
\end{center}
\end{figure}

\gdef\gplbacktext{}%
\gdef\gplfronttext{}%
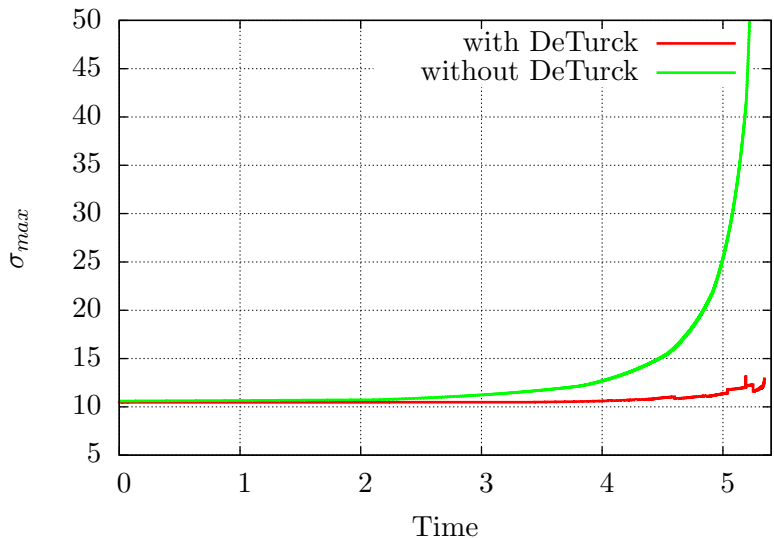
\begin{figure}
\begin{center}
  \begin{picture}(5668.00,3400.00)%
    \gplgaddtomacro\gplbacktext{%
      \csname LTb\endcsname%
      \put(660,110){\makebox(0,0)[r]{\strut{} 5}}%
      \csname LTb\endcsname%
      \put(660,473){\makebox(0,0)[r]{\strut{} 10}}%
      \csname LTb\endcsname%
      \put(660,836){\makebox(0,0)[r]{\strut{} 15}}%
      \csname LTb\endcsname%
      \put(660,1199){\makebox(0,0)[r]{\strut{} 20}}%
      \csname LTb\endcsname%
      \put(660,1562){\makebox(0,0)[r]{\strut{} 25}}%
      \csname LTb\endcsname%
      \put(660,1926){\makebox(0,0)[r]{\strut{} 30}}%
      \csname LTb\endcsname%
      \put(660,2289){\makebox(0,0)[r]{\strut{} 35}}%
      \csname LTb\endcsname%
      \put(660,2652){\makebox(0,0)[r]{\strut{} 40}}%
      \csname LTb\endcsname%
      \put(660,3015){\makebox(0,0)[r]{\strut{} 45}}%
      \csname LTb\endcsname%
      \put(660,3378){\makebox(0,0)[r]{\strut{} 50}}%
      \csname LTb\endcsname%
      \put(792,-110){\makebox(0,0){\strut{} 0}}%
      \csname LTb\endcsname%
      \put(1693,-110){\makebox(0,0){\strut{} 1}}%
      \csname LTb\endcsname%
      \put(2593,-110){\makebox(0,0){\strut{} 2}}%
      \csname LTb\endcsname%
      \put(3494,-110){\makebox(0,0){\strut{} 3}}%
      \csname LTb\endcsname%
      \put(4394,-110){\makebox(0,0){\strut{} 4}}%
      \csname LTb\endcsname%
      \put(5295,-110){\makebox(0,0){\strut{} 5}}%
      \put(22,1744){\rotatebox{90}{\makebox(0,0){\strut{}$\sigma_{max}$}}}%
      \put(3223,-440){\makebox(0,0){\strut{}Time}}%
    }%
    \gplgaddtomacro\gplfronttext{%
      \csname LTb\endcsname%
      \put(4668,3205){\makebox(0,0)[r]{\strut{}with DeTurck}}%
      \csname LTb\endcsname%
      \put(4668,2985){\makebox(0,0)[r]{\strut{}without DeTurck}}%
    }%
    \gplbacktext
    \put(0,0){\includegraphics{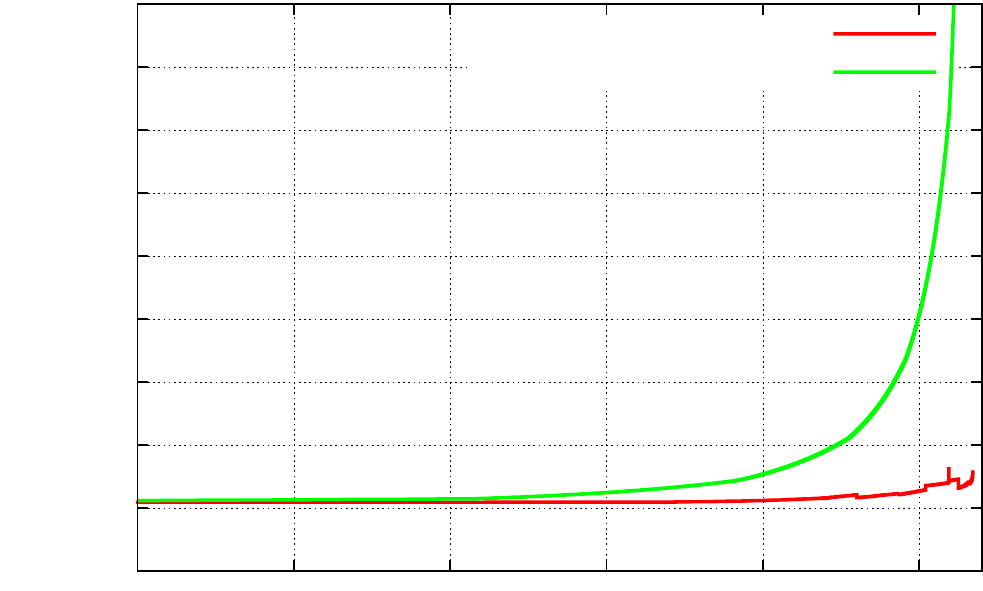}}%
    \gplfronttext
  \end{picture}%
  \vspace*{10mm}
  \caption{Mesh quality $\sigma_{max}$, see (\ref{definition_sigma_max}), for the mesh in Figure 
  \ref{Hele_Shaw_figures}. Without redistribution of vertices, the mesh degenerates when the domain 
  starts to form a cusp. This suggests that the numerical result in Figures \ref{hele_shaw_with_DeT_at_time_4_00}
  to \ref{hele_shaw_with_DeT_at_time_5_34} is a far better approximation to the solution of the Hele-Shaw flow
  than the result in Figures \ref{hele_shaw_without_DeT_at_time_4_00} to \ref{hele_shaw_without_DeT_at_time_5_34}.
  See Example $4$ for more details.  
  }
  \label{hele_shaw_mesh_quality}
\end{center}
\end{figure}

\subsection*{Example 5: The ALE ESFEM and the DeTurck trick for solving PDEs on evolving surfaces}
In the last example, we present how the method proposed in this paper can be used in the ALE ESFEM
introduced in \cite{ES12, EllVen15}. We 
consider the advection-diffusion equation
\begin{align}
	& \partial^\circ p + p \nabla_{\Gamma(t)} \cdot v  
	- \nabla_{\Gamma(t)} \cdot (D \nabla_{\Gamma(t)} p)  = f,
	\quad \textnormal{in $\Gamma(t)$,}
	\label{reaction_diffusion_equation}
\\	
   	& \nu(t) \cdot \nabla_{\Gamma(t)} p = 0, 
   	\quad \textnormal{on $\Gamma(t)$.}	
   	\label{BC_reaction_diffusion_equation}
\end{align}
Here, $D > 0$ denotes a constant scalar diffusivity. $v: \Gamma(t) \rightarrow \mathbb{R}^n$
is the velocity field of the medium in which the diffusion process takes place.
The medium is supposed to be contained in $\Gamma(t)$. More precisely, we assume that $\Gamma(t)$ also moves with velocity $v$.
We here choose $\Gamma(0) = B_{r_1}(0) \setminus B_{r_2}(0) \subset \mathbb{R}^2$ with $r_1 = 2.25$ and $r_2 = 0.25$,
and furthermore,
$$
	v(x_1,x_2) = (-7 \sin(2\pi t)(1 - \tfrac{16}{81}(x_1^2 + x_2^2)),
				   7 \cos(2\pi t)(1 - \tfrac{16}{81}(x_1^2 + x_2^2)))^T
$$ 
for all $(x_1, x_2) \in \Gamma(t)$. 
The material derivative $\partial^\circ p$ is defined by 
$$
	(\partial^\circ p) \circ x := \frac{d}{dt} (p \circ x), 
$$
where $x$ is supposed to be the embedding, whose time derivative is described by $v$; see (\ref{equation_of_motion_non-reparametrized}).
If $p$ is differentiable in an open neighbourhood of $\Gamma(t) \subset \mathbb{R}^n$, 
we obviously obtain
$$
	\partial^\circ p = p_t + v \cdot \nabla p.
$$
Similarly, the derivative $\partial^\bullet$ defined in (\ref{material_derivative_definition}) satisfies 
$$
	\partial^\bullet p =  p_t + \hat{v} \cdot \nabla p.
$$
Hence, we have
\begin{align}
	\partial^\circ p - \partial^\bullet p  = (v - \hat{v}) \cdot \nabla p
	= (v - \hat{v}) \cdot \nabla_{\Gamma(t)} p.
	\label{difference_material_derivatives}
\end{align}
Multiplying (\ref{reaction_diffusion_equation}) by a test function $\varphi(t) \in H^{1,2}(\Gamma(t))$, 
integrating and applying the transport formula, see Theorem 5.1 in \cite{DE13}, yields
\begin{align*}
	\frac{d}{dt} \int_{\Gamma(t)} p \varphi \; do 
	+ D \int_{\Gamma(t)} \nabla_{\Gamma(t)} p \cdot \nabla_{\Gamma(t)} \varphi \; do
	= \int_{\Gamma(t)} p \partial^\circ \varphi  + f \varphi \; do,
	\quad \forall \varphi \in H^{1,2}(\Gamma(t)). 
\end{align*}
Assuming that we only consider test functions $\varphi(t)$ with $\partial^\bullet \varphi =0$
and using (\ref{difference_material_derivatives}), this leads to
\begin{align*}
	\frac{d}{dt} \int_{\Gamma(t)} p \varphi \; do 
	+ D \int_{\Gamma(t)} \nabla_{\Gamma(t)} p \cdot \nabla_{\Gamma(t)} \varphi \; do
	= \int_{\Gamma(t)}  p (v - \hat{v}) \cdot \nabla_{\Gamma(t)} \varphi + f \varphi \; do,
	\quad \forall \varphi \in H^{1,2}(\Gamma(t)). 
\end{align*}
Motivated by the work in \cite{ES12}, we thus define $p_h^{m+1} \in V_h(\Gamma_h^{m+1})$ to be the solution of
\begin{align}
    & \int_{\Gamma^{m+1}_h} \tfrac{1}{\tau} I_h (p_h^{m+1} \varphi^{m+1}_h) 
    + D \nabla_{\Gamma^{m+1}_h} p^{m+1}_h \cdot \nabla_{\Gamma^{m+1}_h} \varphi_h^{m+1} 
    + I_h(p_h^{m+1} v_{DeT,h}^{m+1, \beta}) (\nabla_{\Gamma^{m+1}_h} \varphi^{m+1}_h)_{\beta} \; do 
    \nonumber
   \\ 
    &=  \int_{\Gamma^m_h} \tfrac{1}{\tau} I_h (p_h^m \varphi_h^m) \; do 
         + \int_{\Gamma^{m+1}_h}  I_h(f^{m+1} \varphi_h^{m+1}) \; do,
\label{ALE_scheme}    
\end{align}
for all $\varphi_h^{m+1} \in V_h(\Gamma^{m+1}_h)$,
where $\Gamma^{m+1}_h$ is computed according to Algorithm \ref{algo_DeTurck}
and $\varphi^{m+1}_h \in V_h(\Gamma^{m+1}_h)$ is such that it has the same coefficients
with respect to the Lagrange basis functions of $V_h(\Gamma^{m+1}_h)$ as 
$\varphi^{m}_h \in V_h(\Gamma^{m}_h)$ has with respect to the Lagrange basis functions of 
$V_h(\Gamma^{m}_h)$; see \cite{ES12} for more details. 
The vector field $v^{m+1}_{DeT,h} \in V_h(\Gamma^{m+1}_h)^n$
is defined by $v^{m+1}_{DeT,h} \circ u^{m+1}_h := \tfrac{1}{\tau} (u^{m+1}_h - \tilde{u}_h^m) - I_h v^m$. 
Here we aim to approximate the solution
\begin{align}
p(x,t) =  \cos(2 \pi t) \exp(-|x|^2)
\label{solution_ALE_heat}
\end{align}
of (\ref{reaction_diffusion_equation}) and (\ref{BC_reaction_diffusion_equation}) for the right hand side
\begin{align*}
 f(x,t) = \left( \cos(2\pi t)(2 v_0 \cdot x - 2 v \cdot x + 4D (1 - |x|^2)) 
 	 - 2 \pi \sin(2 \pi t)
 \right) exp(- |x|^2),
\end{align*}
where $v_0 := (\tfrac{112}{81}\sin(2\pi t) , - \tfrac{112}{81} \cos(2\pi t))^T$.
The numerical result of (\ref{ALE_scheme}) for the parameters $\tau = 0.001~ h_{min}^2$, $\alpha = 0.1$, $T_{adapt} = 10^{-3}$ and $D = 2.0$ is presented in Figure \ref{ALE_figures}.
\begin{figure}
\begin{center}
\subfloat[][\centering $p_h$ at $t=0.0$.]
{\includegraphics[width=0.266\textwidth]{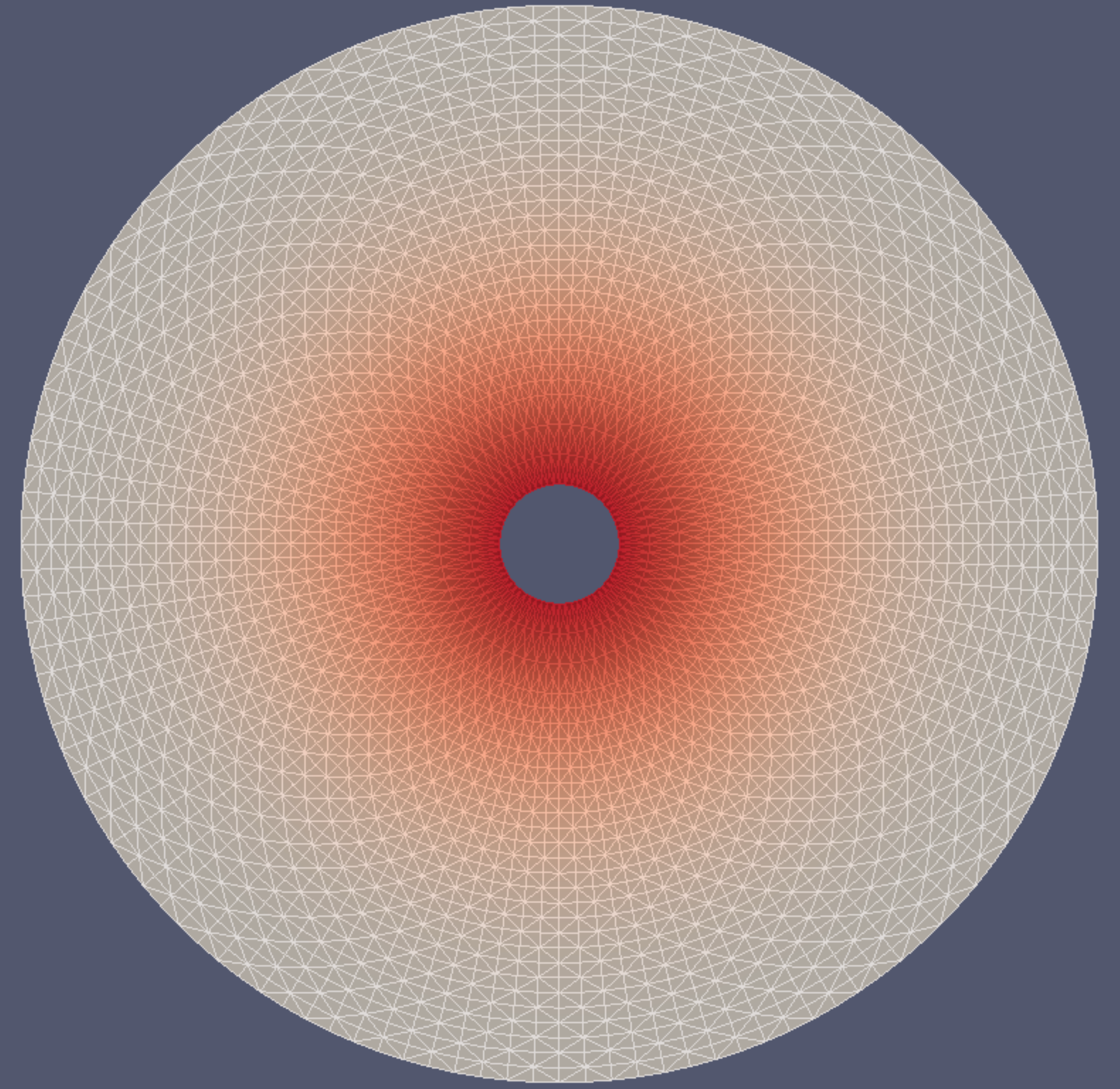}
\label{ALE_heat_at_time_0_00}
} 
~~
\subfloat[][\centering $p_h$ at $t=0.25$.]
{\includegraphics[width=0.266\textwidth]{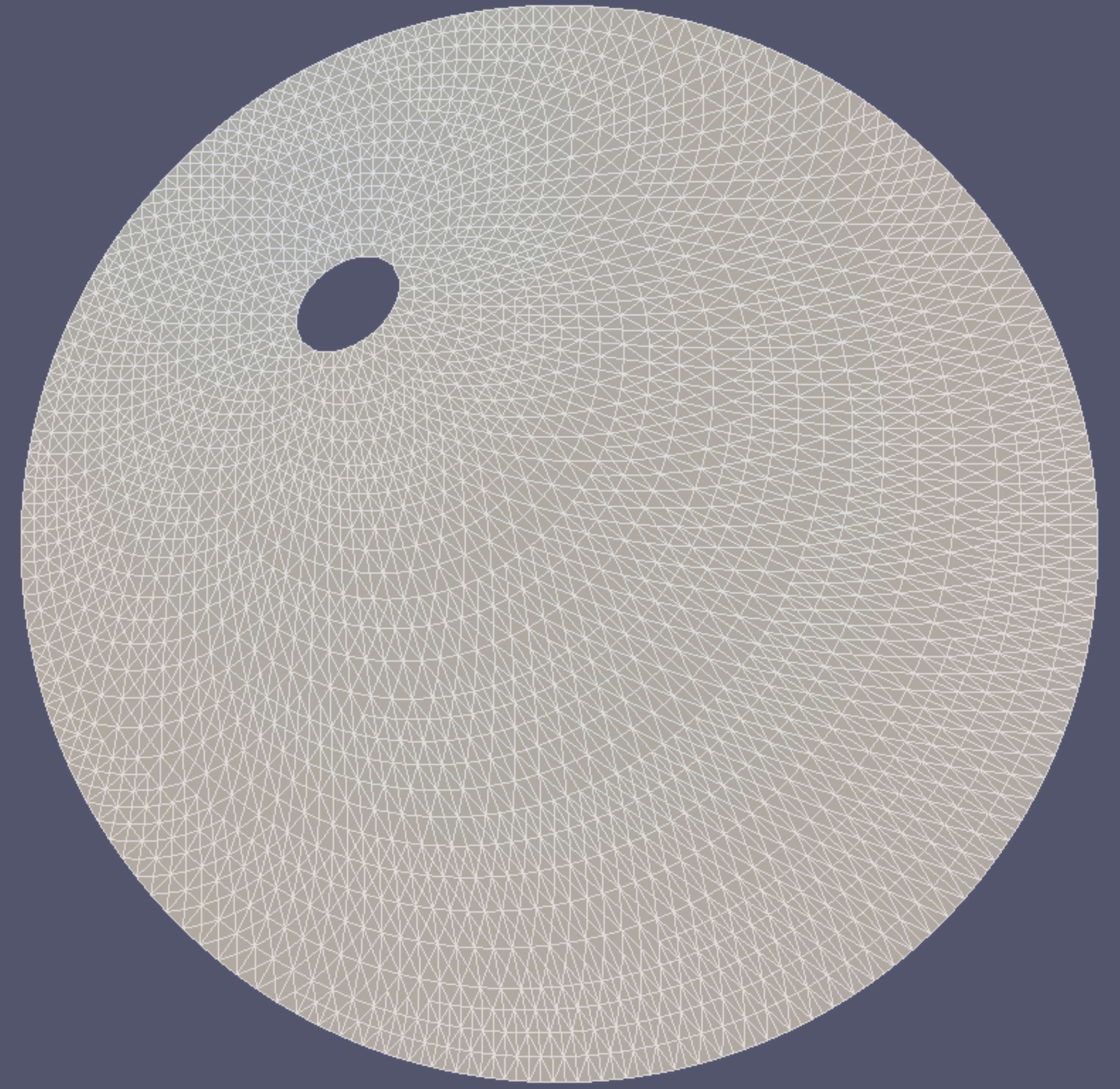}
\label{ALE_heat_at_time_0_25}
} 
~~
\subfloat[][\centering $p_h$ at $t=0.5$.]
{\includegraphics[width=0.266\textwidth]{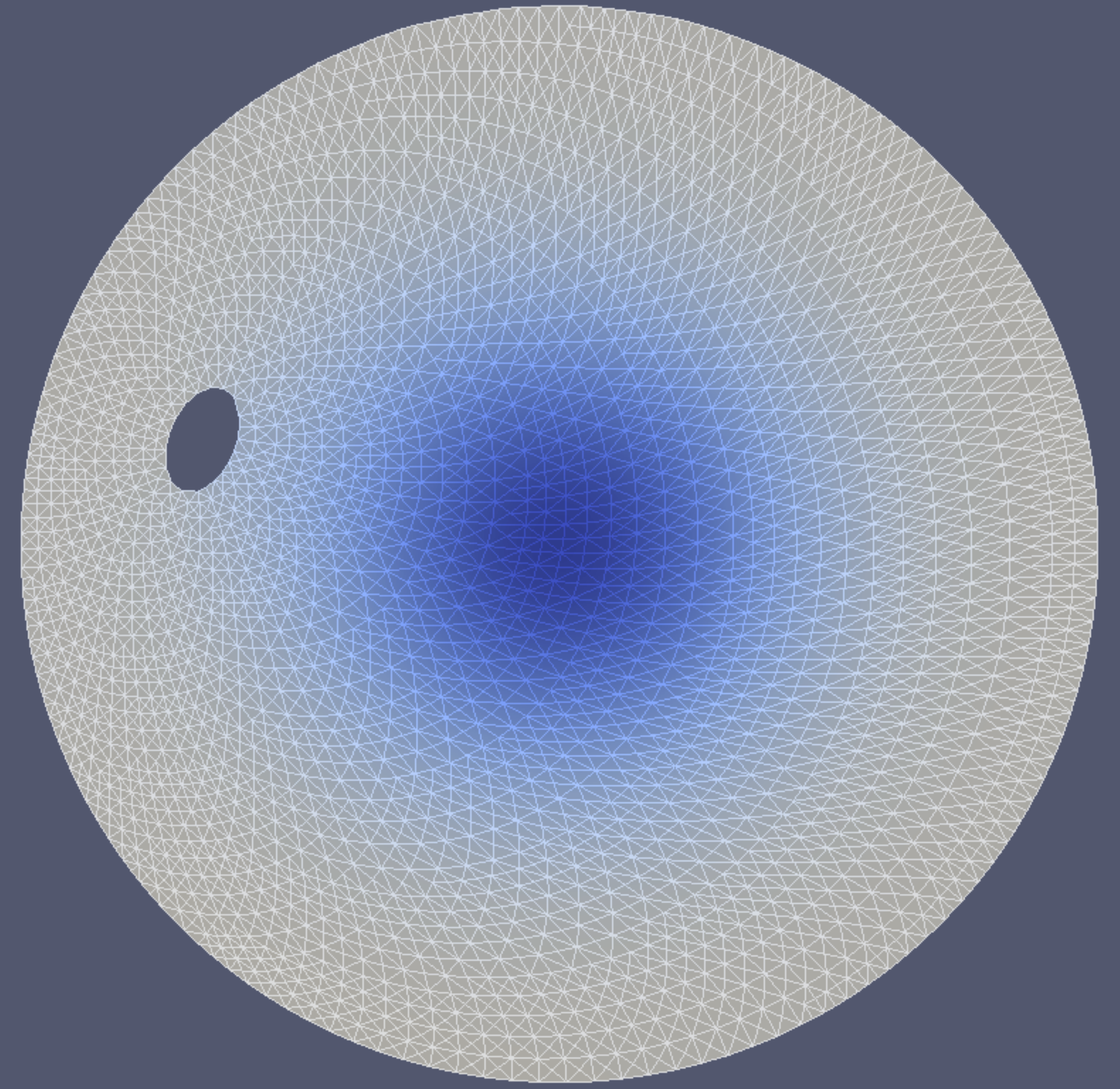}
\label{ALE_heat_at_time_0_50}
} 
\\
\subfloat[][\centering $p_h$ at $t=0.75$.]
{\includegraphics[width=0.266\textwidth]{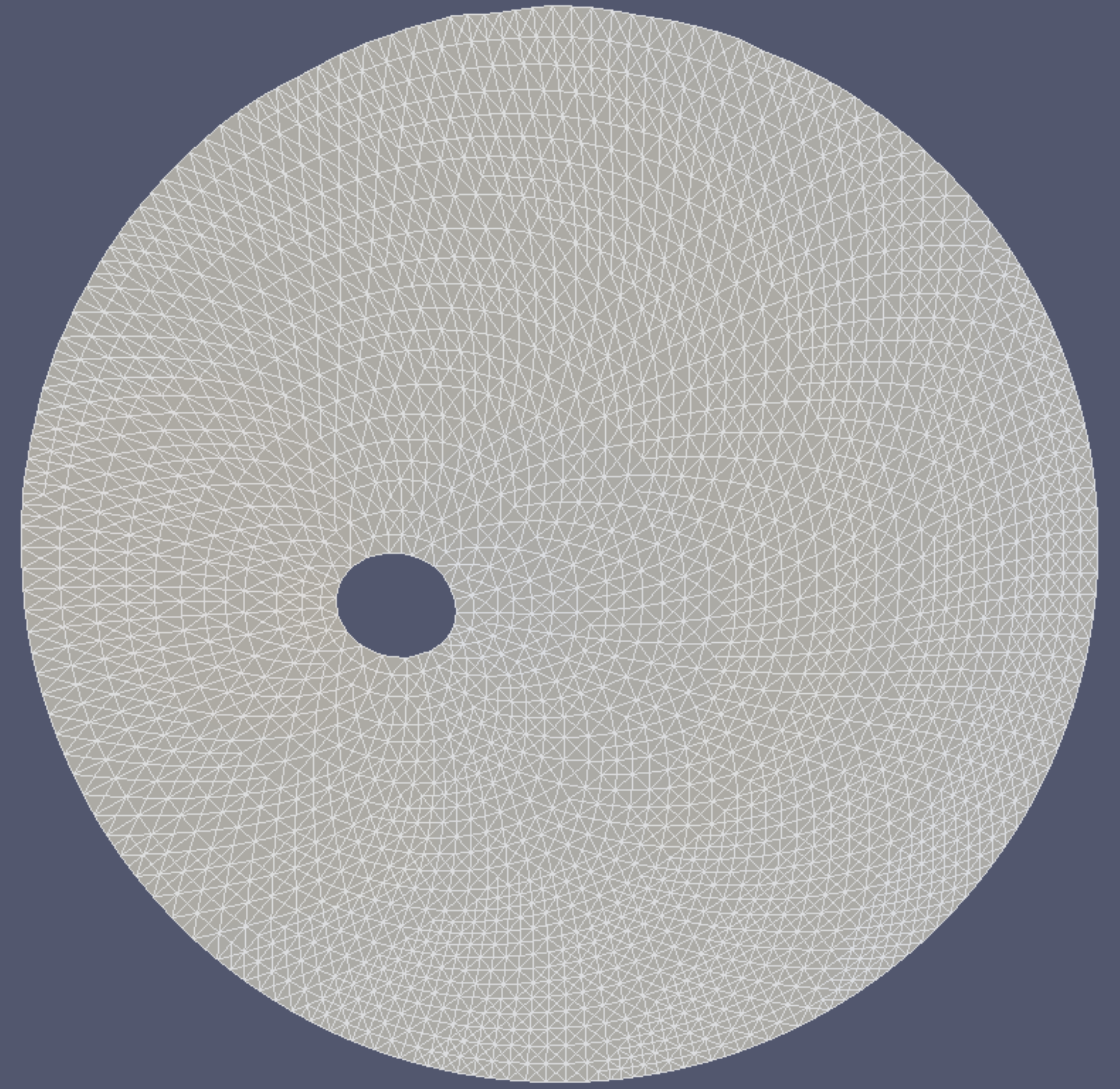}
\label{ALE_heat_at_time_0_75}
} 
~~
\subfloat[][\centering $p_h$ at $t=1.0$.]
{\includegraphics[width=0.266\textwidth]{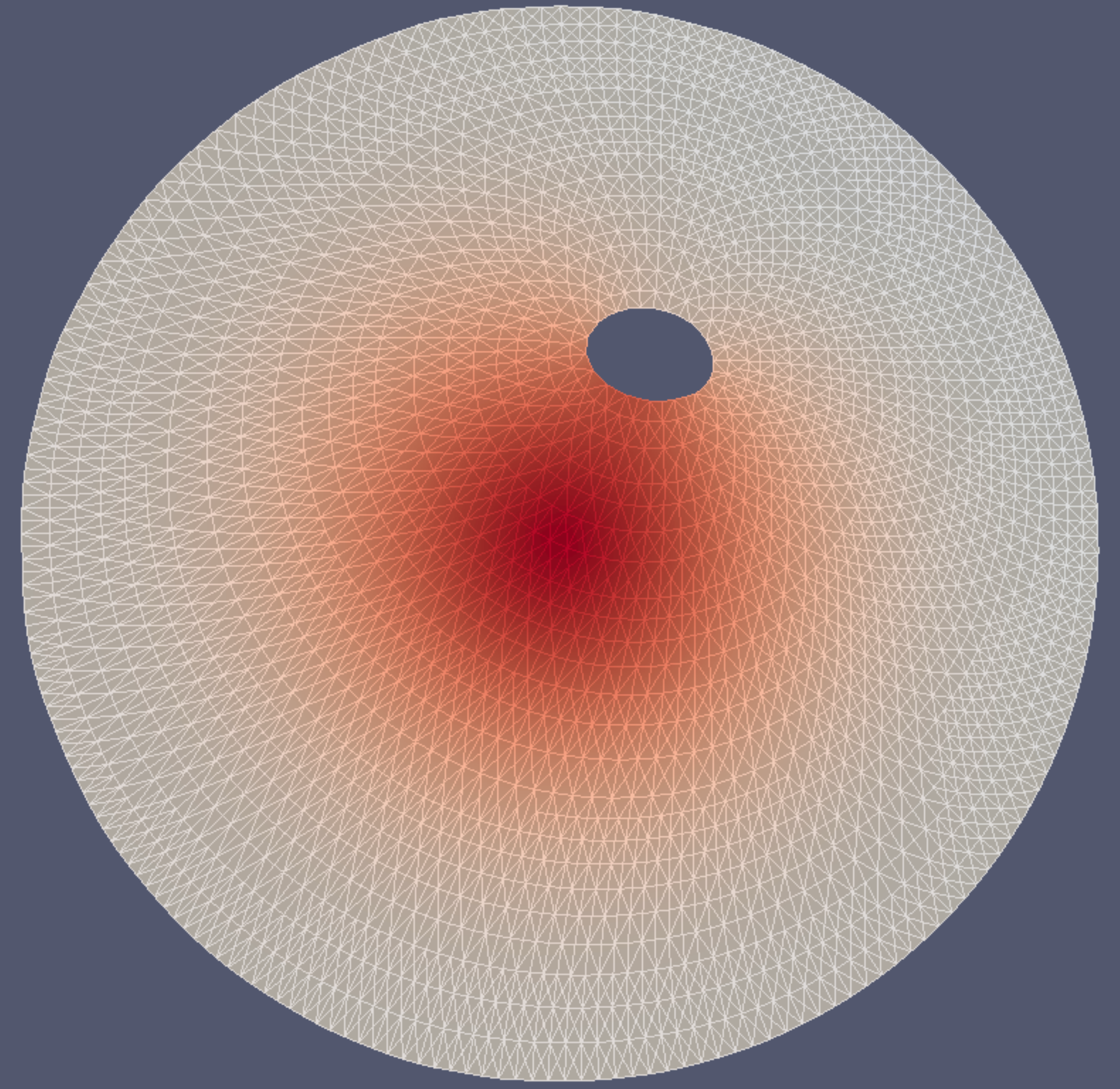}
\label{ALE_heat_at_time_1_00}
} 
\caption{Numerical approximation of the advection-diffusion equation (\ref{reaction_diffusion_equation}) 
on a moving domain $\Gamma(t)$. 
The mesh of the domain is deformed 
according to Algorithms \ref{algo_DeTurck} and \ref{algo_refinement_and_coarsening_strategy}.
The numerical solution $p_h^m$ of (\ref{ALE_scheme}) is indicated by the colour 
scheme red $\approx 1$, grey $\approx 0$ and blue $\approx -1$.
The true solution is given in (\ref{solution_ALE_heat}).
This example demonstrates how Algorithm \ref{algo_DeTurck} in combination with the ALE-method in \cite{ES12}
can be used to solve a PDE on a moving submanifold with boundary.
For more details see Example $5$.
}
\label{ALE_figures}
\end{center}
\end{figure}

\section{Discussion}
\label{section_discussion}
We think that the remeshing approach proposed in this paper has the potential to be very useful for many applications
with moving boundaries. Since it is based on the discretization of a PDE, it would, in principle, be
possible to use standard techniques of Numerical Analysis to estimate discretization errors
connected with this method. This is certainly one of the advantages of our approach.
In our numerical experiments, we observed that our scheme can improve the mesh quality
of a given mesh or preserve the mesh quality for a moving mesh
provided that the reference mesh is of sufficiently high quality.

We also observed that Algorithm \ref{algo_DeTurck} tends to deform 
$\Gamma_h^m$ to a mesh with simplices of different size but similar shape
as those of the reference mesh $\M_h$. 
This behaviour is due to the fact that Algorithm \ref{algo_DeTurck} is based on the DeTurck trick.
Under certain conditions, the harmonic map heat flow converges to a harmonic map as time tends to infinity,
see \cite{Ham75}.
For a stationary submanifold $\Gamma$ this would mean that the mesh $\Gamma_h$ is the image
of the reference mesh $\M_h$ under an approximation to the inverse of a harmonic map.
On the other hand, harmonic maps between compact orientable surfaces are known to be 
conformal maps under certain conditions, see \cite{EW76}.
It is therefore not surprising that in our experiments the triangles of the computational mesh and of the reference mesh often
seem to have similar angles. A side effect of this behaviour is that the area of the simplices of $\Gamma_h$ 
tends to decrease or increase non-homogeneously.
The easiest way to take this into account is to apply a refinement and 
coarsening strategy like in Algorithm \ref{algo_refinement_and_coarsening_strategy}.
Fortunately, as we have seen in Section \ref{section_numerical_results}, the mesh quality
is not critically affected by the mesh refinement or coarsening
-- in particular, because the refinement and coarsening procedure only changes the mesh quality locally.
Since the time step size $\tau$ of the scheme depends critically on the time scale of the remeshing procedure,
that is on $\alpha$,
it would be advantageous to have a strategy for finding the optimal parameter $\alpha$.
Here, optimal means that $\alpha$ should be as large as possible, since this enables larger time steps
$\tau$, and at the same time sufficiently small to ensure a good mesh quality for all times.
This issue remains open for future research.

\subsection*{Acknowledgements}
The second author would like to thank the Alexander von Humboldt Foundation, Germany, for their financial
support by a Feodor Lynen Research Fellowship in collaboration with the University of Warwick,
UK.

\end{document}